\documentclass[11pt]{article}

\usepackage{mathpazo}

\usepackage{microtype}
\microtypesetup{expansion=false}
\renewcommand{\ge}{\geqslant}
\renewcommand{\le}{\leqslant}
\usepackage{amsmath,amssymb,amsthm,mathtools}
\allowdisplaybreaks
\usepackage{booktabs}
\usepackage{array}
\usepackage{longtable}
\usepackage{needspace}
\usepackage{enumitem}
\usepackage{hyperref}
\hypersetup{hypertexnames=false,colorlinks=true,linkcolor=blue,citecolor=blue,filecolor=blue,urlcolor=blue,pdftitle={Certifying non-existence of integer points on curves},pdfauthor={Bogdan Grechuk and Eugene Go}}
\usepackage{xcolor}

\newtheorem{theorem}{Theorem}
\newtheorem{proposition}[theorem]{Proposition}

\newtheorem{corollary}[theorem]{Corollary}
\theoremstyle{definition}
\newtheorem{assumption}{Assumption}

\newtheorem{example}[theorem]{Example}
\newtheorem{problem}[theorem]{Problem}

\newcommand{\Z}{\mathbb Z}
\newcommand{\Q}{\mathbb Q}

\newcommand{\leg}[2]{\left(\frac{#1}{#2}\right)}
\newcommand{\Res}{\operatorname{Res}}
\newcommand{\Norm}{\operatorname{N}}

\newcommand{\edgerow}[2]{\mbox{$\mathbf{#1}{:}$}\nobreak\hspace{0.35em}$#2$\hspace{1.15em plus .3em minus .15em}\allowbreak}

\title{\vspace{-1.5cm}{\bfseries Certifying non-existence of integer points on curves}}
\author{Bogdan Grechuk and Eugene Go}
\date{12 October 2026}

\begin{document}
\maketitle

\begin{abstract}
	We present an elementary but powerful method for proving that a two-variable equation has no integer solutions. The method produces quadratic identities which every integer solution must satisfy. Once these identities are established, a standard quadratic reciprocity argument finishes the proof. We illustrate the power of the method by resolving $25$ out of $30$ two-variable equations of size up to $34$ for which the problem of integer solubility remained open. We also resolve the remaining $5$ equations by deeper methods, and further determine the integer solubility of all two-variable equations of size up to $35$ and of length up to $12$, with two exceptions. Before this work, there were $211$ equations in this category that resisted the previously known methods, and we proved that $209$ of them have no integer solutions, with all proofs formalized in Lean. In addition, we show that our method can be used to solve certain two-variable equations whose integer solution set is non-empty. 
\end{abstract}

\section{Introduction}\label{sec:intro}

In 2018, MathOverflow user Zidane asked\footnote{\url{https://mathoverflow.net/questions/316708}} what the smallest unsolved Diophantine equation is. To make the word ``smallest'' rigorous, Zidane defined the \emph{size} of a polynomial Diophantine equation 
\begin{equation}\label{eq:diofgen}
	P(x_1, \dots, x_n) = 0
\end{equation} 
where the polynomial $P$ consists of $k$ non-similar monomials with integer coefficients $a_1, \dots, a_k$ and degrees $d_1, \dots, d_k$, respectively, as
\begin{equation*}
	h(P)=\sum_{i=1}^k |a_i| 2^{d_i}.
\end{equation*}
To define the word ``unsolved'' rigorously, we may consider Hilbert's tenth problem.

\begin{problem}\label{prob:main}
	Given a Diophantine equation, determine whether it has an integer solution.
\end{problem}

Then ``the smallest unsolved equation'' is the equation of the form \eqref{eq:diofgen} with the smallest size $h(P)$ for which Problem \ref{prob:main} remains open. 

In 2021, the first author started the project of solving Problem \ref{prob:main} systematically for all equations \eqref{eq:diofgen} ordered by size $h$. This started as a recreational project, but it quickly went from elementary equations to much more interesting ones. All equations of size $h\leq 12$ either have easy-to-find integer solutions, or have no integer solutions for the obvious reason that they either have no real solutions, or no solutions modulo some integer $m$. Equation
$$
	x^2y + 2y + 1 = 0
$$ 
of size $h=13$ is the smallest equation solvable in the reals and modulo every integer, but not in integers. Such equations are called counterexamples to the integral Hasse principle. This equation, while solvable in the reals, has all its real solutions in the interval $-1<y<0$, which obviously implies integer insolubility. The smallest counterexample to the integral Hasse principle which has real solutions with $\min\{|x|,|y|\}$ unbounded is the equation
$$
	y^2=x^3-3
$$
of size $h=15$. To see that it has no integer solutions, we may either use the theory of elliptic curves, or rewrite it as
$$
	(x+1)(x^2-x+1)=y^2+4,
$$ 
use the law of quadratic reciprocity to show that the only positive divisors of the right-hand side are $0,1,2,4$ or $5$ modulo $8$, but then observe that the equation has no solutions modulo $8$ for which both factors of the left-hand side satisfy these conditions. The described methods allow us to proceed up to size $h=22$, where the equation
$$
	x^2 + y^2 - z^2 = xyz - 2
$$
requires a different method: Vieta jumping. For larger equations, deeper and deeper methods were required to prove their insolubility. By August 2021, the smallest equations for which Problem \ref{prob:main} remained open were equations
$$
	y^2-xyz+z^2=x^3-5
$$ 
and 
$$
	x(x^2+y^2+1)=z^3-z+1
$$
of size $h=29$. The first equation has been popularized under the name ``fruit Diophantine equation''. 
It was solved by Majumdar and Sury \cite{MR4610191} and has since led to several related works \cite{vaishya2022class, prakash2024generalized, sahoo2024generalized}. The second equation resisted all attempts to prove insolubility, and then Andrew R. Booker found a solution
$$
	(x,y,z)=(4280795, 4360815, 5427173).
$$
Then the frontier moved to $h=31$ and the equation
$$
	y^2-x^3y+z^3+3=0.
$$
In February 2024, Denis Shatrov used cubic reciprocity to prove that this equation has no integer solutions\footnote{\url{https://mathoverflow.net/questions/400714/}}. After this, four equations
$$
y^3-y=x^4+2x-2
$$
$$
y^3+y=x^4+x+4
$$
$$
y^3+xy=x^4+4
$$
$$
y^3+xy=x^4+x+2
$$
of size $h=32$ became the smallest open ones. And here, the project seemed to hit a wall. 

While the listed four equations became the smallest open equations only in 2024, they were known as the smallest open \emph{two-variable} equations from the start of the project in 2021; hence they had been open for over 5 years. In fact, the full list of $30$ two-variable equations of size $h\leq 34$, see Table \ref{tab:2varH34}, was discovered in 2021 and published in 2024 \cite[Section 7.3.1]{grechuk2024polynomial}, but every single equation from this list remained open until this work. 

\begin{table}
	\begin{center}
		\setlength{\tabcolsep}{4pt}
		\begin{tabular}{ |c|c|c|c|c|c| } 
			\hline
			$h$ & Equation & $h$ & Equation & $h$ & Equation \\ 
			\hline\hline
			$32$ & $y^3-y=x^4+2x-2$  & $34$ & $y^3-y^2=x^4+2x+2$ & $34$ & $y^3+xy-y=x^4-4$ \\ 
			\hline
			$32$ & $y^3+y=x^4+x+4$  & $34$ & $y^3+y^2=x^4+2x-2$  & $34$ & $y^3+xy=x^4+x+4$ \\ 
			\hline
			$32$ & $y^3+xy=x^4+4$ & $34$ & $y^3-y^2+xy=x^4+2$ & $34$ & $y^3+xy=x^4-x-4$ \\ 
			\hline
			$32$ & $y^3+xy=x^4+x+2$ & $34$ & $y^3+y=x^4+x+6$ & $34$ & $y^3+xy=x^4+x^2+2$ \\ 
			\hline
			$33$ & $y^3+y^2+xy=x^4+1$ & $34$ & $y^3+y=x^4+x-6$ & $34$ & $y^3-y^2+y=x^4+x+2$ \\ 
			\hline	
			$33$ & $y^3+xy=x^4+x+3$ & $34$ & $y^3+2y=x^4+x+4$ & $34$ & $y^3+y^2+y=x^4+x+2$ \\ 
			\hline	
			$33$ & $y^3+xy=x^4-x-3$ & $34$ & $y^3+y=x^4+2x+4$ & $34$ & $y^3+y^2-y=x^4+x+2$ \\ 
			\hline
			$33$ & $y^3+xy=x^4+x-3$ & $34$ & $y^3+y=x^4+2x-4$ & $34$ & $y^3+xy-y=x^4+x+2$ \\ 
			\hline
			$33$ & $y^3+xy+y=x^4-x+1$ & $34$ & $y^3+xy+y=x^4+4$ & $34$ & $y^3-y=x^4-x^2+x+2$ \\ 
			\hline
			$34$ & $y^3+y^2=x^4+x+4$ & $34$ & $y^3+xy-y=x^4+4$ & $34$ & $y^3-y=x^4-x^2+x-2$ \\ 
			\hline					
		\end{tabular}
		\caption{\label{tab:2varH34} Two-variable equations of size $h\leq 34$ for which Problem \ref{prob:main} was open.}
	\end{center} 
\end{table}

One difficulty is that, for equations in three variables listed in this introduction, a heuristic argument suggests that the number of solutions with $\max\{|x|,|y|,|z|\}\leq B$ ``should'' grow to infinity with $B$. Hence, either we will find a solution if we search long enough, or there should be an ``obstruction'' to the existence of solutions that prevents the heuristic from working. By discovering this obstruction, we may hope to prove that the equation is unsolvable in integers.

For the equations in Table \ref{tab:2varH34}, however, the heuristic suggests that the number of solutions should be finite. In fact, a classical 1929 theorem of Siegel \cite{siegel1929uber} states that every two-variable equation describing a curve of positive genus must have finitely many integer solutions. The proof, however, is not effective, and does not provide an algorithm for finding these solutions. In 1970, Baker and Coates \cite{baker1970integer} proved an effective version of Siegel's theorem for genus 1 curves. For genus $2$ and higher, the fundamental theorem of Faltings \cite{faltings1983endlichkeitssatze}, which confirms an earlier conjecture of Mordell, states that even the number of rational solutions is finite. With an effective version of this result, one could find all rational solutions, and check whether any of them are integers, thus resolving Problem \ref{prob:main}. However, an effective proof of Mordell's conjecture remains elusive. For genus $2$ curves, the general methods for computing all rational points are not available, but there are methods that work well for many families of equations \cite{stoll2026determining}. However, the equations in Table \ref{tab:2varH34} have genus $3$. They also do not belong to any of the families of two-variable equations for which the effective algorithms are known, such as equations quadratic in one of the variables \cite[Proposition 3.65]{grechuk2024polynomial}, super-elliptic equations \cite[Proposition 3.67]{grechuk2024polynomial}, or equations satisfying Runge's condition \cite[Section 3.4.7]{grechuk2024polynomial}. Hence, we had every reason to believe that the equations in Table \ref{tab:2varH34} are out of reach, until the effective Mordell conjecture is proved. And, even if it is proved, the resulting algorithm may still be impractical.

Given that the project of solving Problem \ref{prob:main} in the order of size $h$ has stalled at equations from Table \ref{tab:2varH34}, the focus recently moved to adjacent projects, such as considering only cubic equations \cite{grechuk2026cubic}, arranging the equations in a different order \cite{grechuk2026shortest}, or studying alternative problems such as deciding whether the solution set is finite or infinite \cite{grechuk2026polynomial}. 

This work develops a new method for certifying non-existence of integer points on curves. The method produces, for each equation $F(x,y)=0$, quadratic identities of the form 
$$
H_i(x,y) \cdot T_i(x,y) = U_i(x)^2 + d_{i} \cdot V_i(x)^2, \quad i=1,\dots, k,
$$ 
that must be satisfied by every solution. Once these identities are produced, the standard quadratic reciprocity argument can be used to restrict possible prime factors of $H_i(x,y)$, which may lead to a contradiction. This method, while simple, is surprisingly powerful, and we use it to resolve $25$ of the $30$ equations in Table~\ref{tab:2varH34}. We then resolve the remaining $5$ equations by deeper methods, thus proving that all equations in Table~\ref{tab:2varH34} have no integer solutions. We also extend the analysis to equations of size $35$, equations ordered by a different parameter called \emph{length} defined in \eqref{eq:ldef} below, as well as equations which have some obvious integer solutions, where the problem is to show that there are no others. All our insolubility proofs for the equations that resisted the previously known methods are formalized in Lean \cite{moura2021lean,mathlib2020}. 

This work is organized into six sections. Section \ref{sec:method} describes our method for finding quadratic identities, suitable for applying quadratic reciprocity. Section \ref{sec:smallest} uses this method to prove the integer insolubility of $25$ out of $30$ equations in Table~\ref{tab:2varH34}, including all four previously smallest open equations of size $h=32$. Section \ref{sec:extensions} shows that the same method can be used to solve higher-degree equations, as well as determine the complete set of integer solutions for certain equations for which this set is non-empty. Section \ref{sec:graphs} presents the deeper method of cubic reciprocity graphs, and uses it to solve the remaining $5$ equations from Table~\ref{tab:2varH34}. Section \ref{sec:open} presents further extensions, discussion, a description of AI use, and open questions.  

\section{The quadratic identities method}\label{sec:method}

This section presents a method for certifying non-existence of integer points on curves, which is the main contribution of this work. 

\subsection{The quadratic reciprocity step}\label{sec:repstep}

All equations in Table \ref{tab:2varH34} are of the form
\begin{equation}\label{eq:shape}
        F(x,y):=y^3+a(x)y^2+b(x)y+c(x)=0,
\end{equation}
where $a(x)$, $b(x)$ and $c(x)$ are polynomials with integer coefficients of degree at most $4$. Our aim is to prove that equation \eqref{eq:shape} has no integer solutions. The idea of the method is to prove that equation \eqref{eq:shape} implies one or several identities of the form 
\begin{equation}\label{eq:identity}
	H(x,y) \cdot T(x,y) = U(x)^2 + d \cdot V(x)^2,
\end{equation}
where $H,T,U,V$ are polynomials with integer coefficients, and $d$ is either a fixed nonzero integer, which is the case in most examples, or, in advanced cases, a polynomial in $(x,y)$ of degree at most $2$. 

Once \eqref{eq:identity} is established, the law of quadratic reciprocity can often be used to restrict possible prime divisors of $H$, which may lead to a contradiction with \eqref{eq:shape}. Specifically, we will use the following result. Here $J(a,n)$ denotes the Jacobi symbol \cite[Section 7.3.2]{grechuk2024polynomial} for positive odd $n$, with $J(a,1)=1$, and $v_p(n)$ for nonzero $n$ denotes the exponent with which prime $p$ enters the prime factorization of $|n|$.

\begin{proposition}\label{prop:norm-common-divisor}
	Let $d,H,T,U,V$ be integers with $d\ne0$ and
	\begin{equation}\label{eq:identity2}
	HT=U^2+dV^2\ne0.
	\end{equation}
	Let $H_0$ be the largest positive divisor of $|H|$ coprime to $2d$. If $J(-d,H_0)=-1$, then there is an odd prime $p$ such that
	\[
	J(-d,p)=-1,\qquad p\mid H,T,U,V,
	\]
	and both $v_p(H)$ and $v_p(T)$ are odd. The same statement holds with $H$ and $T$ interchanged.
\end{proposition}
\begin{proof}
	The multiplicativity of the Jacobi symbol \cite[Proposition 7.26]{grechuk2024polynomial} implies that
	\[
	-1=J(-d,H_0)=\prod_{p\mid H_0}J(-d,p)^{v_p(H)},
	\]
	where we have used that $v_p(H_0)=v_p(H)$ whenever $p\mid H_0$. Thus some odd prime $p\nmid d$ satisfies $J(-d,p)=-1$ and has odd exponent in $H$. 
	Since $p\mid H$, reducing \eqref{eq:identity2} modulo $p$ gives
	\[
	U^2\equiv-dV^2\pmod p.
	\]
	Suppose that $p\nmid V$. Then $V$ has a multiplicative inverse $V^{-1}$ modulo $p$, so multiplying the congruence by $(V^{-1})^2$ gives
	\[
	(UV^{-1})^2\equiv-d\pmod p.
	\]
	Hence, $-d$ is a quadratic residue modulo $p$, which contradicts $J(-d,p)=-1$. Therefore $p\mid V$, and then \eqref{eq:identity2} implies that $p\mid U$. Thus, $p$ divides both $U$ and $V$. 
		
	Write $U=p^rU_0$, $V=p^rV_0$, with at least one of $U_0,V_0$ not divisible by $p$. If $p\mid U_0^2+dV_0^2$, the same argument would imply that  $p$ divides both $U_0$ and $V_0$, hence we must have $p\nmid U_0^2+dV_0^2$, and therefore
	\[
	v_p(H)+v_p(T)=v_p(U^2+dV^2)=2r.
	\]
	Since $v_p(H)$ is odd, $v_p(T)$ is odd as well. This proves the claim.
\end{proof}

In most examples, we will apply Proposition \ref{prop:norm-common-divisor} with $|d|\leq 3$, in which case it reduces to the following corollary.

\begin{corollary}\label{cor:norm-congruence-cases}
	Let $H,T,U,V$ be integers satisfying $HT=U^2+dV^2\ne0$, where $d\in\{1,2,-2,3,-3\}$, and let $H_0$ be the largest positive divisor of $|H|$ coprime to $2d$. If the condition in the corresponding row holds, then $H,T,U,V$ have a common odd prime divisor $p$ with the indicated congruence, and both $v_p(H)$ and $v_p(T)$ are odd.
	\begin{center}
		\begin{tabular}{@{}ccl@{}}
			\toprule
			$d$ & Condition on $H_0$ & Congruence of the odd prime $p$\\
			\midrule
			$1$  & $H_0\equiv3\pmod4$       & $p\equiv3\pmod4$\\
			$2$  & $H_0\equiv5,7\pmod8$     & $p\equiv5,7\pmod8$\\
			$-2$ & $H_0\equiv3,5\pmod8$     & $p\equiv3,5\pmod8$\\
			$3$  & $H_0\equiv2\pmod3$       & $p\equiv2\pmod3$\\
			$-3$ & $H_0\equiv5,7\pmod{12}$  & $p\equiv5,7\pmod{12}$\\
			\bottomrule
		\end{tabular}
	\end{center}
\end{corollary}
\begin{proof}
	For positive odd $n$ coprime to $2d$, quadratic reciprocity and its
	supplementary laws make $J(-d,n)=-1$ equivalent to the residue condition
	in the corresponding row. In particular, $J(-3,n)=J(n,3)$ and
	$J(3,n)=(-1)^{(n-1)/2}J(n,3)$. Apply these rules first to $H_0$, then
	to the prime supplied by Proposition~\ref{prop:norm-common-divisor}.
\end{proof}

Once identity \eqref{eq:identity} is established, the analysis of equation \eqref{eq:shape} may then proceed as follows. First we use \eqref{eq:shape} to derive suitable congruence conditions on $H_0$. Then we use Proposition~\ref{prop:norm-common-divisor} or Corollary \ref{cor:norm-congruence-cases} to conclude the existence of a common prime factor $p$ of $H,T,U,V$ satisfying the corresponding congruence conditions. If polynomials $U(x)$ and $V(x)$ in \eqref{eq:identity} are coprime, then polynomial division and clearing coefficient denominators give the B\'ezout identity
\[
E(x)U(x)+G(x)V(x)=c
\] 
for some polynomials $E,G\in\Z[x]$ and a nonzero integer $c$. Hence, $p \mid U(x)$ and $p \mid V(x)$ imply that $p \mid c$, which leaves only finitely many options for $p$. For each candidate $p$, we can solve the system
$$
F(x,y) \equiv H(x,y) \equiv T(x,y) \equiv U(x) \equiv V(x) \equiv 0 \pmod p,
$$  
and, if it has no solutions, obtain the desired contradiction. If $U(x)$ and $V(x)$ are not coprime, and can be written as $U(x)=R(x)U'(x)$ and $V(x)=R(x)V'(x)$ with $R,U',V'\in\Z[x]$ and coprime $U',V'$, then either $p \mid R(x)$ or $p$ is a common divisor of $U'(x)$ and $V'(x)$. The first case should be considered separately, while the second one reduces to finite enumeration in the same way as above.  

This argument, while it may be long, is usually standard and routine. The key step in the method is to derive a suitable identity \eqref{eq:identity}. This is the subject of the next subsections. 

\subsection{Choosing \texorpdfstring{$H$}{H} linear in \texorpdfstring{$y$}{y}}\label{sec:methodlin}
    
In the simplest version of the method, we try to find an auxiliary polynomial $H$ of the form
\begin{equation}\label{eq:Hlin}
	H(x,y)=qy+R(x)
\end{equation}
for some nonzero integer $q$ and polynomial $R(x)\in\Z[x]$. We regard $F$ and $H$ as polynomials in $y$, with $x$ as a parameter, perform polynomial division, and then, if necessary, multiply by $q^3$ to clear denominators. This gives an identity of the form
$$
q^3 F(x,y) = H(x,y)T_0(x,y) + P(x),
$$
for some polynomials $T_0(x,y)$ and $P(x)$ with integer coefficients. Hence, $F(x,y)=0$ implies that
$$
H(x,y)T_0(x,y) = -P(x).
$$
We then choose $R(x)$ so that $-\kappa P(x)=U^2+dV^2$ for a nonzero integer $\kappa$. Setting 
\begin{equation}\label{eq:Tlin}
	T=\kappa T_0
\end{equation}
gives \eqref{eq:identity}. 

We demonstrate this simple approach by solving the following equation, which has been left open in \cite[Section 8.3.2]{grechuk2024polynomial}.

\begin{example}\label{ex:2y3pxypx4p1}
	The equation
	\begin{equation}\label{eq:2y3pxypx4p1}
		2y^3 + xy + x^4 + 1 = 0
	\end{equation}
	has no integer solutions.
\end{example}
\begin{proof}
	Analysis modulo $4$ of this equation shows that $x$ is odd and $y \equiv 2 \,(\text{mod}\, 4)$. If $y \geq 0$, then the left-hand side of \eqref{eq:2y3pxypx4p1} is positive, hence we must have $y<0$. The equation can be rewritten as
	$$
		(-y)(2y^2+x) = (x^2)^2+1.
	$$ 
	Because $-y>0$ and $x$ is odd, the integer $A:=2y^2+x$ is positive and odd. Corollary \ref{cor:norm-congruence-cases} with $H=A$, $T=-y$, $U=x^2$ and $d=V=1$ implies that $A \equiv 1 \,(\text{mod}\, 4)$. Because $y$ is even, this implies that $x \equiv 1 \,(\text{mod}\, 4)$. 
	
	To eliminate this remaining case, let us apply the described method. Let $F(x,y)$ be the left-hand side of \eqref{eq:2y3pxypx4p1}, and let us choose 
	$$
		H(x,y)=x-y.
	$$
	Polynomial division of $F$ by $H$ gives
	$$
		F(x,y)=-(x-y)(2y^2+2xy+2x^2+x)+x^4+2x^3+x^2+1.
	$$
	Hence, $F(x,y)=0$ implies that
	\begin{equation}\label{eq:2y3pxypx4p1_ident}
		(x-y)(2y^2+2xy+2x^2+x) = x^4+2x^3+x^2+1 = x^2(x+1)^2 + 1.
	\end{equation}
	This is an identity of the form \eqref{eq:identity} with $d=1$, $U=x(x+1)$ and $V=1$. Since the right-hand side of \eqref{eq:2y3pxypx4p1_ident} is positive and 
	$$
		2y^2+2xy+2x^2+x = y^2+(x+y)^2+x(x+1) \geq 0, 	
	$$
	we conclude that $x-y>0$. We earlier established that $x \equiv 1 \,(\text{mod}\, 4)$ and $y \equiv 2 \,(\text{mod}\, 4)$, hence $x-y \equiv 3 \,(\text{mod}\, 4)$. However, Corollary \ref{cor:norm-congruence-cases} implies that every odd positive divisor of $U^2+1$ is congruent to $1$ modulo $4$, so \eqref{eq:2y3pxypx4p1_ident} implies that $x-y \equiv 1 \,(\text{mod}\, 4)$. This contradiction shows that equation \eqref{eq:2y3pxypx4p1} has no integer solutions.
\end{proof}

\subsection{Choosing \texorpdfstring{$H$}{H} quadratic in \texorpdfstring{$y$}{y}}\label{sec:methodquad}

The choice of $H$ to be linear in $y$, as in \eqref{eq:Hlin}, is quite restrictive. We next investigate a much more flexible choice
\begin{equation}\label{eq:Hquad}
	H(x,y)=qy^2+B(x)y+C(x)
\end{equation}
for some nonzero integer $q$ and polynomials $B(x)$ and $C(x)$ with integer coefficients. Polynomial division of $F(x,y)$ in \eqref{eq:shape} by $H(x,y)$, followed by multiplication by $q^2$ to clear denominators, now gives 
\begin{equation}\label{eq:division}
	q^2F(x,y)= S(x,y) H(x,y) + L(x) y + M(x)
\end{equation}
for some polynomials $S(x,y)$, $L(x)$ and $M(x)$. In fact, these polynomials can be written explicitly as 
\begin{equation}\label{eq:SLM}
	\begin{aligned}
		S(x,y)&=qy+qa(x)-B(x),\\
		L(x)&=q^2b(x)-qC(x)-qa(x)B(x)+B(x)^2,\\
		M(x)&=q^2c(x)-qa(x)C(x)+B(x)C(x),
	\end{aligned}
\end{equation}
where these formulas follow directly from polynomial division.

We now divide the quadratic polynomial \(H(x,y)\) by the linear polynomial
\(L(x)y+M(x)\), regarding both as polynomials in \(y\) with parameter $x$. Let us assume for a moment that $L(x)\neq 0$. Since the root of \(L(x)y+M(x)\) is
$$
y=-\frac{M(x)}{L(x)},
$$
the remainder of polynomial division is
$$
H\left(x,-\frac{M(x)}{L(x)}\right)
=
\frac{qM(x)^2-B(x)L(x)M(x)+C(x)L(x)^2}{L(x)^2}.
$$

More explicitly, polynomial division gives
\begin{equation}\label{eq:seconddivision}
	H(x,y)
	=
	\frac{qL(x)y+B(x)L(x)-qM(x)}{L(x)^2}
	\bigl(L(x)y+M(x)\bigr)
	+
	\frac{R_H(x)}{L(x)^2},
\end{equation}
where
\begin{equation}\label{eq:RH}
	R_H(x)
	=
	qM(x)^2-B(x)L(x)M(x)+C(x)L(x)^2.
\end{equation}
Clearing denominators in \eqref{eq:seconddivision}, we obtain the polynomial
identity
\begin{equation}\label{eq:seconddivision-cleared}
	L(x)^2H(x,y)
	=
	\bigl(qL(x)y+B(x)L(x)-qM(x)\bigr)
	\bigl(L(x)y+M(x)\bigr)
	+
	R_H(x).
\end{equation}
Although \eqref{eq:seconddivision} was derived under the assumption $L(x)\neq0$, the cleared relation \eqref{eq:seconddivision-cleared} is a polynomial identity and therefore remains valid when $L(x)=0$.

Now suppose that \((x,y)\) is an integer solution of \(F(x,y)=0\).  From
\eqref{eq:division} we then have
\begin{equation}\label{eq:linear-remainder-on-curve}
	L(x)y+M(x)=-S(x,y)H(x,y).
\end{equation}
Substituting \eqref{eq:linear-remainder-on-curve} into
\eqref{eq:seconddivision-cleared} gives

$$
L(x)^2H(x,y)
=
-S(x,y)\bigl(qL(x)y+B(x)L(x)-qM(x)\bigr)H(x,y)
+
R_H(x).
$$

Hence every integer solution of \(F(x,y)=0\) satisfies
\begin{equation}\label{eq:HTRH}
	H(x,y)T_0(x,y)=R_H(x),
\end{equation}
where
\begin{equation*}
	T_0(x,y)
	=
	L(x)^2
	+
	S(x,y)\bigl(qL(x)y+B(x)L(x)-qM(x)\bigr).
\end{equation*}

The objective is to choose $q$, $B(x)$, $C(x)$ in \eqref{eq:Hquad} for which $R_H(x)$ can be represented as 
\begin{equation}\label{eq:kappaRH}
 \kappa R_H(x)=U(x)^2+dV(x)^2,
\end{equation}
for some nonzero integer $\kappa$ and polynomials $U(x), V(x)$. Multiplying \eqref{eq:HTRH} by $\kappa$ gives the identity
\eqref{eq:identity}, with 
\begin{equation}\label{eq:Tgeneral}
 T=\kappa T_0
  =\kappa\bigl[L^2+S(qLy+BL-qM)\bigr].
\end{equation}

Once identity \eqref{eq:identity} is established, we then proceed as in Section \ref{sec:repstep}. If one identity \eqref{eq:identity} does not suffice for a contradiction, we may keep trying different linear \eqref{eq:Hlin} and quadratic \eqref{eq:Hquad} functions $H$ to derive more identities of the form \eqref{eq:identity}. As a result, we conclude that \eqref{eq:shape} implies that
\begin{equation}\label{eq:identitymany}
H_i(x,y) \cdot T_i(x,y) = U_i(x)^2 + d_i \cdot V_i(x)^2, \qquad i=1, \dots, k,
\end{equation}
and then apply Proposition \ref{prop:norm-common-divisor} for each of the $k$ identities in \eqref{eq:identitymany}.

This method is surprisingly powerful. In the remaining sections, we use it to solve many equations that were left open in the previous literature. 

\section{The smallest open equations}\label{sec:smallest}

This section demonstrates the power of the method described in Section \ref{sec:method} by using it to prove that $25$ out of $30$ equations in Table~\ref{tab:2varH34} have no integer solutions. The insolubility of the remaining five equations will be proved in Section \ref{sec:graphs}. Together, this will prove the following result.

\begin{theorem}\label{prop:table-one}
	All equations in Table~\ref{tab:2varH34} have no integer solutions.
\end{theorem}

We give a detailed proof of Theorem \ref{prop:table-one} for the four smallest equations, of size $h=32$. For the remaining $21$ equations solved in this section, we present the key identities in the form of compact tables, and provide references to the detailed arguments and Lean formalizations.   

Section~\ref{sec:easy} solves three equations of size $h=32$ by deriving and using a single identity \eqref{eq:identity} with quadratic $H$ and constant $d$. Section~\ref{sec:variable-norm} treats the remaining more advanced equation of size $h=32$ by using two identities \eqref{eq:identity} with nonconstant $d$. Section~\ref{sec:H33-34} gives a short outline for the solved equations of sizes $33$ and $34$. 

\subsection{Three proofs with one norm identity}\label{sec:easy}

\begin{proposition}\label{prop:h32-2xminusy}
The equation
\begin{equation}\label{eq:x4py3p2xmym2}
    y^3-y=x^4+2x-2
\end{equation}
has no integer solutions.
\end{proposition}

\begin{proof}
Put $F(x,y)=y^3-y-x^4-2x+2$. It is of the
form \eqref{eq:shape} with
\[
    a(x)=0,\qquad b(x)=-1,\qquad c(x)=-x^4-2x+2.
\]
We choose $H(x,y)$ of the form \eqref{eq:Hquad} with $q=1$ and
\[
    B(x)=-2x-2,\qquad C(x)=4x^2+3.
\]
These polynomials were found by a direct search aimed at obtaining representation \eqref{eq:kappaRH} of the polynomial $R_H(x)$ below. With this choice of $B(x)$ and $C(x)$, \eqref{eq:Hquad} gives
\begin{equation*}
    H(x,y)=y^2-(2x+2)y+4x^2+3.
\end{equation*}
The identity
\[
    3H(x,y)=3(y-x-1)^2+(3x-1)^2+5
\]
shows that $H(x,y)>0$. A direct check shows that the only solutions to \eqref{eq:x4py3p2xmym2} modulo $4$ are $(x,y)=(0,2)$ and $(x,y)=(2,2)$. In both cases
\begin{equation}\label{eq:Hmod4x4py3p2xmym2}
    H(x,y)\equiv3\pmod4.
\end{equation}

Next, formulas \eqref{eq:SLM} give
\[
    S(x,y)=y+2x+2,\qquad L(x)=8x,
\]
\[
    M(x)=-x^4-8x^3-8x^2-8x-4.
\]
Substituting these polynomials into \eqref{eq:RH}, we obtain
\[
\begin{aligned}
R_H(x)&=x^8+16x^7+64x^6+200x^4-64x^3+128x^2+16\\
&=U(x)^2+V(x)^2,
\end{aligned}
\]
where
\[
    U(x)=x^4+8x^3-2x^2+12x,
    \qquad
    V(x)=2x^3+2x^2-4.
\]
This is the desired representation \eqref{eq:kappaRH} with $\kappa=d=1$, so \eqref{eq:Tgeneral} gives 
\[
\begin{aligned}
T(x,y)={}&2x^5+x^4y+18x^4+8x^3y+8x^2y+32x^2\\
&\quad{}+8xy^2+8xy-8x+4y+8.
\end{aligned}
\]
Hence the general identity \eqref{eq:HTRH} shows that every integer solution
of \eqref{eq:x4py3p2xmym2} would satisfy
\begin{equation}\label{eq:mainx4py3p2xmym2}
    H(x,y)T(x,y)=U(x)^2+V(x)^2.
\end{equation}

The system $U(x)=V(x)=0$ has no integer solutions, hence $U(x)^2+V(x)^2>0$. Thus $H>0$ and, by \eqref{eq:Hmod4x4py3p2xmym2}, $H\equiv3\pmod4$. Corollary~\ref{cor:norm-congruence-cases}, with $d=1$, applied to \eqref{eq:mainx4py3p2xmym2} implies the existence of a prime $p\equiv3\pmod4$ dividing $H,T,U,V$, with odd exponent in $H$.

On the other hand, direct polynomial division gives the B\'ezout identity
\[
    E(x)U(x)+G(x)V(x)=1216,
\]
where
\[
    E(x)=28x^2+18x+18,
    \qquad
    G(x)=-14x^3-107x^2+54x-304.
\]
Thus $p\mid1216=2^6\cdot19$. Since $p\equiv3\pmod4$, we must have $p=19$.
The simultaneous congruences
\[
    F(x,y)\equiv H(x,y)\equiv0\pmod{19}
\]
have the unique solution $(x,y)\equiv(1,6)\pmod{19}$, and for this pair
\[
    T(x,y)\equiv1\pmod{19}.
\]
This contradicts $19\mid T$. Therefore, equation \eqref{eq:x4py3p2xmym2} has no integer
solutions.
\end{proof}

\begin{proposition}\label{prop:h32-xyplus4}
The equation
\begin{equation}\label{eq:x4py3pxpyp4}
    y^3+y=x^4+x+4
\end{equation}
has no integer solutions.
\end{proposition}

\begin{proof}
Put $F(x,y)=y^3+y-x^4-x-4$. It is of the
form \eqref{eq:shape} with
\[
    a(x)=0,\qquad b(x)=1,\qquad c(x)=-x^4-x-4.
\]
We choose $H(x,y)$ of the form \eqref{eq:Hquad} with $q=1$ and
\[
    B(x)=1,\qquad C(x)=2x^2+4x+5,
\]
so that
\begin{equation*}
    H(x,y)=y^2+y+2x^2+4x+5.
\end{equation*}
The identity
\[
    4H(x,y)=(2y+1)^2+8(x+1)^2+11
\]
shows that $H(x,y)>0$. Also $H$ is odd, because $y(y+1)$ is even. A direct
enumeration of the solutions of \eqref{eq:x4py3pxpyp4} modulo $9$ implies that
\begin{equation}\label{eq:Hmod9-x4py3pxpyp4}
    H\equiv 2,5,6, \, \text{or} \, 8 \pmod 9.
\end{equation}

From \eqref{eq:SLM} we obtain
\[
    S(x,y)=y-1,
    \qquad L(x)=-2x^2-4x-3,
\]
\[
    M(x)=-x^4+2x^2+3x+1.
\]
Formula \eqref{eq:RH} now gives
\begin{equation*}
    4R_H(x)=U(x)^2+3V(x)^2,
\end{equation*}
where
\[
    U(x)=2x^4-x^2+5x+7,
    \qquad
    V(x)=2x^3+11x^2+15x+7.
\]
This is representation \eqref{eq:kappaRH} with $\kappa=4$ and $d=3$. Thus \eqref{eq:Tgeneral} gives 
\[
\begin{aligned}
T(x,y)={}&4\bigl(x^4y+3x^4+16x^3-2x^2y^2-2x^2y+32x^2\\
&\qquad{}-4xy^2-3xy+31x-3y^2-y+13\bigr).
\end{aligned}
\]
Hence, \eqref{eq:HTRH} implies that every integer solution of \eqref{eq:x4py3pxpyp4} would satisfy
\begin{equation*}
    H(x,y)T(x,y)=U(x)^2+3V(x)^2.
\end{equation*}

The system $U(x)=V(x)=0$ has no integer solutions, hence $U(x)^2+3V(x)^2>0$. By \eqref{eq:Hmod9-x4py3pxpyp4}, either $H\equiv2\pmod3$, or $v_3(H)=1$ and $H/3\equiv2\pmod3$. In any case, the positive odd integer $H_0=H/3^{v_3(H)}$ satisfies $H_0\equiv2\pmod3$. Corollary~\ref{cor:norm-congruence-cases}, with $d=3$, implies the existence of an odd prime $p\equiv2\pmod3$ dividing $H,T,U,V$.

The polynomial B\'ezout identity
\begin{equation*}
    E(x)U(x)+G(x)V(x)=1190,
\end{equation*}
where
\[
    E(x)=-316x^2-1304x-445,
\]
\[
    G(x)=316x^3-434x^2+304x+615,
\]
implies that $p\mid1190=2\cdot5\cdot7\cdot17$, and the condition $p\equiv2\pmod3$ leaves only $p=5$ or $p=17$.

There is no solution of $F\equiv H\equiv0\pmod{17}$. Modulo $5$, the unique
solution is $(x,y)\equiv(3,4)$, and in this case
\[
    T(x,y)\equiv2\pmod5,
\]
which is a contradiction with $p\mid T$. Therefore, equation \eqref{eq:x4py3pxpyp4} has no integer solutions.
\end{proof}

The next example illustrates that the method works when polynomials $U(x)$ and $V(x)$ in \eqref{eq:identity} are not coprime. 

\begin{proposition}\label{prop:h32-xy4}
The equation
\begin{equation}\label{eq:x4py3pxyp4}
    y^3+xy=x^4+4
\end{equation}
has no integer solutions.
\end{proposition}

\begin{proof}
Put $F(x,y)=y^3+xy-x^4-4$. It is of the
form \eqref{eq:shape} with
\[
    a(x)=0,\qquad b(x)=x,\qquad c(x)=-x^4-4.
\]
Direct checks modulo $8$ and $3$ show that every solution to \eqref{eq:x4py3pxyp4} must satisfy
\begin{equation}\label{eq:mod4-x4py3pxyp4}
    x\equiv y\equiv2\pmod4, \quad \quad (x,y)\equiv(0,1)\text{ or }(1,1)\pmod3.
\end{equation}

We choose $H(x,y)$ of the form \eqref{eq:Hquad} with $q=1$ and
\[
B(x)=1,\qquad C(x)=4A(x), \quad \text{where} \quad A(x)=x^2+2x+2=(x+1)^2+1,
\]
so that
\begin{equation*}
	H(x,y)=y^2+y+4A(x) = y^2+y+4x^2 + 8x + 8.
\end{equation*}
Since $A(x)>0$ and $y(y+1)\ge0$ for every integer $y$, we have $H(x,y)>0$.
The congruences \eqref{eq:mod4-x4py3pxyp4} give
\begin{equation}\label{eq:Hres-x4py3pxyp4}
    H\equiv2\pmod4,
    \qquad H\equiv1\pmod3.
\end{equation}

Polynomial division of $F(x,y)$ by $H(x,y)$ gives 
\begin{equation}\label{eq:divisioncurrent}
	F(x,y)= S(x,y) H(x,y) + L(x) y + M(x)
\end{equation}
where, by \eqref{eq:SLM}, 
\begin{equation}\label{eq:SLM_Def}
    S(x,y)=y-1,
    \qquad L(x)=x+1-4A(x),
	\qquad 
    M(x)=-A(x)(x^2-2x-2).
\end{equation}
Substituting this in \eqref{eq:RH} results in
\begin{equation*}
    R_H(x)=A(x)^2Q(x), \quad \text{where} \quad Q(x)=x^4-4x^3+60x^2+113x+109.
\end{equation*}

Moreover,
\begin{equation}\label{eq:Qnorm-x4py3pxyp4}
    4Q(x)=U(x)^2+3V(x)^2,
\end{equation}
with
\[
    U(x)=2x^2-4x-19,
    \qquad V(x)=5(2x+1).
\]
This implies that $4R_H=(AU)^2+3(AV)^2$, which is formula \eqref{eq:kappaRH} with $\kappa=4$, $d=3$, and the polynomials of the right-hand side having a common factor $A(x)$. 

Formula \eqref{eq:Tgeneral} gives
\begin{equation}\label{eq:T-x4py3pxyp4}
    T=4\bigl[L^2+(y-1)(Ly+L-M)\bigr].
\end{equation}
Hence, \eqref{eq:HTRH} implies that every integer solution of \eqref{eq:x4py3pxyp4} would satisfy
\begin{equation}\label{eq:main-x4py3pxyp4}
    H(x,y)T(x,y)=4A(x)^2Q(x)=(A(x)U(x))^2+3(A(x)V(x))^2.
\end{equation}
Since $V(x)\ne0$ for integral $x$, \eqref{eq:Qnorm-x4py3pxyp4} gives $Q>0$.
Together with $H>0$, equation \eqref{eq:main-x4py3pxyp4} implies that $T>0$.

By \eqref{eq:Hres-x4py3pxyp4}, $H_0=H/2$ is positive and odd with
$H_0\equiv2\pmod3$.
Corollary~\ref{cor:norm-congruence-cases}, with $d=3$, implies the existence of an odd prime $p\equiv2\pmod3$ with odd exponent in $H$ such that $p\mid H,T,AU,AV$. Hence, we need to consider two cases: either $p\mid A$, or $p\mid U,V$.

First assume that $p\mid U,V$. The congruence $U\equiv0\pmod5$ has no solutions, hence $p\ne5$, and $p \mid V=5(2x+1)$ implies that $p \mid 2x+1$. The identity
\begin{equation*}
    2U(x)=(2x+1)^2-6(2x+1)-33
\end{equation*}
then gives $p\mid33$, so $p=11$. The condition $p\mid V$ gives $x\equiv5\pmod{11}$. The simultaneous congruences
$F\equiv H\equiv0\pmod{11}$ have the unique solution $(x,y)\equiv(5,8)\pmod{11}$, but then
\eqref{eq:T-x4py3pxyp4} implies that
\[
 T \equiv7\pmod{11}.
\]
This contradicts $p\mid T$, so the case $p\mid U,V$ is impossible. 

It remains to consider the case $p\mid A$. In this case, \eqref{eq:SLM_Def} implies that $M\equiv 0$ and $L\equiv x+1$ modulo $p$. Then $L^2 \equiv (x+1)^2 = A-1 \equiv -1$ modulo $p$. In particular, $p \nmid L$. Further, $F(x,y)=0$ and \eqref{eq:divisioncurrent} imply that $p\mid SH+Ly+M$. Because $p\mid H$ and $p\mid M$, this gives $p\mid Ly$. Because $p \nmid L$, we must have $p \mid y$. Because $p\mid T$,  \eqref{eq:T-x4py3pxyp4} then implies that
$$
0 \equiv 4(L^2+(0-1)(L\cdot 0+L-0)) \pmod{p},
$$ 
hence $p\mid 4L(L-1)$. Because $p\nmid 4$ and $p \nmid L$, this implies that $p \mid L-1$, or $p \mid x$. But in this case $p \mid A$ implies that $p \mid 2$, a contradiction.

Hence, equation \eqref{eq:x4py3pxyp4} has no integer solutions.
\end{proof}

\subsection{A proof with two norm identities}\label{sec:variable-norm}

To solve the last equation of size $h=32$, we will derive two identities \eqref{eq:identity} with nonconstant $d$. 

\begin{proposition}\label{prop:tb0902-short}
The equation
\begin{equation}\label{eq:tb0902-short-main}
 y^3+xy=x^4+x+2
\end{equation}
has no integer solutions.
\end{proposition}

\begin{proof}
Let $F(x,y)=y^3+xy-x^4-x-2$, and suppose that $F(x,y)=0$ for integers
$x,y$. Since $x^4+x+2>0$, we have $y(y^2+x)>0$. If $y<0$, then $y^2+x<0$, so $x<0$ and
$|y|<\sqrt{|x|}\le |x|$. In this case,
\[
 y^3+xy=|y|\bigl(|x|-|y|^2\bigr)
 <|y|\,|x|<|x|^2<|x|^4-|x|+2=x^4+x+2,
\]
contradicting \eqref{eq:tb0902-short-main}. Thus $y\ge0$. Direct
substitution also excludes $y=0$ and $y=1$, so $y\ge2$.

Reduction of \eqref{eq:tb0902-short-main} modulo $4$ shows that $y$ is
even. Thus $n:=y-1$ is positive and odd. The equation can be rewritten as
\[
 (y-1)(y^2+y+x+1)=x^4+1.
\]
The two representations
\[
 x^4+1=(x^2)^2+1^2=(x^2-1)^2+2x^2,
 \qquad \gcd(x^2,1)=\gcd(x^2-1,x)=1,
\]
allow us to apply Corollary~\ref{cor:norm-congruence-cases} with $d=1,2$.
They show that $x^4+1$ has no positive odd divisor congruent to $3$ modulo
$4$ or to $5,7$ modulo $8$: such a divisor would force a common prime
divisor of the corresponding coprime norm coordinates. Hence
$n\equiv1\pmod8$. Substituting $y\equiv2\pmod8$ into the original
equation gives
\begin{equation}\label{eq:tb0902-short-local}
 y\equiv2\pmod8,\qquad x\equiv2\text{ or }3\pmod8.
\end{equation}

We now apply the method of Section~\ref{sec:method}, first choosing
\[
 H_1(x,y)=y+x.
\]
This polynomial has the form \eqref{eq:Hlin}, with $q=1$ and $R(x)=x$.
Polynomial division of $F$ by $H_1$ gives
\begin{equation}\label{eq:tb0902-division1}
\begin{aligned}
 F(x,y)&=(y+x)(y^2-xy+x^2+x)-Q_1(x),\\
 Q_1(x)&=x^4+x^3+x^2+x+2.
\end{aligned}
\end{equation}
As in Section~\ref{sec:methodlin}, we seek a representation
\begin{equation}\label{eq:tb0902-aux-iden}
 \kappa Q_1(x)=U(x)^2+d V(x)^2
\end{equation}
for a nonzero integer $\kappa$. Since a useful representation with constant
$d$ was not found, we allow $d$ to be a polynomial of low degree. If $V$
is coprime to $Q_1$ in $\Q[x]$, then $V$ is invertible modulo $Q_1$, and
\eqref{eq:tb0902-aux-iden} gives
$(UV^{-1})^2\equiv-d\pmod{Q_1}$. This suggests looking at the
remainders of small squares. Polynomial division gives
\[
 x^4\equiv-x^3-x^2-x-2,\qquad
 x^5\equiv2-x,\qquad x^6\equiv2x-x^2\pmod{Q_1}.
\]
Thus $2x-x^2$ is congruent to the square $(x^3)^2$ modulo $Q_1$, suggesting
$d(x)=x^2-2x$. And indeed a quick search gives
\[
  2Q_1(x)=U_1(x)^2+(x^2-2x)V_1(x)^2 \qquad \text{where} \quad U_1(x)=x^2+x+2,\quad V_1(x)=x+1.
\]
Combining this with \eqref{eq:tb0902-division1}, we conclude that every integer solution of $F(x,y)=0$ satisfies
\[
 H_1(x,y)t_1(x,y)=U_1(x)^2+(x^2-2x)V_1(x)^2,
\]
where the quotient is
\[
 t_1(x,y)=2(y^2-xy+x^2+x).
\]
Adding $kH_1V_1^2$ to both sides gives the family of identities
\begin{equation}\label{eq:tb0902-identity1k}
 H_1(t_1+kV_1^2)=U_1^2+\bigl(x^2-2x+kH_1\bigr)V_1^2,
 \qquad k\in\Z,
\end{equation}
all valid on $F=0$.

To obtain a second character restriction, we seek a second norm identity
whose first factor is the coefficient of $V_1^2$ in
\eqref{eq:tb0902-identity1k}. This suggests trying
$H_2=x^2-2x+kH_1$ for small $k$. The choice $k=3$ gives
\[
 H_2(x,y)=3y+x^2+x,
\]
and will produce the required pair of identities. This polynomial again
has the form \eqref{eq:Hlin}, now with $q=3$ and $R(x)=x^2+x$.
The division prescribed in Section~\ref{sec:methodlin} gives, after clearing
denominators,
\begin{equation}\label{eq:tb0902-division2}
 27F(x,y)=H_2(x,y)
 \bigl(9y^2-3(x^2+x)y+(x^2+x)^2+9x\bigr)-Q_2(x),
\end{equation}
where
\[
 Q_2(x)=x^6+3x^5+30x^4+10x^3+9x^2+27x+54.
\]
For this remainder, we seek a representation \eqref{eq:tb0902-aux-iden} with the same $d$, possibly up to sign, and the search returned 
\begin{equation*}
	2Q_2(x)=3U_2(x)^2-(x^2-2x)V_2(x)^2,
\end{equation*}
where
\[
 U_2(x)=3x^3+2x^2-x-6,\qquad V_2(x)=5x^2+8x+3=(x+1)(5x+3).
\]

Combining this with \eqref{eq:tb0902-division2}, we conclude that every integer solution of $F(x,y)=0$ satisfies
\begin{equation}\label{eq:tb0902-second-aux}
 H_2(x,y)t_2(x,y)=3U_2(x)^2-(x^2-2x)V_2(x)^2,
\end{equation}
where
\[
 t_2(x,y)=18y^2-6(x^2+x)y+2(x^2+x)^2+18x.
\]
To obtain an identity of the form \eqref{eq:identity} with norm coefficient $H_1$, observe that
\[
 H_2-(x^2-2x)=3(y+x)=3H_1.
\]
Adding $H_2V_2^2$ to both sides of \eqref{eq:tb0902-second-aux} therefore gives
\[
 H_2(t_2+V_2^2)=3U_2^2+3H_1V_2^2.
\]
The polynomial $t_2+V_2^2$ has every coefficient divisible by $3$.
Dividing by $3$ yields
\begin{equation}\label{eq:tb0902-identity2}
 H_2(x,y)T_2(x,y)=U_2(x)^2+H_1(x,y)V_2(x)^2,
\end{equation}
where the quotient is explicitly
\begin{equation*}
\begin{aligned}
 T_2(x,y)&=\frac{t_2(x,y)+V_2(x)^2}{3}\\
 &=6y^2-(2x^2+2x)y+9x^4+28x^3+32x^2+22x+3.
\end{aligned}
\end{equation*}
We will use this together with \eqref{eq:tb0902-identity1k} for $k=3$, namely
\begin{equation}\label{eq:tb0902-identity1}
 H_1(x,y)T_1(x,y)=U_1(x)^2+H_2(x,y)V_1(x)^2,
\end{equation}
where
\[
 T_1(x,y)=t_1(x,y)+3V_1(x)^2
 =2y^2-2xy+5x^2+8x+3.
\]

The derivation also gives the polynomial identities
\[
 H_1T_1-U_1^2-H_2V_1^2=2F,\qquad
 H_2T_2-U_2^2-H_1V_2^2=18F.
\]
We could present a much shorter proof by just stating the final identities and checking these two equalities by expansion. The preceding derivation explains how the crucial identities \eqref{eq:tb0902-identity2} and \eqref{eq:tb0902-identity1} were found.

Now we are ready to derive the final contradiction. First observe that $H_2=x(x+1)+3y>0$. Also $H_1>0$: otherwise $x=-a$ with
$a\ge y\ge2$, and
\[
 (a-y)(a^2+ay+y^2-a)\ge0
 \quad\Longrightarrow\quad
 y^3-ay\le a^3-a^2<a^4-a+2,
\]
contrary to the equation. Further, no odd prime $p$ can divide both $H_1$ and $H_2$. Indeed, $H_1\equiv H_2\equiv0\pmod p$ gives $y\equiv-x\pmod p$ and $x(x-2)\equiv0\pmod p$. Thus $x\equiv0 \pmod p$ or $x\equiv2 \pmod p$. Substituting this together with $y\equiv-x\pmod p$ in the original equation \eqref{eq:tb0902-short-main} would result in $-2\equiv0 \pmod p$ or $-32\equiv0 \pmod p$, but neither of these congruences is possible modulo an odd prime $p$.

Write $H_1=2^{e_1}h_1$ and $H_2=2^{e_2}h_2$, with $h_1,h_2$ positive and odd. Since $H_1,H_2$ have no common odd prime divisor, $h_i$ is the largest divisor of $H_i$ coprime to $2H_{3-i}$ for $i=1,2$. The right-hand sides of both identities \eqref{eq:tb0902-identity2} and \eqref{eq:tb0902-identity1} are nonzero: $U_1=x^2+x+2>0$, while
$V_2=(x+1)(5x+3)$ can vanish at an integer only when $x=-1$, and $U_2(-1)=-6$.

If $J(-H_2,h_1)=-1$, Proposition~\ref{prop:norm-common-divisor} applied to \eqref{eq:tb0902-identity1} gives an odd prime
$p\mid U_1,V_1$. This contradicts $U_1-xV_1=2$. Hence $J(-H_2,h_1)=1$.

Similarly, if $J(-H_1,h_2)=-1$, Proposition~\ref{prop:norm-common-divisor} applied to \eqref{eq:tb0902-identity2} gives an odd prime $p$ such that
\[
p\mid H_2,T_2,U_2,V_2,\qquad J(-H_1,p)=-1.
\]
Since $V_2=(x+1)(5x+3)$, either $p\mid x+1$, in which case $U_2(-1)=-6$ gives $p=3$, or $p\mid 5x+3$. For the latter case, polynomial division gives
\[
125U_2=(5x+3)(75x^2+5x-28)-666, \qquad 666=2\cdot3^2\cdot37.
\]
Thus $p=3$ or $p=37$. For $p=3$, the congruences $F\equiv H_2\equiv0$ force $(x,y)\equiv(0,2)$, giving $J(-H_1,3)=1$, a contradiction. For $p=37$, the congruences $U_2\equiv V_2\equiv H_2\equiv0$ force $(x,y)\equiv(29,6)$, where $T_2\equiv8\pmod{37}$, contradicting $p\mid T_2$. Therefore
\begin{equation}\label{eq:tb0902-short-Jacobi}
	J(-H_2,h_1)=1,\qquad J(-H_1,h_2)=1.
\end{equation}

We now use the two cases in \eqref{eq:tb0902-short-local}. If $x\equiv2\pmod8$, then $H_1\equiv H_2\equiv4\pmod8$, so
$e_1=e_2=2$. Moreover,
\[
 H_1+H_2=x^2+2x+4y\equiv0\pmod{16},
\]
and hence $h_1+h_2\equiv0\pmod4$. Exactly one of $h_1,h_2$ is
$3\pmod4$. Put $\alpha=(h_1-1)/2$ and $\beta=(h_2-1)/2$.
Quadratic reciprocity gives
\[
 J(-H_2,h_1)J(-H_1,h_2)=J(-h_2,h_1)J(-h_1,h_2)
 =(-1)^{\alpha+\beta+\alpha\beta}=-1,
\]
contradicting \eqref{eq:tb0902-short-Jacobi}.

If $x\equiv3\pmod8$, then $H_1=h_1\equiv5\pmod8$, $H_2=2h_2$,
and $h_2\equiv1\pmod4$. Thus $J(-2,h_1)=-1$ and $J(-1,h_2)=1$.
Equation~\eqref{eq:tb0902-short-Jacobi} gives $J(h_2,h_1)=-1$ and
$J(h_1,h_2)=1$, whereas $h_1\equiv h_2\equiv1\pmod4$ makes these
symbols equal by quadratic reciprocity. This is the final contradiction, which proves that equation \eqref{eq:tb0902-short-main} has no integer solutions.
\end{proof}

\subsection{Equations of sizes 33 and 34}\label{sec:H33-34}

This subsection outlines the proof of Theorem \ref{prop:table-one} for all equations in Table~\ref{tab:2varH34} of size $h=33$, and for $16$ out of $21$ equations of size $h=34$. 

We present the corresponding certificates below. Nineteen equations use the constant quadratic norms in Table~\ref{tab:compact-norms}; two use the coupled quadratic characters displayed afterwards. 

\begingroup
\small
\setlength{\tabcolsep}{4pt}
\renewcommand{\arraystretch}{1.12}
\begin{longtable}{@{}p{0.36\textwidth}p{0.60\textwidth}@{}}
\caption{Constant quadratic norm certificates. Each row gives $HT=U^2+dV^2$; $T$ can be derived from
\eqref{eq:Tlin} or \eqref{eq:Tgeneral} with $\kappa$ given in the table. Two rows for the same equation mean that both identities are used.}
\label{tab:compact-norms}\\
\toprule
Equation, $d$ and $\kappa$ & Auxiliary polynomials $H$, $U$ and $V$ \\
\midrule
\endfirsthead
\multicolumn{2}{l}{\textit{Table \thetable\ continued}}\\
\toprule
Equation, $d$ and $\kappa$ & Auxiliary polynomials $H$, $U$ and $V$ \\
\midrule
\endhead
\bottomrule
\endfoot
$y^3+y^2+xy=x^4+1$\par\smallskip
$d=1,\quad\kappa=1$\par & 
$\begin{aligned} H&=x^{2} + 2 x + 2 y,\\ U&=x^{3} + 3 x^{2} + 2,\\ V&=2 - 3 x^{2}.\end{aligned}$
\\\addlinespace[5pt]
$y^3+xy=x^4+x+3$\par\smallskip
$d=2,\quad\kappa=1$\par & 
$\begin{aligned} H&=2 x^{2} - 5 x + y^{2} + 2 y \left(x - 2\right) + 6,\\ U&=x^{4} - 2 x^{3} + 2 x^{2} - 3 x + 11,\\ V&=2 \left(x^{3} - 3 x^{2} + 7 x - 4\right).\end{aligned}$
\\\addlinespace[5pt]
$y^3+xy=x^4+x+3$\par\smallskip
$d=2,\quad\kappa=1$\par & 
$\begin{aligned} H&=4 x^{2} + 4 x + 4 y^{2} + 5 y + 3,\\ U&=4 \left(8 x^{4} - 11 x^{2} + 2 x - 10\right),\\ V&=4 \left(13 - 16 x^{2}\right).\end{aligned}$
\\\addlinespace[5pt]
$y^3+xy=x^4-x-3$\par\smallskip
$d=-3,\quad\kappa=1$\par & 
$\begin{aligned} H&=x^{2} + 4 x + 2 y + 1,\\ U&=x^{3} + 6 x^{2} + 25 x + 22,\\ V&=3 x^{2} + 14 x + 13.\end{aligned}$
\\\addlinespace[5pt]
$y^3+xy=x^4+x-3$\par\smallskip
$d=2,\quad\kappa=1$\par & 
$\begin{aligned} H&=8 x^{2} - 2 x + y^{2} + 2 y \left(x + 2\right) + 11,\\ U&=\left(x^{2} + 2\right) \left(x^{2} - 20 x + 12\right),\\ V&=\left(5 x + 11\right) \left(x^{2} + 2\right).\end{aligned}$
\\\addlinespace[5pt]
$y^3+xy+y=x^4-x+1$\par\smallskip
$d=2,\quad\kappa=1$\par & 
$\begin{aligned} H&=- 4 x + 2 y + 5,\\ U&=8 x - 11,\\ V&=2 \left(x^{2} - 4 x + 2\right).\end{aligned}$
\\\addlinespace[5pt]
$y^3+y^2=x^4+x+4$\par\smallskip
$d=7,\quad\kappa=4$\par & 
$\begin{aligned} H&=2 x + y + 3,\\ U&=2 x^{2} + 8 x + 9,\\ V&=2 x + 1.\end{aligned}$
\\\addlinespace[5pt]
$y^3-y^2=x^4+2x+2$\par\smallskip
$d=-2,\quad\kappa=1$\par & 
$\begin{aligned} H&=x^{2} - 2 x + 2 y - 3,\\ U&=x^{3} - 3 x^{2} + 11 x - 5,\\ V&=3 x^{2} - 8 x + 3.\end{aligned}$
\\\addlinespace[5pt]
$y^3+y^2=x^4+2x-2$\par\smallskip
$d=-6,\quad\kappa=1$\par & 
$\begin{aligned} H&=- 2 y + \left(x - 1\right)^{2},\\ U&=\left(x - 1\right)^{3} - 4,\\ V&=x^{2} + 1.\end{aligned}$
\\\addlinespace[5pt]
$y^3-y^2+xy=x^4+2$\par\smallskip
$d=-5,\quad\kappa=4$\par & 
$\begin{aligned} H&=x^{2} + 5 x - 2 y + 5,\\ U&=2 x^{3} + 15 x^{2} + 35 x + 19,\\ V&=3 x^{2} + 3 x + 5.\end{aligned}$
\\\addlinespace[5pt]
$y^3+y=x^4+x+6$\par\smallskip
$d=7,\quad\kappa=1$\par & 
$\begin{aligned} H&=3 x^{2} + 6 x + 2 y^{2} - y \left(x - 2\right) + 8,\\ U&=5 x^{4} + 11 x^{3} + 13 x^{2} + 8,\\ V&=x^{4} - 3 x^{3} - 13 x^{2} - 20 x - 8.\end{aligned}$
\\\addlinespace[5pt]
$y^3+y=x^4+x-6$\par\smallskip
$d=3,\quad\kappa=4$\par & 
$\begin{aligned} H&=3 x^{2} - 3 x + y^{2} + y \left(2 x + 1\right) + 5,\\ U&=2 x^{4} - 10 x^{3} + 11 x^{2} - 19 x - 11,\\ V&=4 x^{3} - 3 x^{2} + 9 x - 15.\end{aligned}$
\\\addlinespace[5pt]
$y^3+y=x^4+2x+4$\par\smallskip
$d=-2,\quad\kappa=1$\par & 
$\begin{aligned} H&=- 2 x + y + 2,\\ U&=x^{2} - 4 x + 8,\\ V&=2 x - 5.\end{aligned}$
\\\addlinespace[5pt]
$y^3+xy=x^4+x+4$\par\smallskip
$d=3,\quad\kappa=1$\par & 
$\begin{aligned} H&=y - 1,\\ U&=x^{2},\\ V&=1.\end{aligned}$
\\\addlinespace[5pt]
$y^3+xy=x^4+x+4$\par\smallskip
$d=6,\quad\kappa=2$\par & 
$\begin{aligned} H&=4 x + 2 y + 3,\\ U&=4 \left(x^{2} + 4 x + 2\right),\\ V&=3.\end{aligned}$
\\\addlinespace[5pt]
$y^3+xy=x^4-x-4$\par\smallskip
$d=3,\quad\kappa=4$\par & 
$\begin{aligned} H&=x^{2} + y^{2} + y \left(x - 2\right) + 4,\\ U&=2 x^{4} - 2 x^{3} - 2 x^{2} - 7 x - 4,\\ V&=2 x^{3} + 5 x - 4.\end{aligned}$
\\\addlinespace[5pt]
$y^3+xy=x^4+x^2+2$\par\smallskip
$d=-2,\quad\kappa=2$\par & 
$\begin{aligned} H&=- 4 x + 2 y + 1,\\ U&=2 \left(2 x^{2} - 8 x + 7\right),\\ V&=3 \left(4 x - 3\right).\end{aligned}$
\\\addlinespace[5pt]
$y^3+xy=x^4+x^2+2$\par\smallskip
$d=-210,\quad\kappa=10$\par & 
$\begin{aligned} H&=20 x + 10 y - 19,\\ U&=10 \left(10 x^{2} + 40 x - 11\right),\\ V&=40 x - 17.\end{aligned}$
\\\addlinespace[5pt]
$y^3+y^2+y=x^4+x+2$\par\smallskip
$d=-21,\quad\kappa=4$\par & 
$\begin{aligned} H&=2 x^{2} - x y + 7 x + y^{2} + 7,\\ U&=2 x^{4} + 5 x^{3} + 31 x^{2} + 57 x + 39,\\ V&=x^{3} + x^{2} - x - 3.\end{aligned}$
\\\addlinespace[5pt]
$y^3+y^2-y=x^4+x+2$\par\smallskip
$d=7,\quad\kappa=32$\par & 
$\begin{aligned} H&=x + 2 y + 2,\\ U&=16 x^{2} + x + 2,\\ V&=3 \left(x + 2\right).\end{aligned}$
\\\addlinespace[5pt]
$y^3-y=x^4-x^2+x+2$\par\smallskip
$d=7,\quad\kappa=4$\par & 
$\begin{aligned} H&=x^{2} + x + y,\\ U&=2 x^{3} + 3 x^{2} + 1,\\ V&=x^{2} - 1.\end{aligned}$
\\\addlinespace[5pt]
$y^3-y=x^4-x^2+x+2$\par\smallskip
$d=7,\quad\kappa=4$\par & 
$\begin{aligned} H&=3 x + 8 y + 2,\\ U&=5 x^{2} - 13 x - 58,\\ V&=17 x^{2} + x - 6.\end{aligned}$
\\\addlinespace[5pt]
$y^3-y=x^4-x^2+x-2$\par\smallskip
$d=-5,\quad\kappa=4$\par & 
$\begin{aligned} H&=x^{2} + x - 2 y - 3,\\ U&=2 x^{3} + 3 x^{2} - 15 x - 7,\\ V&=x^{2} - 5 x - 3.\end{aligned}$
\\\addlinespace[5pt]
\end{longtable}
\endgroup

\medskip
\noindent\textbf{Two further coupled quadratic constructions.}

For $y^3+2y=x^4+x+4$, put
\[
 A=y-4x-5,\quad B=y-1,\quad C=y-x+6,\quad D=y-2,
\]
\[
 R_1=332+580x+262x^2-4x^3,\qquad
 R_2=72-x+2x^2-x^3.
\]
The four norm identities on the curve are
\begin{equation*}
\begin{aligned}
 AT_A&=R_1^2-15B\cdot13^2,& BT_B&=(2x^2)^2-A,\\
 CT_C&=R_2^2+3D\cdot19^2,& DT_D&=(x^2)^2-C.
\end{aligned}
\end{equation*}
Here $F=y^3+2y-x^4-x-4$. If $N_Z$ denotes the displayed
right-hand side corresponding to $Z\in\{A,B,C,D\}$, the cofactor is
$T_Z=(N_Z-q_ZF)/Z$, where
\[
 (q_A,q_B,q_C,q_D)=(-16x^2+1072x+764,-4,-x^2+3x+15,-1).
\]
Each division is exact in $\Z[x,y]$.

For $y^3+y=x^4+2x-4$, put
\[
 C=y-2x+4,\qquad D=2y-x+5,
\]
\[
 R_1=144-120x+24x^2-3x^3,\qquad
 R_2=405-24x+219x^2-24x^3.
\]
The two identities are
\begin{equation*}
 CT_C=R_1^2+3D\cdot16^2,\qquad
 DT_D=R_2^2+3C\cdot332^2.
\end{equation*}
Here $F=y^3+y-x^4-2x+4$ and, with $N_C,N_D$ denoting the two
right-hand sides, the integral cofactors are again
$T_Z=(N_Z-q_ZF)/Z$, now with
\[
 q_C=-9x^2+72x-288,\qquad
 q_D=-576x^2+10440x-46728.
\]
Both divisions are exact in $\Z[x,y]$.

Once the listed identities are established, the proofs are straightforward and are not presented here. We refer the reader to \cite{certificates}
for detailed arguments and Lean formalizations.

\section{Extending the scope of the method}\label{sec:extensions}

All equations in Table \ref{tab:2varH34} are of the form \eqref{eq:shape}, that is, cubic in $y$ and quartic in $x$, and the method was used to prove that the equations have no integer solutions. In Section \ref{sec:length} we demonstrate that the method is applicable to higher-degree equations. In Section \ref{sec:solexists} we demonstrate that it can be used to solve certain two-variable equations whose set of integer solutions is non-empty.

\subsection{Higher-degree examples of small length}\label{sec:length}

The database of equations ordered by $h$ naturally puts lower-degree equations first. To allow higher-degree equations as well, we also consider a database of two-variable equations ordered by length. The \emph{length} of a polynomial $P$ consisting of monomials of degrees $d_1, \dots, d_k$ with coefficients $a_1, \dots, a_k$ is defined by
\begin{equation}\label{eq:ldef}
	l(P) = \sum_{i=1}^k  \log_2(|a_i|) + \sum_{i=1}^k d_i = \log_2 L(P), \quad \text{where} \quad L(P) = \prod_{i=1}^k\left(|a_i|\cdot 2^{d_i}\right).
\end{equation}
Ordering polynomials by length $l(P)$ is equivalent to ordering them by $L(P)$.

Before this work, the shortest open two-variable equation was equation \eqref{eq:2y3pxypx4p1}, now solved in Example \ref{ex:2y3pxypx4p1}. We next generated a complete database of all two-variable equations of length up to $12$, and solved Problem \ref{prob:main} for all of them with two exceptions.
Below we present two examples of higher-degree equations solved by our method, which show that its scope is not limited to cubic-quartic equations \eqref{eq:shape}. 

We start with an equation quartic in both $x$ and $y$, for which the linear choice of $H$ given by \eqref{eq:Hlin} produces a useful identity \eqref{eq:identity} with constant $d$.

\begin{proposition}\label{prop:easyquatric}
	The equation
	\begin{equation}\label{eq:easyquatric}
		2x^4+x-y^4+2y+1=0
	\end{equation}
	has no integer solutions.
\end{proposition}
\begin{proof}
Let $F(x,y)$ be the left-hand side of \eqref{eq:easyquatric}, and define
\[
H(x,y)=2y-3x+2.
\]
Because $H(x,y)$ is linear in $y$, we proceed as in Section \ref{sec:methodlin}. Treating $F$ and $H$ as polynomials in $y$ with parameter $x$ and performing the polynomial division, we obtain
\[
16F(x,y)=H(x,y)T_0(x,y)-R(x),
\]
where
\[
T_0(x,y)=-8y^3-4(3x-2)y^2-2(3x-2)^2y+16-(3x-2)^3
\]
and
\[
R(x)=49x^4-216x^3+216x^2-160x+32.
\]
Because
$$
98 R(x) = -U(x)^2+ 2 V(x)^2 , \quad \quad U(x) = 100 x - 8, \quad V(x) = 49 x^2 - 108 x + 40,
$$
we conclude that every integer solution to $F(x,y)=0$ satisfies
$$
H(x,y) T(x,y) = U(x)^2 - 2V(x)^2,
$$ 
where $T(x,y)=-98 T_0(x,y)$. This is exactly an identity of the form \eqref{eq:identity}. Because $U(x)=100x-8$ cannot be $0$ for integer $x$, we have $U(x)^2 - 2V(x)^2 \neq 0$. Solving the original equation \eqref{eq:easyquatric} modulo $16$, we conclude that in all cases $H(x,y)$ is equal to either $10$ or $11$ modulo $16$. Hence, its largest positive odd divisor $H_0$ can be either $3$ or $5$ modulo $8$. Hence, by the $d=-2$ case of Corollary \ref{cor:norm-congruence-cases}, integers $H$, $T$, $U$, $V$ must have a common prime factor $p$ equal to either $3$ or $5$ modulo $8$. 

On the other hand, direct polynomial division gives the B\'ezout identity
\[
(2602-1225x)U(x)+2500 V(x)=79184=2^4\cdot 7^2\cdot 101.
\]
Hence the only possible common prime factors of $U$ and $V$ are $2$, $7$ and $101$. Because $p$ is equal to either $3$ or $5$ modulo $8$, we must have $p=101$, while $101\mid U(x)$ implies that $x\equiv -8\pmod{101}$. Then $101\mid H(x,y)$ implies that $y\equiv -13\pmod{101}$. But in this case $T(x,y)\equiv 72\pmod{101}$, contradicting $p\mid T(x,y)$. Thus, equation \eqref{eq:easyquatric} has no integer solutions.
\end{proof}

In the next example, the equation is quintic in $x$ and quartic in $y$, and the resulting identity \eqref{eq:identity} has nonconstant $d$.	

\begin{proposition}\label{prop:quintic}
	The equation
	\begin{equation}\label{eq:quintic}
	2x^5+xy+y^4+1=0
	\end{equation}
	has no integer solutions.
\end{proposition}
\begin{proof}
	Let $F(x,y)=2x^5+xy+y^4+1$ be the left-hand side of \eqref{eq:quintic}, and define
	\[
	H(x,y)=y^2-2x.
	\]
	This polynomial is quadratic in $y$, but it is linear in $x$, and we can proceed as in Section \ref{sec:methodlin} with variables $x$ and $y$ interchanged. Treating $F$ and $H$ as polynomials in $x$ with parameter $y$ and performing the polynomial division, with denominators cleared, we obtain
	\[
	16F(x,y)=-H(x,y)T(x,y)+P(y),
	\]
	where
	\[
	T(x,y)=\sum_{j=0}^{4}(2x)^j y^{8-2j}+8y
	\]
	and
	\[
	P(y)=y^{10}+16y^4+8y^3+16.
	\]
	Since
	\[
	P(y)=U(y)^2+2yV(y)^2, \qquad U(y)=y^5-4, \qquad V(y)=2y(y+1),
	\]
	we 
	conclude that every integer solution to $F(x,y)=0$ satisfies
	\begin{equation}\label{eq:quintic-cofactor}
	H(x,y)T(x,y)=U(y)^2+2yV(y)^2.
	\end{equation}
	This is exactly an identity of the form \eqref{eq:identity}, with $d=2y$.
	
	It is easy to see that \eqref{eq:quintic} implies $x<0$. Hence, $H(x,y)>0$. Also, solving the original equation \eqref{eq:quintic} modulo $8$, we conclude that $y$ is odd and $H(x,y)$ is equal to $5$ modulo $8$.
	
	Further, the right-hand side in \eqref{eq:quintic-cofactor} cannot be $0$. Indeed, if
	\[
	U^2+2yV^2=0,
	\]
	then either $U=V=0$, which is impossible since $U=y^5-4\neq 0$ for integer $y$, or
	\[
	-2y=\left(\frac{U}{V}\right)^2.
	\]
	But $y$ is odd, so $-2y$ is equal to $2$ modulo $4$, and therefore cannot be a square of a rational number.
	
	Further, for every prime $p\mid y$, the original equation \eqref{eq:quintic} 
	gives
	$
	2x^5\equiv-1\pmod p.
	$
	In particular, $p \nmid x$, hence $p \nmid H$. Because $H$ is odd, this implies that
	$$
		\gcd(H,2y)=1.
	$$
	Further, 
	\[
	H\equiv-2x\equiv(2x^3)^2\not\equiv0\pmod p.
	\]
	It follows that
	$
	J(H,|y|)=1.
	$
	Since $H\equiv5\pmod8$, in particular $H\equiv1\pmod4$, quadratic reciprocity then gives
	$
	J(|y|,H)=1.
	$
	Because $J(-1,H)=1$, this implies that $J(y,H)=1$. But $H\equiv5\pmod8$ implies that
	$
	J(2,H)=-1.
	$
	Therefore
	\[
	J(-2y,H) =J(-1,H)J(2,H)J(y,H) =-1.
	\]
	
	Hence, Proposition \ref{prop:norm-common-divisor}, applied to \eqref{eq:quintic-cofactor} with $d=2y$ and $H_0=H$, implies that there exists an odd prime $p$ that divides $H,T,U,V$. Since $\gcd(H,2y)=1$, we have $p\nmid2y$, and $p\mid V=2y(y+1)$ therefore
	implies
	\[
	y\equiv-1\pmod p.
	\]
	But then
	\[
	U(y)=y^5-4 \equiv (-1)^5 - 4 = -5 \pmod p,
	\]
	hence $p \mid U$ is possible only if $p=5$. Since $5\mid H$ and $y\equiv-1\pmod5$, we obtain
	$
	x\equiv3\pmod5.
	$
	Substituting these congruences into the formula for $T$ gives
	$
	T(x,y)\equiv2\pmod 5,
	$
	which contradicts $p \mid T$. Hence the equation \eqref{eq:quintic} has no integer solutions.
\end{proof}

\subsection{Equations with integer solutions}\label{sec:solexists}

So far, we have used our method to prove that an equation has no integer solution. In this subsection, we demonstrate that it can also be used to certify that an equation with some obvious integer solutions has no other integer solutions. Specifically, we determine the complete set of integer solutions to the equations $x^4+y^3+x-2y=0$ and $x^4+y^3+x-2y+1=0$, which were left open in \cite[Section 3.5.1]{grechuk2024polynomial}.

\begin{proposition}\label{prop:solexists}
	The only integer solutions to the equation
	\begin{equation}\label{eq:solexists-0}
		x^4+y^3+x-2y=0
	\end{equation}
	are $(x,y)=(-1,0)$ and $(x,y)=(0,0)$, while the only integer solutions to the equation
	\begin{equation}\label{eq:solexists-1}
		x^4+y^3+x-2y+1=0
	\end{equation}
	are $(x,y)=(-1,1)$ and $(x,y)=(0,1)$.
\end{proposition}
\begin{proof}
	Let us write both equations as
	\begin{equation}\label{eq:solexists-e}
		x^4+y^3+x-2y+e=0,
	\end{equation}
	where $e\in\{0,1\}$. Let $(x,y)$ be any integer solution to \eqref{eq:solexists-e}, and define
	\[
	H=R-y,\qquad \text{where}\qquad R=x^2+x+e-1.
	\]
	Then the algorithm in Section \ref{sec:methodlin} leads to the polynomial
	identity
	\begin{equation*}
		HT-U^2-3V^2=-4(x^4+y^3+x-2y+e),
	\end{equation*}
	where
	\[
	T=4(y^2+Ry+R^2-2),\qquad
	U=2x^3+3x^2+(3e-2)x-2,\qquad
	V=x(x+e),
	\]
	which can be verified by direct expansion, using $e^2=e$.
	
	Hence, any integer solution of \eqref{eq:solexists-e} satisfies
	\begin{equation}\label{eq:solexists-norm}
		HT=U^2+3V^2.
	\end{equation}
	The right-hand side of \eqref{eq:solexists-norm} is strictly positive:
	if it vanished, then $V=x(x+e)=0$, so $x=0$ or $x=-e$, but in either
	case $U=-2$. Thus $H$ and $T$ are nonzero and have the same sign.
	
	First assume that $H>0$ and $T>0$. Modulo $2$, equation
	\eqref{eq:solexists-e} gives
	$
	y\equiv e\pmod2,
	$
	and it follows that
	$
	H=x^2+x+e-1-y
	$
	is odd. Further, modulo $3$, equation \eqref{eq:solexists-e} gives
	$
	x^2+x-y+e\equiv0\pmod3,
	$
	and therefore
	\[
	H=x^2+x+e-1-y\equiv-1\equiv2\pmod3.
	\]
	Hence, the case $d=3$ of Corollary \ref{cor:norm-congruence-cases},
	applied to \eqref{eq:solexists-norm}, implies the existence of an odd prime
	$p\equiv2\pmod3$ that is a common factor of $H,T,U,V$. Now $p\mid V=x(x+e)$ implies that 
	$x\equiv0$ or $x\equiv-e\pmod p$, and in either case we obtain $U\equiv-2\pmod p$, which is a contradiction.
	
	Hence, we must have $H<0$ and $T<0$. The identity
	$$
		T=(2y+R)^2+3R^2-8
	$$
	implies that $T<0$ is possible only if $|R|\le1$. Since
	\[
	R=x(x+1)+e-1
	\]
	and $x(x+1)$ is a nonnegative even integer, for $e=1$ this forces $x(x+1)=0$, while for $e=0$ we have $x(x+1)\in \{0,2\}$. In either case this forces $x \in \{-2,-1,0,1\}$. Substituting these values in the original equations \eqref{eq:solexists-0} and \eqref{eq:solexists-1} and solving for $y$, we obtain the stated integer solutions.
\end{proof}

\section{Cubic reciprocity graphs}\label{sec:graphs}

In Section \ref{sec:smallest} we solved all equations from Table~\ref{tab:2varH34}, except for the five equations listed in Table~\ref{tab:remaining}. In this section, we solve these remaining equations, thus finishing the proof of Theorem \ref{prop:table-one}. 

\begin{table}[!htbp]
	\begin{center}
		\setlength{\tabcolsep}{6pt}
		\begin{tabular}{ |c|c| }
			\hline
			$h$ & Equation \\
			\hline\hline
			$34$ &  $y^3-y^2+y=x^4+x+2$ \\
			\hline
			$34$ &  $y^3+xy-y=x^4+x+2$ \\
			\hline
			$34$ &  $y^3+xy+y=x^4+4$ \\
			\hline
			$34$ &  $y^3+xy-y=x^4+4$ \\
			\hline
			$34$ &  $y^3+xy-y=x^4-4$ \\
			\hline
		\end{tabular}
		\caption{\label{tab:remaining} Equations from Table~\ref{tab:2varH34} not solved in Section \ref{sec:smallest}.}
	\end{center}
\end{table}

Four equations will be solved by using cubic reciprocity (two of them after a preliminary restriction by quadratic reciprocity), while the fifth requires cubic and quadratic reciprocity applied jointly. In all cases, we introduce auxiliary polynomials whose coefficients may be Eisenstein integers, and use cubic reciprocity to obtain incompatible conditions on their values at a hypothetical integer solution. For all equations considered in this section, the insolubility certificates are too large to be verified manually in their entirety. We therefore provide only outlines of the proofs, relying on Lean-verified certificates \cite{certificates} to establish their correctness.

\subsection{Eisenstein integers and cubic reciprocity}\label{sec:cubic-background}

\paragraph{Arithmetic in the Eisenstein integers.}
Let
\[
 \omega=\frac{-1+\sqrt{-3}}2,\qquad
 \mathcal E=\Z[\omega]=\{a+b\omega:a,b\in\Z\}.
\]
Thus $\omega^3=1$, $\omega\ne1$, and $\omega^2=-1-\omega$.
The last identity shows that sums and products of elements of
$\mathcal E$ again belong to $\mathcal E$. Its elements are called
\emph{Eisenstein integers}. The expression $a+b\omega$ is unique.
The bar denotes complex conjugation: for $\alpha=a+b\omega$,
\[
 \overline\alpha=a+b\omega^2=(a-b)-b\omega.
\]
The \emph{norm} of $\alpha$ is
\begin{equation*}
 \Norm(\alpha)=\alpha\overline\alpha=a^2-ab+b^2.
\end{equation*}
It is a positive integer when $\alpha\ne0$ and satisfies
$\Norm(\alpha\beta)=\Norm(\alpha)\Norm(\beta)$.
For an ordinary integer $n$, regarded as an element of $\mathcal E$,
its norm is $n^2$.

For $\alpha,\beta\in\mathcal E$, the notation $\beta\mid\alpha$
means that $\alpha=\beta\gamma$ for some $\gamma\in\mathcal E$;
$\alpha\equiv\gamma\pmod\beta$ means $\beta\mid\alpha-\gamma$.
An element is a \emph{unit} if its multiplicative inverse also belongs
to $\mathcal E$. The units are
$1,-1,\omega,-\omega,\omega^2,-\omega^2$, precisely the elements
of norm $1$. Two nonzero elements are \emph{associates} if one is
a unit times the other. They have the same divisors, up to units, and
define the same congruences. Elements are \emph{coprime} if their only
common divisors are units.

A \emph{prime element} of $\mathcal E$ is a nonzero nonunit $\pi$
such that $\pi\mid\alpha\beta$ implies $\pi\mid\alpha$ or
$\pi\mid\beta$. We use the standard unique-factorization theorem:
every nonzero Eisenstein integer is a unit times a product of prime
elements, uniquely up to their order and replacement by associates
\cite[Chapter~1]{ireland1990classical}. We also use the B\'ezout
property: if $\alpha,\beta$ are coprime, some $r,s\in\mathcal E$
satisfy $r\alpha+s\beta=1$ \cite[Chapter~1]{ireland1990classical}. In particular, $\alpha$ then
has a multiplicative inverse modulo $\beta$.
For $z\ne0$,
\[
 v_\pi(z)=\max\{e\in\Z_{\ge0}:\pi^e\mid z\}
\]
is well defined. It is the exponent of $\pi$ in the factorization of
$z$, and does not depend on which associate represents $\pi$.
We set $v_\pi(0)=\infty$ when discussing congruence calculations,
but residue-symbol expressions below always have nonzero arguments.
The \emph{$\pi$-unit part} of $z$ is
$z/\pi^{v_\pi(z)}$; this is an Eisenstein integer not divisible by
$\pi$. We remark that ``$\pi$-unit'' need not be a unit of $\mathcal E$.

We also use the standard description of the prime elements
\cite[Chapter~9]{ireland1990classical}. Put $\lambda=1-\omega$.
Then
\[
 \Norm(\lambda)=3,\qquad \lambda^2=-3\omega,
 \qquad 3=-\omega^2\lambda^2.
\]
Thus $\lambda$ is prime, and every prime element dividing $3$
is an associate of $\lambda$. For an ordinary positive prime $p\ne3$:
\begin{itemize}[leftmargin=2em,itemsep=2pt]
\item If $p\equiv2\pmod3$, then $p$ itself is a prime element of
      $\mathcal E$, of norm $p^2$.
\item If $p\equiv1\pmod3$, there is a prime element $\pi$ such that
      $p=\pi\overline\pi$ and $\Norm(\pi)=\Norm(\overline\pi)=p$.
      The factors $\pi$ and $\overline\pi$ are conjugate but not associates.
\end{itemize}

This classification is exhaustive: every prime element of $\mathcal E$
is an associate of $\lambda$, of an ordinary prime
$p\equiv2\pmod3$, or of one of the two factors $\pi,\overline\pi$
in a factorization $p=\pi\overline\pi$ with $p\equiv1\pmod3$. For example, $3+\omega$ is a prime, and it occurs in the factorization 
$
7=(3+\omega)(2-\omega).
$
The factors are complex conjugates, since $\overline{3+\omega}=3+\omega^2=2-\omega$.

\paragraph{Cubic residues and the cubic residue symbol.}
Let $\eta\in\mathcal E$ be nonzero, and suppose that $\alpha$ and
$\eta$ are coprime. We call $\alpha$ a \emph{cubic residue modulo
$\eta$} if there is a $z\in\mathcal E$ satisfying
$z^3\equiv\alpha\pmod\eta$. Otherwise it is a cubic nonresidue.
The variable $z$ here ranges over Eisenstein integers, even when
$\alpha$ or $\eta$ is an ordinary integer.

Let $\pi$ be a prime element with $\Norm(\pi)\ne3$, and put
$q=\Norm(\pi)$. The preceding classification gives $q\equiv1\pmod3$:
either $q=p\equiv1\pmod3$, or $q=p^2$ with $p\equiv2\pmod3$.
In particular, $(q-1)/3$ is an integer. A standard form of Euler's
criterion for Eisenstein integers states that, if $\pi\nmid\alpha$,
there is exactly one $j\in\{0,1,2\}$ for which
\[
 \alpha^{(q-1)/3}\equiv\omega^j\pmod\pi;
\]
see \cite[Proposition~9.3.2]{ireland1990classical}.
This existence and uniqueness justify the definition
\begin{equation}\label{eq:cubic-residue-symbol}
 \left(\frac\alpha\pi\right)_3=
 \begin{cases}
  0,&\pi\mid\alpha,\\
  \omega^j,&\pi\nmid\alpha\text{ and }
       \alpha^{(\Norm(\pi)-1)/3}\equiv\omega^j\pmod\pi.
 \end{cases}
\end{equation}
The value is one of the four complex numbers $0,1,\omega,\omega^2$.

We use the following classical properties
\cite[Proposition~9.3.3]{ireland1990classical}:
\begin{align*}
 \pi\nmid\alpha:\quad
 \left(\frac\alpha\pi\right)_3=1
 &\quad\Longleftrightarrow\quad
       \alpha\text{ is a cubic residue modulo }\pi,\\
 \left(\frac{\alpha\gamma}\pi\right)_3
 &=\left(\frac\alpha\pi\right)_3
   \left(\frac\gamma\pi\right)_3,\\
 \alpha\equiv\gamma\pmod\pi
 &\quad\Longrightarrow\quad
   \left(\frac\alpha\pi\right)_3=\left(\frac\gamma\pi\right)_3.
\end{align*}
For example, modulo $\pi=3+\omega$ we have $\omega\equiv-3\equiv4$,
and $2^{(7-1)/3}=4$. Hence $(2/\pi)_3=\omega$, so $2$ is not a
cubic residue modulo this prime element.

For a nonzero $\eta$ coprime to $\lambda$, write
$\eta=u\prod_i\pi_i^{e_i}$, where $u$ is a unit and the $\pi_i$
are pairwise nonassociate primes. Define
\begin{equation}\label{eq:cubic-composite-symbol}
 \left(\frac\alpha\eta\right)_3
   =\prod_i\left(\frac\alpha{\pi_i}\right)_3^{e_i}.
\end{equation}
An empty product is $1$, so a unit denominator gives $1$.
This definition is independent of the factorization choices: associates
have the same norm and give the same congruences in
\eqref{eq:cubic-residue-symbol}. The symbol is multiplicative in both
arguments. Multiplying either argument by a nonzero cube does not
change the symbol when the resulting arguments are still coprime
and the denominator is coprime to $\lambda$.
A cubic residue modulo $\eta$ has symbol $1$;
for composite $\eta$ the converse can fail, because nontrivial factors
in \eqref{eq:cubic-composite-symbol} can cancel.

\paragraph{Reciprocity and the supplementary laws.}
The usual convention calls $A$ \emph{primary} if $A\equiv-1\pmod3$.
For the formulas below it is more convenient to call $A$
\emph{normalized} if $A\equiv1\pmod3$. Every element coprime to
$\lambda$ has exactly one normalized associate: this is the negative
of its unique primary associate
\cite[Corollary~3.8]{bosma2005cubic}.
The change of sign has no effect on the symbols, since
$-1=(-1)^3$ and a unit factor in the denominator is ignored.

If $A=a+b\omega$ is normalized, the congruence $A\equiv1\pmod3$ implies that 
$a\equiv1\pmod3$ and $b\equiv0\pmod3$. Write
$a=1+3r$, $b=3s$, with $r,s\in\Z$. Then
\[
 \Norm(A)=1+3(2r-s)+9(r^2-rs+s^2).
\]
Define
\begin{equation*}
 M(A)=\frac{\Norm(A)-1}{3}\pmod3,\qquad
 L(A)=\frac{a-1}{3}\pmod3.
\end{equation*}
Equivalently,
$M(A)=2r-s\pmod3$ and $L(A)=r\pmod3$.

\begin{theorem}[Cubic reciprocity and supplementary laws]\label{thm:cubic-classical}
For coprime normalized Eisenstein integers $A,B$,
\begin{equation*}
 \left(\frac AB\right)_3=\left(\frac BA\right)_3.
\end{equation*}
For every normalized $A$,
\begin{equation*}
 \left(\frac\omega A\right)_3=\omega^{M(A)},\qquad
 \left(\frac\lambda A\right)_3=\omega^{L(A)}.
\end{equation*}
If $2\nmid A$, there is a unique $j(A)\in\{0,1,2\}$ with
$A\equiv\omega^{j(A)}\pmod2$, and
\begin{equation*}
 \left(\frac2A\right)_3=\left(\frac A2\right)_3
      =\omega^{j(A)}.
\end{equation*}
\end{theorem}
These are classical results; see, e.g., 
\cite[Theorems~3.9--3.10 and Lemma~3.11]{bosma2005cubic}, where reciprocity is stated for arbitrary primary elements and the supplementary laws for primes. The supplementary laws for composite $A$ follow by multiplicativity, since $M$ and $L$ are additive on normalized elements; see also \cite[Chapter~9, Section~3 and Exercises~17--20]{ireland1990classical}. The law for $2$ follows from reciprocity applied to the normalized prime $-2$ and from Euler's criterion $(A/2)_3\equiv A\pmod2$; for primes it is in \cite[Chapter~9, Section~6]{ireland1990classical}.

\paragraph{An additive form of the reciprocity law.}
We now record these laws in a form that will let us add the contributions
from many pairs of auxiliary values. Write
$\mathbb F_3=\Z/3\Z$ for arithmetic with the residues $0,1,2$.
Every $z\in\mathcal E\setminus\{0\}$ can be uniquely represented in the form
\begin{equation}\label{eq:cubic-normalization}
 z=\varepsilon\omega^{k(z)}\lambda^{v_\lambda(z)}B(z),
 \quad \varepsilon\in\{1,-1\},\quad k(z)\in\{0,1,2\},
 \quad B(z)\equiv1\pmod3.
\end{equation}
Put $m(z)=M(B(z))$ and $\ell(z)=L(B(z))$.
For a prime $\pi$ with $\Norm(\pi)\ne3$ and $\pi\nmid u$, define
$\chi_\pi(u)\in\mathbb F_3$ by
$(u/\pi)_3=\omega^{\chi_\pi(u)}$.
For nonzero $a,b\in\mathcal E$, set
\begin{equation}\label{eq:cubic-beta}
\begin{split}
 \beta_\pi(a,b)&:=v_\pi(a)\chi_\pi\bigl(b/\pi^{v_\pi(b)}\bigr)
                -v_\pi(b)\chi_\pi\bigl(a/\pi^{v_\pi(a)}\bigr),
                       \qquad \Norm(\pi)\ne3,\\
 \beta_\lambda(a,b)&:=k(a)m(b)-m(a)k(b)
       +v_\lambda(a)\ell(b)-\ell(a)v_\lambda(b).
\end{split}
\end{equation}
The right-hand sides are evaluated modulo $3$. Each $\beta_\pi$
is additive with respect to multiplication of either argument; for
example $\beta_\pi(aa',b)=\beta_\pi(a,b)+\beta_\pi(a',b)$.
It also satisfies $\beta_\pi(b,a)=-\beta_\pi(a,b)$ and
$\beta_\pi(a,a)=0$. For $\beta_\lambda$ these properties hold because $z\mapsto(v_\lambda(z),k(z),m(z),\ell(z))$ is a homomorphism from $\mathcal E\setminus\{0\}$ to $\Z\times\mathbb F_3^3$. For $\Norm(\pi)\ne3$, changing the representative of the prime
to an associate does not change $\beta_\pi$: the changes in the two unit
parts cancel in its defining expression.

Cubic reciprocity and the supplementary laws, expressed in this
notation, give
\begin{equation}\label{eq:cubic-product}
 \beta_\lambda(a,b)
       +\sum_{\Norm(\pi)\ne3}\beta_\pi(a,b)=0
       \qquad\text{in }\mathbb F_3.
\end{equation}
The sum takes one representative of each class of associate prime
elements. Apart from the separately written $\lambda$ term, only
prime divisors of $ab$ can contribute. To interpret the formula,
factor both arguments into prime elements. Reciprocity cancels the
terms involving two normalized primes; the terms involving $\omega$
and $\lambda$ are exactly the four terms defining $\beta_\lambda$.
Thus \eqref{eq:cubic-product} is a way to rewrite the classical
laws above, rather than an additional reciprocity theorem.

\subsection{Graphs and their polynomial checks}\label{sec:cubic-graphs}

Fix a nonconstant polynomial $f\in\Z[x,y]$. We will build a certificate that shows that the equation $f(x,y)=0$ has no integer solutions. As a first part of this certificate, we fix an integer $m\ge1$, nonzero polynomials $H_i\in\mathcal E[x,y]$ and nonzero constants $c_i\in\mathcal E$ for $0\le i<m$. Further, we choose weights $e_{ij}\in\{0,1,2\}$ for $0\le i<j<m$. For each equation, these data will be chosen such that all the numbered assumptions stated below are satisfied. 

It is convenient to represent this as graph data: vertex $i$ carries $H_i$, and an edge joins $i<j$ precisely when
$e_{ij}\ne0$. Unless stated otherwise, vertex indices range from $0$ to $m-1$ and edge indices satisfy $0\le i<j<m$ and $e_{ij}\ne0$. Set $\rho_{ij}=e_{ij}$, $\rho_{ji}=-e_{ij}$ and $\rho_{ii}=0$ in $\mathbb F_3$. Write $\widetilde\rho_{ij}\in\{0,1,2\}$ for
the integer representative of $\rho_{ij}$. Put $d=\deg f$, where
degree means total degree, and fix an integer
$h\ge\max_i\deg H_i$.

Our first assumption is the following one.

\begin{assumption}[Nonzero auxiliary values]\label{ass:cubic-nonzero}
	For every $0\le i<m$, the system
	\begin{equation*}
		f(x,y)=H_i(x,y)=0
	\end{equation*}
	has no solutions in integers $(x,y)$.
\end{assumption}

Now suppose that $(x,y)\in\Z^2$ satisfies $f(x,y)=0$. Then, by Assumption \ref{ass:cubic-nonzero}, every $H_i = H_i(x,y)$ is nonzero. For $\pi=\lambda$ and for each prime element $\pi \in \mathcal E$ with $\Norm(\pi)\ne3$, and each ordinary prime $p$, define
\begin{equation}\label{eq:cubic-graphsum}
 B_\pi:=\sum_{i<j}e_{ij}\beta_\pi(H_i,H_j)
                 +\sum_i\beta_\pi(H_i,c_i),\qquad
 B_p:=\sum_{\pi\mid p}B_\pi.
\end{equation}

Here, in the sum 
$
\sum_{\pi\mid p}B_\pi,
$
we include one prime element from each collection of associates 
dividing $p$, without recording its multiplicity. Explicitly:
\begin{itemize}
	\item For $p=3$, the only index is $\lambda$, so $B_3=B_\lambda$.
	Although $\lambda$ occurs twice in
	$3=-\omega^2\lambda^2$, it contributes only one summand.
	\item For $p\equiv2\pmod3$, the only index is the Eisenstein
	prime element $p$ itself.
	\item For $p\equiv1\pmod3$, choose any $\pi$ with
	$p=\pi\overline\pi$. The two indices are $\pi$ and
	$\overline\pi$, so $B_p=B_\pi+B_{\overline\pi}$.
\end{itemize}
Replacing a prime element by its associate does not change
$B_\pi$, so these choices do not affect $B_p$.

We call $B_p$ the contribution at the prime $p$.
Summing \eqref{eq:cubic-product} for the pairs in
\eqref{eq:cubic-graphsum} yields
\[
\sum_{\substack{p\ge2\\p\text{ prime}}} B_p=0,
\]
where $p$ ranges over the ordinary prime integers $2,3,5,7,\ldots$.
For each fixed integer pair $(x,y)$ with all $H_i(x,y)\ne0$,
this sum has only finitely many nonzero terms: if $p\ne3$
and $p$ divides none of the integers $\Norm(H_i(x,y))$ and $\Norm(c_i)$,
then $B_p=0$.
Our aim is to show that, at an integer solution of the equation,
the possible contributions cannot add to zero. This contradiction will show that the equation has no integer solutions. 

\paragraph{The polynomial data and their assumptions.}
For polynomial calculations we use
\[
 K=\Q(\omega)=\{r+s\omega:r,s\in\Q\}.
\]
Conjugation and norm extend to $K$ by
$\overline{r+s\omega}=r+s\omega^2$ and
$\Norm(r+s\omega)=r^2-rs+s^2$. A nonzero element $\alpha\in K$
has positive rational norm and inverse
$\overline\alpha/\Norm(\alpha)$. The notation $K[t]$ means
polynomials in one variable $t$ with coefficients in $K$.
Every Eisenstein integer has a residue modulo a prime element $\pi$,
but a fraction in $K$ need not. A fraction $\alpha/\beta$, with
$\alpha,\beta\in\mathcal E$, can be reduced by multiplying the residue
of $\alpha$ by the inverse of the residue of $\beta$ only when
$\pi\nmid\beta$ (after cancelling common factors). For example,
$1/2\equiv3\pmod5$, whereas $1/2$ has no residue modulo $2$.
Thus, when reducing polynomial identities with coefficients in $K$,
we must check that the coefficient denominators are invertible.

\begin{assumption}[Parametrizations]\label{ass:cubic-charts}
For each $i$, polynomials $X_i,Y_i,\Delta_i\in K[t]$ are specified,
with $\Delta_i\ne0$, such that
\begin{equation*}
 (x,y)=\sigma_i(t):=
 \left(\frac{X_i(t)}{\Delta_i(t)},
       \frac{Y_i(t)}{\Delta_i(t)}\right),\qquad
 H_i\left(\frac{X_i(t)}{\Delta_i(t)},
          \frac{Y_i(t)}{\Delta_i(t)}\right)=0
\end{equation*}
identically in $t$, and at least one coordinate of $\sigma_i$ is nonconstant.
\end{assumption}

Such a pair of formulas is called a \emph{rational parametrization}, or a \emph{chart}. They are defined where $\Delta_i(t)\ne0$.
An arbitrary polynomial curve need not admit a rational parametrization.
Assumption~\ref{ass:cubic-charts} therefore restricts the choice of polynomials $H_i$. 

Define
\begin{equation}\label{eq:cubic-FG}
 F_i(t)=\Delta_i(t)^d f(\sigma_i(t)),\qquad
 G_{ij}(t)=\Delta_i(t)^h H_j(\sigma_i(t)).
\end{equation}
Here, by $f(\sigma_i(t))$ we mean $f(x,y)$ for $x=X_i(t)/\Delta_i(t)$ and $y=Y_i(t)/\Delta_i(t)$. The choices of $d,h$ clear the denominators involving $t$, so $F_i,G_{ij}\in K[t]$.

\begin{assumption}[Cube identities]\label{ass:cubic-identities}
For every $i$, $F_i\ne0$, and polynomials $W_i,Q_i\in K[t]$
and an integer $n_i\ge0$ are specified for which
\begin{equation}\label{eq:cubic-vertex}
 \Delta_i^{n_i}\left(
  \Delta_i^{h\sum_{j=0}^{m-1}\widetilde\rho_{ij}}W_i^3
  -c_i\prod_{j=0}^{m-1}G_{ij}^{\widetilde\rho_{ij}}
 \right)=Q_iF_i
\end{equation}
as an equality in $K[t]$. Factors with exponent zero are omitted;
an empty product is $1$.
\end{assumption}

The identity is checked by exact polynomial division. At a parameter value $t_0$ modulo a prime
element $\pi$, provided all coefficient denominators are invertible modulo $\pi$ and
$F_i(t_0)=0$, $\Delta_i(t_0)\ne0$, it gives
\[
 W_i(t_0)^3\equiv
 c_i\prod_{j=0}^{m-1}
 H_j(\sigma_i(t_0))^{\widetilde\rho_{ij}}\pmod\pi.
\]
Thus the product on the right is a cube modulo $\pi$.

For nonzero polynomials $A,B$ of degrees $r,s$, their
\emph{resultant} $\Res_t(A,B)$ is the determinant of the Sylvester
matrix, whose rows are the coefficient lists of
$t^{s-1}A,\ldots,A,t^{r-1}B,\ldots,B$ in descending powers of $t$.
 If $r=0$ there are no rows from $B$, and if $s=0$ there are no rows from $A$; for $r=s=0$ the empty determinant is $1$.
The resultant has a B\'ezout identity
$\Res_t(A,B)=UA+VB$ with $U,V\in K[t]$; see
\cite[Chapter~3, Section~6]{cox2015ideals}.

\begin{assumption}[Edge resultants]\label{ass:cubic-resultants}
For every edge $i<j$, $G_{ij}\ne0$ and
\begin{equation}\label{eq:cubic-resultant}
 R_{ij}:=\Res_t(F_i,G_{ij})\ne0.
\end{equation}

\end{assumption}

Assumption \ref{ass:cubic-resultants} implies the existence of polynomials $U_{ij},V_{ij}\in K[t]$ satisfying
\begin{equation}\label{eq:cubic-bezout}
	R_{ij}=U_{ij}F_i+V_{ij}G_{ij} \neq 0.
\end{equation}
These polynomials can be computed by the Euclidean algorithm. 
When $\deg F_i+\deg G_{ij}>0$, the coefficients of $U_{ij},V_{ij}$ can also be
obtained from the adjugate of the Sylvester matrix, that is, the
transpose of its matrix of signed minors. Whenever the denominators of these coefficients are invertible modulo $\pi$, a common
root of $F_i,G_{ij}$ modulo $\pi$ forces $R_{ij}=0$ modulo $\pi$.

\paragraph{Exceptional primes and coverage.}
For $r+s\omega\in K$, its \emph{coefficient denominators} are
the positive denominators of $r,s\in\Q$ in lowest terms.
Coefficientwise reduction modulo $\pi$ multiplies each numerator by the inverse
of its denominator; this is valid when those denominators are
not divisible by $\pi$. For a finite set $S$ of ordinary positive primes, a nonzero integer
is \emph{supported on $S$} if all its positive prime divisors belong to $S$. For a nonzero
rational number this means that both its numerator and denominator
in lowest terms are supported on $S$; for a nonzero element of
$K$ it means that its norm is supported on $S$.

\begin{samepage}
\begin{assumption}[Exceptional primes]\label{ass:cubic-primes}
A specified finite set $S$ of ordinary positive primes satisfies:
\begin{enumerate}[label=\textup{(\roman*)},nosep]
\item $3\in S$;
\item every $\Norm(c_i)$ and every $\Norm(R_{ij})$ is supported on $S$;
\item the coefficient denominators of $X_i,Y_i,\Delta_i,F_i,W_i,Q_i$
for every vertex $i$, of $G_{ij}$ for $0\le i,j<m$, and of
$U_{ij},V_{ij}$ for every edge $i<j$, are all supported on $S$.
\end{enumerate}
\end{assumption}
\end{samepage}

For $p\notin S$ and $\pi\mid p$, these conditions ensure that all
coefficient denominators in the specified identities are invertible modulo $\pi$
and every $c_i,R_{ij}$ is nonzero modulo $\pi$.

\begin{assumption}[Coverage modulo primes]\label{ass:cubic-coverage}
For every ordinary prime $p\notin S$, every prime element $\pi\mid p$,
every $i$, and every pair of residue classes $x_0,y_0$ of
Eisenstein integers modulo $\pi$, the implication
\begin{equation*}
\begin{gathered}
 f(x_0,y_0)\equiv H_i(x_0,y_0)\equiv0\pmod\pi\quad\Longrightarrow\\
 \exists\,t_0\pmod\pi:\quad
 \Delta_i(t_0)\not\equiv0,\qquad
 x_0\equiv X_i(t_0)/\Delta_i(t_0),\quad
 y_0\equiv Y_i(t_0)/\Delta_i(t_0)\pmod\pi
\end{gathered}
\end{equation*}
holds. The witness $t_0$ is also a residue class of Eisenstein integers.
\end{assumption}

Two elementary checks are useful. 
\begin{itemize}
	\item For a line $ax+by+c=0$ with
	$b\ne0$, the chart $x=t$, $y=-(at+c)/b$ covers every point
	when $\pi\nmid b$; if $b=0$ and $a\ne0$, use $x=-c/a$, $y=t$, which covers every point when $\pi\nmid a$.
	\item For a conic
	\[
	H=(y-ax-b)(y-cx-d)-q,\qquad (c-a)q\ne0,
	\]
	put $t=y-ax-b$. Then
	\[
	t\bigl(t+(a-c)x+b-d\bigr)=q,
	\quad
	x=\frac{t^2+(b-d)t-q}{(c-a)t},\quad y=t+ax+b.
	\]
	If $\pi\nmid(c-a)q$, the equation forces $t\ne0$, so these
	formulas cover every zero of $H$ modulo $\pi$. Including the
	coefficient denominators and prime divisors of
	$\Norm((c-a)q)$ in $S$ therefore verifies
	Assumption~\ref{ass:cubic-coverage} for this chart.
\end{itemize}

\paragraph{The finite obstruction.}

The final piece of our certificate is a collection of sets $C_p\subseteq\mathbb F_3$ for each $p\in S$, which satisfies the following two assumptions.

\begin{assumption}[Local bounds]\label{ass:cubic-local}
For every $p\in S$ and every $(x,y)\in\Z^2$ with $f(x,y)=0$,
\begin{equation*}
 B_p(x,y)\in C_p,
\end{equation*}
where $B_p(x,y)$ is the expression in \eqref{eq:cubic-graphsum}
evaluated at that pair.
\end{assumption}

\begin{assumption}[Incompatibility]\label{ass:cubic-exclusion}
With addition in $\mathbb F_3$,
\begin{equation*}
 0\notin\sum_{p\in S}C_p,\quad \text{where} \quad
 \sum_{p\in S}C_p:=
 \left\{\sum_{p\in S}b_p:
                b_p\in C_p\text{ for every }p\in S\right\}.
\end{equation*}
\end{assumption}

The certificate is now complete, and we are ready to formulate the main result on which the method is based.

\begin{theorem}\label{prop:cubic-certificate}
Let $f\in\Z[x,y]$ be nonconstant, $m\ge1$, and let
$H_0,\ldots,H_{m-1}\in\mathcal E[x,y]$ and
$c_0,\ldots,c_{m-1}\in\mathcal E$ be nonzero. Choose
$e_{ij}\in\{0,1,2\}$ for $0\le i<j<m$, put
$\rho_{ij}=e_{ij}$, $\rho_{ji}=-e_{ij}$, $\rho_{ii}=0$ in
$\mathbb F_3$, and let $\widetilde\rho_{ij}\in\{0,1,2\}$ represent
$\rho_{ij}$. Set $d=\deg f$ and choose an integer
$h\ge\max_i\deg H_i$.

Suppose that polynomials $X_i,Y_i,\Delta_i,W_i,Q_i\in K[t]$
and integers $n_i\ge0$ are specified for $0\le i<m$, together with
$U_{ij},V_{ij}\in K[t]$ satisfying \eqref{eq:cubic-bezout} for every $0\le i<j<m$ with $e_{ij}\ne0$,
and a finite set $S$ of ordinary positive primes. Suppose these data satisfy
Assumptions~\ref{ass:cubic-nonzero}--\ref{ass:cubic-coverage},
with $F_i,G_{ij}$ defined by \eqref{eq:cubic-FG} and
$R_{ij}$ by \eqref{eq:cubic-resultant}. Then every integer solution
of $f=0$ has
\[
 B_p(x,y)=0\qquad\text{for every ordinary prime }p\notin S,
\]
where $B_p$ is defined by \eqref{eq:cubic-graphsum}.
If sets $C_p\subseteq\mathbb F_3$, one for each $p\in S$,
also satisfy Assumptions~\ref{ass:cubic-local}
and~\ref{ass:cubic-exclusion}, then $f=0$ has no integer solutions.
\end{theorem}
\begin{proof}
Suppose $f(x,y)=0$ with $(x,y)\in\Z^2$. By
Assumption~\ref{ass:cubic-nonzero}, every $H_i=H_i(x,y)$ is nonzero.
Fix $p\notin S$ and $\pi\mid p$. All reductions used here are
valid by Assumption~\ref{ass:cubic-primes}.

If $\pi$ divided both ends $H_i,H_j$ of an edge $i<j$,
Assumption~\ref{ass:cubic-coverage} would supply $t_0$ representing
$(x,y)$ modulo $\pi$, with $\Delta_i(t_0)\ne0$.
Equation~\eqref{eq:cubic-FG} would give
$F_i(t_0)=G_{ij}(t_0)=0$ modulo $\pi$. The B\'ezout identity \eqref{eq:cubic-bezout} would imply
$R_{ij}=0$ modulo $\pi$, contradicting
Assumption~\ref{ass:cubic-primes}. Thus adjacent values $H_i,H_j$ cannot
both be divisible by $\pi$.

If $\pi\mid H_i$, every neighbouring value is therefore nonzero modulo $\pi$, and so is $c_i$ by Assumption~\ref{ass:cubic-primes}. Apply Assumption~\ref{ass:cubic-coverage}
at vertex $i$ and evaluate the identity in
Assumption~\ref{ass:cubic-identities}. It makes
$c_i\prod_jH_j^{\widetilde\rho_{ij}}$ a nonzero cube modulo $\pi$.
The cubic-residue criterion and \eqref{eq:cubic-beta} show that
all terms of $B_\pi$ involving this vertex have sum
\[
 v_\pi(H_i)\,
 \chi_\pi\left(c_i\prod_jH_j^{\widetilde\rho_{ij}}\right)=0.
\]
No edge is counted twice, since its two ends cannot both be
divisible by $\pi$. In every remaining term both arguments are
indivisible by $\pi$, so that term is zero by
\eqref{eq:cubic-beta}. Hence $B_\pi=0$ and $B_p=0$.

If Assumption~\ref{ass:cubic-local} holds, reciprocity gives
\[
 0=\sum_{p\in S}B_p(x,y)\in\sum_{p\in S}C_p,
\]
contradicting Assumption~\ref{ass:cubic-exclusion}.
\end{proof}

\subsection{Finite computations and reproducibility}\label{sec:cubic-local}

To apply Theorem \ref{prop:cubic-certificate}, we need to determine sets $C_p$ for the finitely many $p\in S$, satisfying Assumptions~\ref{ass:cubic-local} and~\ref{ass:cubic-exclusion}. We will do this by partitioning all possible integer solutions into congruence classes, and allowing every residue-symbol value consistent with each class. A class is called \emph{terminal} when its possible graph contributions have been determined and no further refinement is needed.

\paragraph{Refining the congruence classes.}
Fix a prime $p\in S$. Start with all pairs $(a,b)$ modulo $p$ satisfying $f(a,b)\equiv0$.
For $k\ge1$, a \emph{lift} of $(a,b)$ modulo $p^k$ to modulus
$p^{k+1}$ has the form
$(a+p^ks,b+p^kt)$, where $0\le s,t<p$.
Taylor expansion shows that it satisfies $f\equiv0\pmod{p^{k+1}}$
exactly when
\begin{equation}\label{eq:cubic-lifting}
 f_x(a,b)s+f_y(a,b)t\equiv-f(a,b)/p^k\pmod p.
\end{equation}
Here $f_x,f_y$ are the formal partial derivatives. The quotient
$f(a,b)/p^k$ is an integer because the starting class satisfies
$f(a,b)\equiv0\pmod{p^k}$. This rule remains valid when one or
both derivatives vanish modulo $p$; such classes are not discarded.

\paragraph{The information needed when $p\ne3$.}
For each $\pi\mid p$, formula \eqref{eq:cubic-beta} needs only
$v_\pi(H_i)\pmod3$ and the character of the $\pi$-unit part of
$H_i$. If $H_i(a,b)$ has exact valuation $e<k$, its valuation
is fixed throughout the class, and its unit part is known modulo
$\pi^{k-e}$, hence modulo $\pi$. If instead
$H_i(a,b)\equiv0\pmod{\pi^k}$, its exact valuation is not yet
known. One may safely allow all three valuation residues and all
three unit characters, or refine the class further.

When $p\equiv1\pmod3$, both prime factors of $p$ must be included.
Conjugation satisfies
\[
 \left(\frac{\overline u}{\overline\pi}\right)_3
     =\overline{\left(\frac u\pi\right)_3},
 \qquad
 \chi_{\overline\pi}(\overline u)=-\chi_\pi(u)
\]
\cite[Lemma~3.5]{bosma2005cubic}.
For an ordinary integer value $H_i$, the two valuations are equal
and the two unit characters are opposite. When $p\equiv2\pmod3$,
the $\pi$-unit part of an ordinary integer value has character $0$. These restrictions must
be retained when they are used in a terminal calculation.

\paragraph{The information needed at $3$.}
We calculate with pairs of ordinary integers. In particular,
$\lambda\mid a+b\omega$ exactly when $a+b\equiv0\pmod3$,
and in this case
\begin{equation*}
 \frac{a+b\omega}{\lambda}
       =\frac{2a-b}{3}+\frac{a+b}{3}\omega.
\end{equation*}
Both coefficients are integers: divisibility of $a+b$ by $3$
also implies divisibility of $2a-b=3a-(a+b)$ by $3$.
Repeated division gives $v_\lambda(z)$ and the unit part needed
in \eqref{eq:cubic-normalization}. The four entries
$(v_\lambda(z),k(z),m(z),\ell(z))$, reduced modulo $3$, determine
every occurrence of $z$ in $\beta_\lambda$.

Suppose a value $z$ is known modulo $3^k$, and its representative
has exact $\lambda$-valuation $v<2k$. Since $3$ is a unit times
$\lambda^2$, the $\lambda$-unit part of $z$ is determined modulo
$\lambda^{2k-v}$. If $2k-v\ge4$, it is determined modulo $9$,
which is enough to compute its normalized associate and
$k(z),m(z),\ell(z)$. At lower precision, enumerate the $54$ pairs
$a+b\omega$ modulo $9$ with $a+b\not\equiv0\pmod3$; keep those
consistent with the known unit part and normalize each by multiplying
by a unit. If $z$ is entirely zero at the current precision, allow
all $3^4=81$ quadruples, unless additional restrictions are available.

For a nonzero ordinary integer $z$, the useful restrictions are
\begin{equation}\label{eq:cubic-rational}
 k(z)=v_\lambda(z),\qquad m(z)=-\ell(z)
       \qquad\text{in }\mathbb F_3.
\end{equation}
Indeed, extracting powers of $3=-\omega^2\lambda^2$ gives the first
relation. The remaining normalized factor is an ordinary integer
$1+3r$, for which $M=2r=-r=-L\pmod3$, giving the second.

For each class, substitute every allowed collection of coordinates
in \eqref{eq:cubic-graphsum}. An unknown coordinate belonging to a
vertex must have the same value in all its occurrences. Coordinates
of different vertices may be allowed independently: this can enlarge
$C_p$, but cannot exclude the contribution of an actual solution.
Relations such as conjugation must be retained whenever a sharper
set is claimed. If the resulting set is not sufficient for the
contradiction, refine the class by \eqref{eq:cubic-lifting}.
The terminal classes must cover all initial classes and all their
possible lifts. A computation at just one representative does not
establish this covering property.

\paragraph{Recovering cube roots omitted from the graph tables.}
Given $H_i,c_i,e_{ij}$ and the charts, polynomials $W_i$ and $Q_i$ in \eqref{eq:cubic-vertex} need not be explicitly listed, because they can be reconstructed by finite polynomial arithmetic. Remove from $F_i$ every irreducible factor that divides $\Delta_i$,
with its full multiplicity, and call the result $\widehat F_i$.
If $\widehat F_i$ is constant, take $W_i=0$ and choose $n_i$
so that $F_i$ divides $\Delta_i^{n_i}$; then compute $Q_i$ by
\eqref{eq:cubic-vertex}. Otherwise proceed as follows.
Because $\Delta_i$ and $\widehat F_i$ are coprime, the Euclidean
algorithm gives an inverse of $\Delta_i$ modulo $\widehat F_i$
\cite[Chapter~1, Section~5]{cox2015ideals}.
Thus the expression
\begin{equation}\label{eq:cubic-recover}
 A_i\equiv c_i\prod_j(G_{ij}/\Delta_i^h)^{\widetilde\rho_{ij}}
                          \pmod{\widehat F_i}
\end{equation}
is calculated using polynomial remainders in $K[t]$.

Factor $\widehat F_i$ over $K$. Here an irreducible polynomial is
a nonconstant polynomial that cannot be written as a product of
two polynomials of positive degree with coefficients in $K$.
For each irreducible factor $\phi$,
do arithmetic with polynomials of degree below $\deg\phi$, reducing
products modulo $\phi$. Every nonzero remainder has an inverse by
the Euclidean algorithm. Factor $Z^3-A_i$ in this arithmetic and
choose a linear factor, hence a root $W$ modulo $\phi$.
If $\phi$ occurs more than once, lift a root modulo $\phi^r$ to one modulo $\phi^{r+1}$ by
solving
\[
 3W^2s\equiv(A_i-W^3)/\phi^r\pmod\phi,
 \qquad W\longmapsto W+\phi^rs.
\]
The division on the right is exact because $W$ is already a root
modulo $\phi^r$. Its coefficient $3W^2$ is invertible modulo
$\phi$ provided $\phi$ divides no $G_{ij}$ with $\widetilde\rho_{ij}\ne0$, so that $A_i\not\equiv0\pmod\phi$; for $j>i$ this follows from $R_{ij}\ne0$, and for $j<i$ it is checked directly, for example by $\Res_t(F_i,G_{ij})\ne0$.
Combine the roots for the different factors by the polynomial
Chinese remainder theorem (the usual Chinese remainder construction,
with coprime polynomials in place of coprime integers).
The required simultaneous remainder is computed by repeated
applications of the same Euclidean algorithm.

Take the resulting $W_i$ of degree below $\deg\widehat F_i$ and
verify \eqref{eq:cubic-vertex}, increasing $n_i$ if needed to clear
the removed factors. Include all primes in its coefficient
denominators in $S$, or choose different cube roots if the
certificate specifies a fixed $S$. The possible root choices are
finite. This reconstructs polynomial identities. No curve point is lost by removing the factors
of $\Delta_i$: \eqref{eq:cubic-vertex} is only evaluated where $\Delta_i(t_0)\not\equiv0$, and points not reached by a chart are handled when Assumption~\ref{ass:cubic-coverage} is verified.

\subsection{The equation \texorpdfstring{$y^3-y^2+y=x^4+x+2$}{y\textasciicircum 3-y\textasciicircum 2+y=x\textasciicircum 4+x+2}}\label{sec:detailed}

\begin{proposition}\label{prop:detailed}
The equation
\begin{equation}\label{eq:detailed}
 y^3-y^2+y=x^4+x+2
\end{equation}
has no integer solutions.
\end{proposition}
\begin{proof}[Proof outline]
Suppose that $(x,y)$ is an integer solution, and put
$f=y^3-y^2+y-x^4-x-2$. We will construct a graph for which every
prime other than $2,3$ contributes $0$, while $B_2=2$ and $B_3=0$.
Their sum cannot be $0$ in $\mathbb F_3$, contrary to reciprocity.

Use the twenty auxiliary polynomials in
Table~\ref{tab:detailed-vertices}. Its triple $(\kappa_i,r_i,s_i)$
specifies the constant $c_i=\omega^{\kappa_i}2^{r_i}\lambda^{s_i}$.
Table~\ref{tab:detailed-edges} specifies the $58$ edges and their weights.
Set $m=20$, $d=4$, $h=2$ and $S=\{2,3\}$; all unlisted
weights are zero. We verify Assumptions~\ref{ass:cubic-nonzero}--\ref{ass:cubic-exclusion}
for these data.

\begingroup
\small
\setlength{\tabcolsep}{7pt}
\begin{longtable}{@{}rll@{}}
\caption{Vertices for \eqref{eq:detailed}; $c_i=\omega^{\kappa_i}2^{r_i}\lambda^{s_i}$.}\label{tab:detailed-vertices}\\
\toprule
$i$ & $H_i(x,y)$ & $(\kappa_i,r_i,s_i)$\\
\midrule
\endfirsthead
\toprule
$i$ & $H_i(x,y)$ & $(\kappa_i,r_i,s_i)$\\
\midrule
\endhead
\bottomrule
\endfoot
0 & $y+1+\omega$ & $(0,2,0)$\\
1 & $y-1-\omega$ & $(2,2,0)$\\
2 & $x$ & $(1,2,0)$\\
3 & $x+1+\omega$ & $(1,1,0)$\\
4 & $x+1$ & $(0,2,0)$\\
5 & $x+\omega$ & $(1,1,0)$\\
6 & $x-\omega$ & $(0,1,0)$\\
7 & $x-1-\omega$ & $(0,0,0)$\\
8 & $x+y+\omega$ & $(0,2,0)$\\
9 & $x+(y-2)\omega$ & $(1,1,1)$\\
10 & $-x+y-2-(x^2+x+1)\omega$ & $(0,1,0)$\\
11 & $y-1$ & $(1,0,0)$\\
12 & $x-y+1$ & $(1,2,1)$\\
13 & $x+y$ & $(1,0,0)$\\
14 & $y-2-\omega$ & $(1,1,0)$\\
15 & $y-1-x\omega$ & $(1,1,2)$\\
16 & $-x+y+(1-x)\omega$ & $(0,1,2)$\\
17 & $y^2-y+1-xy\omega$ & $(0,0,0)$\\
18 & $xy+x+y^2+(x+y-1)\omega$ & $(2,0,0)$\\
19 & $-x^2+y^2+1-(x^2+xy)\omega$ & $(2,0,0)$\\
\end{longtable}
\endgroup

\begingroup
\small
\setlength{\tabcolsep}{7pt}
\begin{longtable}{@{}rll@{}}
\caption{All nonzero upper-triangular weights. Only indices $j>i$ are listed.}\label{tab:detailed-edges}\\
\toprule
$i$ & $j$ with $e_{ij}=1$ & $j$ with $e_{ij}=2$\\
\midrule
\endfirsthead
\toprule
$i$ & $j$ with $e_{ij}=1$ & $j$ with $e_{ij}=2$\\
\midrule
\endhead
\bottomrule
\endfoot
0 & $8,13,16$ & $17$\\
1 & $4,6,12$ & $2,5,10,13,15,17$\\
2 & --- & $9,14,18,19$\\
3 & $11,12,16,18$ & $9,10$\\
4 & $12$ & $10,14,19$\\
5 & --- & $8,11,12,18,19$\\
6 & $10$ & $8,9,18,19$\\
7 & $11,17$ & ---\\
8 & $9$ & $16,17$\\
9 & --- & $12,14,15$\\
10 & $11$ & $13,17$\\
11 & --- & $16$\\
12 & --- & $15$\\
13 & $15,18,19$ & ---\\
14 & $15$ & ---\\
15 & $17$ & $18$\\
16 & $19$ & ---\\
17 & --- & ---\\
18 & --- & $19$\\
19 & --- & ---\\
\end{longtable}
\endgroup

\emph{Assumption~\ref{ass:cubic-nonzero}: nonzero auxiliary values.} 
Assume that $(x,y)\in {\mathbb Z}^2$ is an integer solution to $f(x,y)=0$. We need to prove that $H_i(x,y)\neq 0$ for all $i$. An Eisenstein integer $a+b\omega$ is zero exactly when both integer coefficients are zero.
For $i=0,1,3,5,6,7,8,14$ the coefficient
of $\omega$ is a nonzero constant. For $i=10$ it is
$-(x^2+x+1)\ne0$. The zeros of $H_9,H_{15},H_{16}$ would be
$(0,2),(0,1),(1,1)$, where $f$ is respectively $4,-1,-3$.
For the rational-valued vertices substitution gives
\[
\begin{array}{c|l}
	i&\text{necessary polynomial equation in an integer }t\\\hline
	2,4&t^3-t^2+t-2=0\\
	11&-t^4-t-1=0\\
	12&-t^4+5t^3-7t^2+4t-2=0\\
	13&-t^4-t^3-t^2-2t-2=0.
\end{array}
\]
An integer root of a polynomial with integer coefficients and
nonzero constant term divides that constant term. Thus only $\pm1$
need be tested in row $11$, and only $\pm1,\pm2$ in the other rows;
substitution excludes all of them. Finally,
$y^2-y+1>0$ excludes $H_{17}=0$; $H_{18}=0$ would give $x+y=1$
and first coefficient $1$; $H_{19}=0$ would give $x(x+y)=0$
and first coefficient $y^2+1$ or $1$. Thus every $H_i$ is nonzero.

\emph{Assumption~\ref{ass:cubic-charts}: parametrizations.} For $i=0,1,8,10,11,13,14,15,16$,
put $x=t$ and solve $H_i=0$ for $y$. For the other $i<17$, put
$y=t$ and solve for $x$. These charts have no denominators.
The three remaining vertices factor as
\[
 H_{17}=y(y-\omega x-1)+1,\quad
 H_{18}=(y+1+\omega)(y+x-1)+1,\quad
 H_{19}=(y+x)(y-(1+\omega)x)+1.
\]
Putting the first factor equal to $t$ gives
\begin{equation}\label{eq:detailed-conics}
\begin{split}
 \sigma_{17}(t)&=((t^2-t+1)/(\omega t),\ t),\\
 \sigma_{18}(t)&=(-t+2+\omega-1/t,\ t-1-\omega),\\
 \sigma_{19}(t)&=((t^2+1)/((2+\omega)t),\
                    t-(t^2+1)/((2+\omega)t)).
\end{split}
\end{equation}
To clear the variable denominators, take
$\Delta_{17}=\omega t$, $\Delta_{18}=-t$ and
$\Delta_{19}=(2+\omega)t$; set $\Delta_i=1$ for the other vertices.
In each case $X_i=\Delta_i(x\circ\sigma_i)$ and
$Y_i=\Delta_i(y\circ\sigma_i)$ are polynomials in $K[t]$.
Substitution gives $H_i\circ\sigma_i=0$ identically. The coordinate
chosen as $t$, or the displayed linear factor chosen as $t$, is
nonconstant, so each chart satisfies Assumption~\ref{ass:cubic-charts}.

\emph{Assumption~\ref{ass:cubic-identities}: cube identities.}
Define $F_i=\Delta_i^4f\circ\sigma_i$ and
$G_{ij}=\Delta_i^2H_j\circ\sigma_i$. Each of the twenty $F_i$
is nonzero: their degrees, in vertex order, are
\[
 4,4,3,3,3,3,3,3,4,4,6,4,4,4,4,4,4,8,8,8.
\]
Use the polynomials $W_i$ in
Table~\ref{tab:detailed-roots}, put $n_i=0$, and define $Q_i$ by
dividing the left side of \eqref{eq:cubic-vertex} by $F_i$.
The congruences
\begin{equation*}
	\Delta_i^{2\sum_j\widetilde\rho_{ij}}W_i^3
	\equiv c_i\prod_jG_{ij}^{\widetilde\rho_{ij}}\pmod{F_i}\text{}
\end{equation*}
hold for all twenty vertices; equivalently, all twenty divisions are exact. For example, at vertex $0$, $F_0=-t^4-t-2-2\omega$, and the
neighbours restrict to $t-1$, $t-1-\omega$, $-1-t-t\omega$ and
$2+2\omega-t$, with weights $1,1,1,2$. The full identity is
\[
\begin{aligned}
 &(4+2t^3)^3-4(t-1)(t-1-\omega)(-1-t-t\omega)(2+2\omega-t)^2\\
 &\quad=(-t^4-t-2-2\omega)
 (-8t^5-40t^2+12t+(12t+24)\omega).
\end{aligned}
\]
All twenty polynomial remainders are zero; the quotient degrees,
in order, are
\[
 5,13,10,6,8,12,12,0,7,7,14,7,6,5,0,7,5,29,33,31.
\]
The only denominator in the roots is $3$ in $W_{19}$.

\emph{Assumption~\ref{ass:cubic-resultants}: edge resultants.}
The $58$ numbers $\Norm(\Res_t(F_i,G_{ij}))$ for the listed edges
are nonzero and are supported on $\{2,3\}$. These are fully reproducible
determinants from the displayed data; for example
$\Res(F_0,G_{0,8})=F_0(1)=-4-2\omega$, of norm $12$.
For each edge, take the B\'ezout polynomials $U_{ij},V_{ij}$
obtained from the adjugate of its Sylvester matrix; this specifies
\eqref{eq:cubic-bezout} using exact arithmetic.

\emph{Assumption~\ref{ass:cubic-primes}: exceptional primes.}
The set $S=\{2,3\}$ contains $3$. Since
$c_i=\omega^{\kappa_i}2^{r_i}\lambda^{s_i}$, its norm is
$4^{r_i}3^{s_i}$, and every edge-resultant norm is supported on $\{2,3\}$. The chart polynomials, $F_i$ and $G_{ij}$ have
Eisenstein-integer coefficients. The roots $W_i$ have coefficient denominators supported on
$\{3\}$, and all the quotients $Q_i$ are Eisenstein-integral,
as verified by the exact divisions. The adjugate construction makes
$U_{ij},V_{ij}$ Eisenstein-integral as well. This checks each
part of Assumption~\ref{ass:cubic-primes}.

\emph{Assumption~\ref{ass:cubic-coverage}: coverage.}
For $i<17$ the coefficient used to solve for a coordinate is a
unit of $\mathcal E$, so every zero of $H_i$ is represented modulo
every prime element. For $i=17,18,19$, each displayed factorization
has the form $H_i=tL+1$; hence $H_i=0$ forces $t\ne0$.
The additional divisors used in the charts are $\omega,-1,2+\omega$,
whose norms are $1,1,3$. They are nonzero modulo every $\pi\mid p$
with $p\notin S$. Thus every zero of each $H_i$, and in particular
every common zero of $f,H_i$, is covered as required.

All six algebraic assumptions have now been checked, so the first
conclusion of Theorem~\ref{prop:cubic-certificate} gives
\begin{equation*}
 B_p=0\quad(p\ne2,3).
\end{equation*}

\begingroup
		\footnotesize
		\setlength{\tabcolsep}{3pt}
		\renewcommand{\arraystretch}{1.16}
		\begin{longtable}{@{}rr@{\hspace{9pt}}p{0.425\textwidth}@{\hspace{7pt}}p{0.425\textwidth}@{}}
\caption{The polynomials $W_i(t)=(P_i(t)+\omega R_i(t))/\mu_i$ in the vertex identities for \eqref{eq:detailed}.}\label{tab:detailed-roots}\\
			\toprule $i$ & $\mu_i$ & $P_i(t)$ & $R_i(t)$\\\midrule
		0 & 1 & $2 t^{3} + 4$ & $0$\\\addlinespace[3pt]
		1 & 1 & $- 4 t - 4$ & $- 4 t - 4$\\\addlinespace[3pt]
		2 & 1 & $2 t^{2}$ & $2 t$\\\addlinespace[3pt]
		3 & 1 & $2 t^{2} + 4 t - 8$ & $6 t - 4$\\\addlinespace[3pt]
		4 & 1 & $2 t^{2} - 2 t + 2$ & $2 t^{2} - 2 t + 2$\\\addlinespace[3pt]
		5 & 1 & $2 t - 2$ & $- 2$\\\addlinespace[3pt]
		6 & 1 & $2 t^{2} - 4 t + 8$ & $2 t^{2} - 4 t + 2$\\\addlinespace[3pt]
		7 & 1 & $1$ & $0$\\\addlinespace[3pt]
		8 & 1 & $- 2 t^{3} - 2 t^{2} + 4 t$ & $2 t^{3} - 4 t^{2} + 2 t$\\\addlinespace[3pt]
		9 & 1 & $4 t^{3} - 16 t^{2} + 24 t - 12$ & $2 t^{3} - 8 t^{2} + 12 t - 6$\\\addlinespace[3pt]
		10 & 1 & $2 t^{4} - 2 t - 2$ & $- 2 t^{3} - 4 t^{2} - 2 t - 2$\\\addlinespace[3pt]
		11 & 1 & $- t + 1$ & $- t^{3} + t^{2}$\\\addlinespace[3pt]
		12 & 1 & $4 t^{3} - 8 t^{2} + 8 t - 4$ & $2 t^{3} - 4 t^{2} + 4 t - 2$\\\addlinespace[3pt]
		13 & 1 & $- t^{3} + 1$ & $t^{2} + t + 1$\\\addlinespace[3pt]
		14 & 1 & $2$ & $2$\\\addlinespace[3pt]
		15 & 1 & $3 t^{3} - 2 t^{2} - 2 t - 1$ & $- 4 t^{2} + 2 t - 2$\\\addlinespace[3pt]
		16 & 1 & $- t^{3} - t^{2} - t$ & $- 2 t^{3} - 2 t^{2} - 2 t$\\\addlinespace[3pt]
		17 & 1 & $\begin{aligned}[t]&- t^{7} + t^{6} - 4 t^{5} - t^{4}\\&\quad + 3 t^{3} - 4 t^{2} + 6 t - 4\end{aligned}$ & $\begin{aligned}[t]&- t^{7} + 2 t^{6} - 3 t^{5} - t^{4}\\&\quad + 2 t^{3} - 7 t^{2} + 4 t - 4\end{aligned}$\\\addlinespace[3pt]
		18 & 1 & $\begin{aligned}[t]&- 5 t^{7} + 28 t^{6} - 32 t^{5} - 16 t^{4}\\&\quad + 47 t^{3} - 45 t^{2} + 18 t - 1\end{aligned}$ & $\begin{aligned}[t]&- 4 t^{7} + 41 t^{6} - 125 t^{5} + 176 t^{4}\\&\quad - 164 t^{3} + 98 t^{2} - 40 t + 8\end{aligned}$\\\addlinespace[3pt]
		19 & 3 & $\begin{aligned}[t]&- t^{7} - t^{5} + 9 t^{4} + 18 t^{3}\\&\quad + 8 t^{2} - 3 t + 2\end{aligned}$ & $\begin{aligned}[t]&- 2 t^{7} - 3 t^{6} - 2 t^{5} + t^{2}\\&\quad - 12 t - 2\end{aligned}$\\\addlinespace[3pt]
			\bottomrule
		\end{longtable}
	\endgroup
	
	We now prove that $C_2=\{2\}$ and $C_3=\{0\}$ satisfy the last two assumptions.

\emph{Assumption~\ref{ass:cubic-local}: local bounds.}
Put $u_i=v_2(H_i)\pmod3$
and define $\eta_i\in\mathbb F_3$ by
$H_i/2^{v_2(H_i)}\equiv\omega^{\eta_i}\pmod2$.
Modulo $3$, write $(v_i,k_i,m_i,\ell_i)$ for the four coordinates of
$H_i$ in \eqref{eq:cubic-normalization}, all reduced modulo $3$.
Substitution in \eqref{eq:cubic-beta} and
\eqref{eq:cubic-graphsum} gives
\begin{equation}\label{eq:detailed-direct-sums}
\begin{aligned}
 -B_2=D_2+Q_2,&\quad
 D_2=\sum_i\bigl(r_i\eta_i+(s_i-\kappa_i)u_i\bigr),\\
 &Q_2=\sum_{i<j}e_{ij}(\eta_i u_j-\eta_j u_i),\\
 -B_3=D_3+Q_3,&\quad
 D_3=\sum_i\bigl(\kappa_i m_i+s_i\ell_i+r_i(v_i-k_i)\bigr),\\
 &Q_3=\sum_{i<j}e_{ij}
 (m_i k_j-m_j k_i+\ell_i v_j-\ell_j v_i).
\end{aligned}
\end{equation}
Here $D_p$ is the contribution of the constants to $-B_p$ and
$Q_p$ is the contribution of the edges. The minus signs follow
from the ordering of the arguments in \eqref{eq:cubic-beta}. Here we used $\lambda\equiv\omega^2\pmod2$, so that $c_i/2^{r_i}\equiv\omega^{\kappa_i-s_i}\pmod2$, and $2=-(-2)$ with $-2\equiv1\pmod3$, $M(-2)=1$ and $L(-2)=-1$, so that the coordinates of $c_i$ at $\lambda$ are $(s_i,\kappa_i,r_i,-r_i)$.

Modulo $2$, $f=0$ forces $y$ even. The class $(1,0)$ is terminal;
$(0,0)$ lifts to $(2,0),(0,2)$ modulo $4$, and the latter lifts
to $(0,6),(4,2)$ modulo $8$. Thus the following four rows cover
all integer solutions. The last three columns are obtained from
Table~\ref{tab:detailed-input2} and \eqref{eq:detailed-direct-sums}.
\begin{center}\small
\begin{tabular}{@{}rrrrrrr@{}}
\toprule Row&Modulus&$x$&$y$&$D_2$&$Q_2$&$-B_2$\\\midrule
1&2&1&0&1&0&1\\
2&4&2&0&$-1$&$-1$&1\\
3&8&0&6&$-a-1$&$a-1$&1\\
4&8&4&2&$-1$&$-1$&1\\\bottomrule
\end{tabular}
\end{center}
Here $a=u_2$ in row $3$ is unrestricted. In row $1$ the valuations
of $H_4,H_{12}$ are also unrestricted; their coefficients cancel.
Every row therefore gives $B_2=2$.

\begingroup
\small
\setlength{\tabcolsep}{7pt}
\begin{longtable}{@{}rcccc@{}}
\caption{Input pairs $(u_i,\eta_i)$ for the four classes at $2$. Letters denote unrestricted elements of $\mathbb F_3$ in their column.}\label{tab:detailed-input2}\\
\toprule
$i$ & Row 1 & Row 2 & Row 3 & Row 4\\
\midrule
\endfirsthead
\toprule
$i$ & Row 1 & Row 2 & Row 3 & Row 4\\
\midrule
\endhead
\bottomrule
\endfoot
0 & $(0,2)$ & $(0,2)$ & $(0,2)$ & $(0,2)$\\
1 & $(0,2)$ & $(0,2)$ & $(0,2)$ & $(0,2)$\\
2 & $(0,0)$ & $(1,0)$ & $(a,0)$ & $(2,0)$\\
3 & $(0,1)$ & $(0,2)$ & $(0,2)$ & $(0,2)$\\
4 & $(a,0)$ & $(0,0)$ & $(0,0)$ & $(0,0)$\\
5 & $(0,2)$ & $(0,1)$ & $(0,1)$ & $(0,1)$\\
6 & $(0,2)$ & $(0,1)$ & $(0,1)$ & $(0,1)$\\
7 & $(0,1)$ & $(0,2)$ & $(0,2)$ & $(0,2)$\\
8 & $(0,2)$ & $(0,1)$ & $(0,1)$ & $(0,1)$\\
9 & $(0,0)$ & $(1,2)$ & $(2,1)$ & $(2,0)$\\
10 & $(0,2)$ & $(0,1)$ & $(0,1)$ & $(0,1)$\\
11 & $(0,0)$ & $(0,0)$ & $(0,0)$ & $(0,0)$\\
12 & $(b,0)$ & $(0,0)$ & $(0,0)$ & $(0,0)$\\
13 & $(0,0)$ & $(1,0)$ & $(1,0)$ & $(1,0)$\\
14 & $(0,1)$ & $(0,1)$ & $(0,1)$ & $(0,1)$\\
15 & $(0,2)$ & $(0,0)$ & $(0,0)$ & $(0,0)$\\
16 & $(0,0)$ & $(0,1)$ & $(0,1)$ & $(0,1)$\\
17 & $(0,0)$ & $(0,0)$ & $(0,0)$ & $(0,0)$\\
18 & $(0,0)$ & $(0,1)$ & $(0,1)$ & $(0,1)$\\
19 & $(0,1)$ & $(0,0)$ & $(0,0)$ & $(0,0)$\\
\end{longtable}
\endgroup

Modulo $3$, the only solution is $(1,1)$. Its lifts modulo $9$
are $(1,7),(4,1),(7,4)$, and its nine lifts modulo $27$ are
listed below. They follow either by substitution or by
\eqref{eq:cubic-lifting}, with $f_x=-4x^3-1$,
$f_y=3y^2-2y+1$. All entries in the sum columns are in
$\mathbb F_3$.
\begin{center}\small
\begin{tabular}{@{}crrrrr@{}}
\toprule Row&$x$&$y$&$D_3$&$Q_3$&$-B_3$\\\midrule
A&1&7&0&0&0\\
B&4&10&$1-t$&$t-1$&0\\
C&7&22&$-t-1$&$t+1$&0\\
D&10&16&$-1$&1&0\\
E&13&19&$1-t$&$t-1$&0\\
F&16&4&$1-t$&$t-1$&0\\
G&19&25&1&$-1$&0\\
H&22&1&$1-t$&$t-1$&0\\
I&25&13&$-t$&$t$&0\\\bottomrule
\end{tabular}
\end{center}
In each row the letter $t$ ranges over $\mathbb F_3$ and records
an undetermined coordinate. Its occurrences cancel in $D_3+Q_3$.
It is unrelated to the parameter in \eqref{eq:detailed-conics}.
Table~\ref{tab:detailed-input3} supplies every input to these sums.
There an entry $1001$ means $(v_i,k_i,m_i,\ell_i)=(1,0,0,1)$,
and the uncertain entries mean exactly the following allowed sets:
\[
\begin{split}
 \mathsf R_1&=\{(1,1,-t,t):t\in\mathbb F_3\},\\
 \mathsf R&=\{(s,s,-t,t):s,t\in\mathbb F_3\},\\
 \mathsf S_{ab}&=\{(0,a,b,t):t\in\mathbb F_3\}.
\end{split}
\]
For example, in row $A$, $H_0\equiv8+\omega\pmod{27}$; division
by $\lambda$ gives $5+3\omega$, whose normalized associate is
$-5-3\omega$. This gives $(v_0,k_0,m_0,\ell_0)=(1,0,0,1)$.
In row $H$, $H_{11}=y-1$ is zero modulo $27$ but is nonzero
as an integer. 
Its valuation is not determined by the class, so \eqref{eq:cubic-rational} only gives the set $\mathsf R$; the coordinate $s$ does not occur in $D_3$ or $Q_3$, and $t$ cancels in $D_3+Q_3$. In rows $B$ and $E$, $H_{11}\equiv9$ and $18\pmod{27}$, respectively, so $v_\lambda(H_{11})=4$ and \eqref{eq:cubic-rational} gives $\mathsf R_1$. The entries $\mathsf S_{ab}$ of $H_{16}$, whose value has $v_\lambda=3$, follow by the enumeration modulo $9$ described above.
Consequently every lift in every row gives $B_3=0$.

\begingroup
\small
\setlength{\tabcolsep}{5pt}
\begin{longtable}{@{}rccccccccc@{}}
\caption{Input quadruples $(v_i,k_i,m_i,\ell_i)$ for the nine classes at $3$; the sets $\mathsf R_1,\mathsf R,\mathsf S_{ab}$ are defined in the text.}\label{tab:detailed-input3}\\
\toprule
$i$ & $A$ & $B$ & $C$ & $D$ & $E$ & $F$ & $G$ & $H$ & $I$\\
\midrule
\endfirsthead
\toprule
$i$ & $A$ & $B$ & $C$ & $D$ & $E$ & $F$ & $G$ & $H$ & $I$\\
\midrule
\endhead
\bottomrule
\endfoot
0 & $1001$ & $1201$ & $1122$ & $1002$ & $1202$ & $1120$ & $1000$ & $1200$ & $1121$\\
1 & $0122$ & $0100$ & $0111$ & $0122$ & $0100$ & $0111$ & $0122$ & $0100$ & $0111$\\
2 & $0000$ & $0021$ & $0012$ & $0000$ & $0021$ & $0012$ & $0000$ & $0021$ & $0012$\\
3 & $1200$ & $1120$ & $1001$ & $1201$ & $1121$ & $1002$ & $1202$ & $1122$ & $1000$\\
4 & $0012$ & $0021$ & $0000$ & $0012$ & $0021$ & $0000$ & $0012$ & $0021$ & $0000$\\
5 & $0200$ & $0210$ & $0220$ & $0200$ & $0210$ & $0220$ & $0200$ & $0210$ & $0220$\\
6 & $1000$ & $1122$ & $1202$ & $1002$ & $1121$ & $1201$ & $1001$ & $1120$ & $1200$\\
7 & $0100$ & $0111$ & $0122$ & $0100$ & $0111$ & $0122$ & $0100$ & $0111$ & $0122$\\
8 & $1001$ & $1121$ & $1200$ & $1000$ & $1120$ & $1202$ & $1002$ & $1122$ & $1201$\\
9 & $1222$ & $1121$ & $1012$ & $1220$ & $1122$ & $1010$ & $1221$ & $1120$ & $1011$\\
10 & $0001$ & $0001$ & $0001$ & $0001$ & $0001$ & $0001$ & $0001$ & $0001$ & $0001$\\
11 & $2212$ & $\mathsf R_1$ & $2212$ & $2221$ & $\mathsf R_1$ & $2200$ & $2200$ & $\mathsf R$ & $2221$\\
12 & $0021$ & $0021$ & $0021$ & $0021$ & $0021$ & $0021$ & $0021$ & $0021$ & $0021$\\
13 & $0000$ & $0021$ & $0012$ & $0000$ & $0021$ & $0012$ & $0000$ & $0021$ & $0012$\\
14 & $0210$ & $0200$ & $0220$ & $0210$ & $0200$ & $0220$ & $0210$ & $0200$ & $0220$\\
15 & $0122$ & $0121$ & $0120$ & $0122$ & $0121$ & $0120$ & $0122$ & $0121$ & $0120$\\
16 & $2212$ & $2120$ & $\mathsf S_{11}$ & $2202$ & $2101$ & $\mathsf S_{22}$ & $2222$ & $2112$ & $\mathsf S_{02}$\\
17 & $1011$ & $1220$ & $1121$ & $1012$ & $1221$ & $1122$ & $1010$ & $1222$ & $1120$\\
18 & $0101$ & $0121$ & $0111$ & $0101$ & $0121$ & $0111$ & $0101$ & $0121$ & $0111$\\
19 & $0210$ & $0202$ & $0221$ & $0210$ & $0202$ & $0221$ & $0210$ & $0202$ & $0221$\\
\end{longtable}
\endgroup

The four classes at $2$ and the nine classes at $3$ cover every
integer solution, and all undetermined coordinates cancel within
each class. Thus $B_2\in C_2=\{2\}$ and $B_3\in C_3=\{0\}$,
which verifies Assumption~\ref{ass:cubic-local}.

\emph{Assumption~\ref{ass:cubic-exclusion}: incompatibility.}
In $\mathbb F_3$ one has $C_2+C_3=\{2\}$, which excludes $0$.

All eight assumptions are therefore satisfied.
Theorem~\ref{prop:cubic-certificate} proves that
\eqref{eq:detailed} has no integer solutions.
\end{proof}

\subsection{Three further equations}\label{sec:cubic-outlines}

In this subsection, we prove that equations
\begin{align}
	y^3+xy-y&=x^4+x+2,\label{eq:three-cubic-1}\\
	y^3+xy+y&=x^4+4,\label{eq:three-cubic-2}\\
	y^3+xy-y&=x^4+4\label{eq:three-cubic-3}
\end{align}
have no integer solutions. For each equation, let $f$ denote its left-hand side minus its right-hand side. 
For equations \eqref{eq:three-cubic-2} and \eqref{eq:three-cubic-3}, before applying the main method, we will first use quadratic norm identities to restrict the possible congruence classes of an integer solution.

\paragraph{Preliminary quadratic restrictions.}
Write the last two equations as $f_s=0$, where
\[
f_s=y^3+xy+sy-x^4-4,\qquad s\in\{1,-1\}.
\]
Here $s=1$ corresponds to \eqref{eq:three-cubic-2}, and $s=-1$
to \eqref{eq:three-cubic-3}. At an integer solution we have $y>0$. Indeed, $y=0$ is impossible, and if $y=-z\le-1$, then $x+y^2+s<0$ forces $x=-n$ with $n\ge z^2$, so that $n^4+4=z(n-z^2-s)\le zn\le n^2$, which is absurd. Moreover, 
\begin{equation*}
	y(x+y^2+s)=(x^2)^2+2^2.
\end{equation*}
For the additional norm identity, define
\[
H_s=g_s+r_s,\qquad
T_s=b_s-g_s^3+g_s^2r_s-g_sr_s^2+r_s^3,
\]
using the following polynomials and constants:
\[
\begin{array}{c|l|l}
	&s=1&s=-1\\\hline
	g_s&2xy&4x\\
	r_s&y^2-2y+2&2y+5\\
	b_s&8y^4&64y\\
	d_s&-2&2\\
	U_s&y^4-12y^3+12y^2-40y+68&4y^2-12y+9\\
	V_s&-8y^3+8y^2-28y+48&16y+28
\end{array}
\]
Exact polynomial expansion gives
\begin{equation*}
	H_sT_s-U_s^2-d_sV_s^2=
	\begin{cases}
		16y^4f_s,&s=1,\\
		256f_s,&s=-1.
	\end{cases}
\end{equation*}
Thus $H_sT_s=U_s^2+d_sV_s^2$ at a solution. Since $\Res_y(U_1,V_1)=2^{10}\cdot3313$ and $\Res_y(U_{-1},V_{-1})=2^6\cdot13^2$ are nonzero, and $V_{-1}=16y+28\ne0$, the right-hand side $U_s^2+d_sV_s^2$ is nonzero. Moreover $H_{-1}>0$: this is clear for $x>0$, while for $x\le0$ we have $T_{-1}=64y+(r_{-1}-g_{-1})(r_{-1}^2+g_{-1}^2)>0$. A common odd prime divisor of $U_s$ and $V_s$ divides the resultant above. Here $3313\equiv1\pmod8$, and for $s=-1$ the congruences $H_{-1}\equiv U_{-1}\equiv V_{-1}\equiv0\pmod{13}$ force $(x,y)\equiv(11,8)\pmod{13}$, where $T_{-1}\equiv12\pmod{13}$.

Therefore Corollary~\ref{cor:norm-congruence-cases}, applied with $d=1$ to $y(x+y^2+s)=(x^2)^2+2^2$ and with $d=d_s$ to $H_sT_s=U_s^2+d_sV_s^2$, shows that the odd parts of $y$ and $x+y^2+s$ are $1$ modulo $4$, that the odd part of $|H_1|$ is $\pm1$ modulo $8$, and that $H_{-1}\equiv1,3\pmod8$. Reduction modulo $8$ then gives
\begin{equation}\label{eq:cubic-preliminary}
	\begin{aligned}
		\eqref{eq:three-cubic-2}:&\quad
		(x,y)\equiv(3,1),(7,5)\pmod8,\\
		\eqref{eq:three-cubic-3}:&\quad
		(x,y)\equiv(1,5),(5,1)\pmod8.
	\end{aligned}
\end{equation}
The cubic graphs will now exclude the remaining classes.

\paragraph{Reading the graph tables.}
We write $\langle A,B\rangle=A+B\omega$, for integers or
polynomials $A,B$. Each graph is over the field $K$ given in its header and has vertex indices $0,\ldots,m-1$. The header of a graph also gives $n$, the exponent of $W_i$ in its vertex identity: $n=3$ for cubic graphs, and $n=2$ for the quadratic graphs used for \eqref{eq:H34last}. A row
$i:j_1^{e_1}\;j_2^{e_2}\cdots$ records the edges
$\{i,j_k\}$, with $i<j_k$ and weight $e_{ij_k}=e_k$.
Unlisted weights are zero; a superscript $1$ may be omitted.
Reverse orientations have the negative weight modulo $n$.
An isolated vertex with constant $1$ makes no contribution.

\needspace{8\baselineskip}\paragraph{Cubic graph for \eqref{eq:three-cubic-1}.}
$f=y^3+xy-y-x^4-x-2$.
\par\smallskip\noindent
\textit{Graph: $K=\Q(\omega),\ n=3,\ m=21$, $54$ edges.}
\begingroup
\footnotesize
\setlength{\tabcolsep}{3pt}
\setlength{\LTpre}{3pt}
\setlength{\LTpost}{3pt}
\renewcommand{\arraystretch}{1.05}
\begin{longtable}{@{}rll@{\hspace{12pt}}rll@{}}
	\toprule
	$i$ & $H_i$ & $c_i$ & $i$ & $H_i$ & $c_i$\\
	\midrule\endfirsthead
	\multicolumn{6}{l}{\textit{Vertices (continued)}}\\
	\toprule
	$i$ & $H_i$ & $c_i$ & $i$ & $H_i$ & $c_i$\\
	\midrule\endhead
	\bottomrule\endfoot
	0 & $\langle y+1,1\rangle$ & $\langle-1,-1\rangle$ & 11 & $\langle x+1,y+1\rangle$ & $\langle1,0\rangle$\\
	1 & $\langle y-2,0\rangle$ & $\langle-6,0\rangle$ & 12 & $\langle x+y,0\rangle$ & $\langle-4,-2\rangle$\\
	2 & $\langle x+1,1\rangle$ & $\langle1,2\rangle$ & 13 & $\langle3x-y+1,-2y-1\rangle$ & $\langle0,1\rangle$\\
	3 & $\langle x+1,0\rangle$ & $\langle1,0\rangle$ & 14 & $\langle x+y+1,x+1\rangle$ & $\langle2,0\rangle$\\
	4 & $\langle x,1\rangle$ & $\langle-4,-4\rangle$ & 15 & $\langle y+1,x+1\rangle$ & $\langle0,1\rangle$\\
	5 & $\langle x,-1\rangle$ & $\langle1,-1\rangle$ & 16 & $\langle y,x\rangle$ & $\langle1,-1\rangle$\\
	6 & $\langle x-1,0\rangle$ & $\langle3,3\rangle$ & 17 & $\langle y-1,0\rangle$ & $\langle2,0\rangle$\\
	7 & $\langle x-1,-1\rangle$ & $\langle2,4\rangle$ & 18 & $\langle x^2+y-1,0\rangle$ & $\langle-2,-2\rangle$\\
	8 & $\langle x+2,0\rangle$ & $\langle2,4\rangle$ & 19 & $\langle x+y+1,1-x^2\rangle$ & $\langle-3,0\rangle$\\
	9 & $\langle-x+y,0\rangle$ & $\langle4,-4\rangle$ & 20 & $\langle-x^2+y,-x^2+x\rangle$ & $\langle2,4\rangle$\\
	10 & $\langle x+y,y\rangle$ & $\langle-1,-1\rangle$ & & &\\
\end{longtable}\addtocounter{table}{-1}
\par\noindent
\edgerow{0}{2^{1}\;\allowbreak 6^{1}\;\allowbreak 14^{1}\;\allowbreak 18^{2}}
\edgerow{1}{2^{2}\;\allowbreak 4^{1}\;\allowbreak 5^{2}\;\allowbreak 13^{1}\;\allowbreak 18^{2}\;\allowbreak 20^{2}}
\edgerow{2}{12^{2}\;\allowbreak 15^{2}\;\allowbreak 19^{2}\;\allowbreak 20^{2}}
\edgerow{3}{14^{2}\;\allowbreak 18^{1}}
\edgerow{4}{10^{2}\;\allowbreak 14^{2}\;\allowbreak 16^{2}\;\allowbreak 17^{1}\;\allowbreak 18^{2}\;\allowbreak 19^{2}\;\allowbreak 20^{1}}
\edgerow{5}{11^{2}\;\allowbreak 15^{1}}
\edgerow{6}{10^{2}\;\allowbreak 20^{2}}
\edgerow{7}{10^{2}\;\allowbreak 13^{2}\;\allowbreak 15^{1}\;\allowbreak 17^{1}\;\allowbreak 19^{2}}
\edgerow{8}{9^{2}\;\allowbreak 17^{1}}
\edgerow{9}{11^{2}\;\allowbreak 13^{2}\;\allowbreak 17^{1}\;\allowbreak 18^{1}\;\allowbreak 20^{1}}
\edgerow{10}{12^{2}\;\allowbreak 19^{2}\;\allowbreak 20^{2}}
\edgerow{11}{19^{2}}
\edgerow{12}{14^{1}\;\allowbreak 17^{2}\;\allowbreak 18^{1}}
\edgerow{13}{16^{2}\;\allowbreak 19^{2}\;\allowbreak 20^{1}}
\edgerow{14}{16^{1}\;\allowbreak 18^{1}\;\allowbreak 20^{1}}
\edgerow{15}{19^{1}}
\edgerow{16}{18^{1}}
\par\endgroup
\par\addvspace{10pt}

\needspace{8\baselineskip}\paragraph{Cubic graph for \eqref{eq:three-cubic-2}.}
$f=y^3+xy+y-x^4-4$.
\par\smallskip\noindent
\textit{Graph: $K=\Q(\omega),\ n=3,\ m=37$, $123$ edges.}
\begingroup
\footnotesize
\setlength{\tabcolsep}{3pt}
\setlength{\LTpre}{3pt}
\setlength{\LTpost}{3pt}
\renewcommand{\arraystretch}{1.05}
\begin{longtable}{@{}rll@{\hspace{12pt}}rll@{}}
	\toprule
	$i$ & $H_i$ & $c_i$ & $i$ & $H_i$ & $c_i$\\
	\midrule\endfirsthead
	\multicolumn{6}{l}{\textit{Vertices (continued)}}\\
	\toprule
	$i$ & $H_i$ & $c_i$ & $i$ & $H_i$ & $c_i$\\
	\midrule\endhead
	\bottomrule\endfoot
	0 & $\langle y,0\rangle$ & $\langle-2,-2\rangle$ & 19 & $\langle x+y+1,x\rangle$ & $\langle1,2\rangle$\\
	1 & $\langle y+1,1\rangle$ & $\langle1,0\rangle$ & 20 & $\langle x+y,1\rangle$ & $\langle-4,-4\rangle$\\
	2 & $\langle y+1,0\rangle$ & $\langle0,1\rangle$ & 21 & $\langle-x+y-1,-x\rangle$ & $\langle4,0\rangle$\\
	3 & $\langle y,-1\rangle$ & $\langle3,-5\rangle$ & 22 & $\langle x+y-1,-1\rangle$ & $\langle1,0\rangle$\\
	4 & $\langle y-1,0\rangle$ & $\langle0,4\rangle$ & 23 & $\langle x+y,y\rangle$ & $\langle4,-4\rangle$\\
	5 & $\langle y-1,-1\rangle$ & $\langle0,-3\rangle$ & 24 & $\langle x,y\rangle$ & $\langle2,0\rangle$\\
	6 & $\langle y-1,-2\rangle$ & $\langle-20,12\rangle$ & 25 & $\langle x,-y\rangle$ & $\langle6,6\rangle$\\
	7 & $\langle y+2,2\rangle$ & $\langle-12,0\rangle$ & 26 & $\langle x-y,0\rangle$ & $\langle-2,-2\rangle$\\
	8 & $\langle y-2,0\rangle$ & $\langle0,-21\rangle$ & 27 & $\langle x+1,y+1\rangle$ & $\langle1,0\rangle$\\
	9 & $\langle x+1,0\rangle$ & $\langle-12,0\rangle$ & 28 & $\langle x-y+1,1-y\rangle$ & $\langle4,0\rangle$\\
	10 & $\langle x,1\rangle$ & $\langle-3,0\rangle$ & 29 & $\langle x+y+2,y+1\rangle$ & $\langle1,0\rangle$\\
	11 & $\langle x,-1\rangle$ & $\langle-1,-1\rangle$ & 30 & $\langle x-y+2,0\rangle$ & $\langle-12,0\rangle$\\
	12 & $\langle x-1,0\rangle$ & $\langle-5,3\rangle$ & 31 & $\langle y,x^2\rangle$ & $\langle0,-6\rangle$\\
	13 & $\langle x-1,-1\rangle$ & $\langle0,1\rangle$ & 32 & $\langle x^2-x+y+1,0\rangle$ & $\langle-1,-1\rangle$\\
	14 & $\langle x+2,1\rangle$ & $\langle-1,2\rangle$ & 33 & $\langle x+y-1,x^2+x\rangle$ & $\langle-1,-1\rangle$\\
	15 & $\langle x+1,-1\rangle$ & $\langle-8,-3\rangle$ & 34 & $\langle-x^2+y-1,1-x^2\rangle$ & $\langle1,2\rangle$\\
	16 & $\langle x+2,2\rangle$ & $\langle16,20\rangle$ & 35 & $\langle y-2,x^2-1\rangle$ & $\langle1,0\rangle$\\
	17 & $\langle y,x\rangle$ & $\langle-2,-2\rangle$ & 36 & $\langle y^2-2y+1,-x^2-xy-2y+2\rangle$ & $\langle-1,-1\rangle$\\
	18 & $\langle x+y+1,1\rangle$ & $\langle0,1\rangle$ & & &\\
\end{longtable}\addtocounter{table}{-1}
\par\noindent
\edgerow{0}{17^{1}\;\allowbreak 34^{2}\;\allowbreak 35^{2}}
\edgerow{1}{13^{1}\;\allowbreak 14^{1}\;\allowbreak 20^{1}\;\allowbreak 27^{1}\;\allowbreak 28^{2}\;\allowbreak 29^{1}\;\allowbreak 34^{1}}
\edgerow{2}{11^{2}\;\allowbreak 12^{1}\;\allowbreak 14^{1}\;\allowbreak 15^{1}\;\allowbreak 17^{1}\;\allowbreak 19^{2}\;\allowbreak 20^{2}\;\allowbreak 23^{1}\;\allowbreak 24^{2}\;\allowbreak 32^{2}\;\allowbreak 33^{2}\;\allowbreak 34^{1}}
\edgerow{3}{10^{1}\;\allowbreak 15^{1}\;\allowbreak 32^{1}\;\allowbreak 35^{2}}
\edgerow{4}{12^{1}\;\allowbreak 20^{1}\;\allowbreak 22^{1}\;\allowbreak 28^{2}\;\allowbreak 30^{1}\;\allowbreak 36^{2}}
\edgerow{5}{20^{2}\;\allowbreak 24^{1}\;\allowbreak 26^{2}\;\allowbreak 28^{1}\;\allowbreak 33^{2}}
\edgerow{6}{12^{2}\;\allowbreak 13^{1}\;\allowbreak 25^{1}\;\allowbreak 32^{1}}
\edgerow{7}{9^{2}\;\allowbreak 10^{1}\;\allowbreak 13^{2}\;\allowbreak 16^{2}\;\allowbreak 25^{1}\;\allowbreak 31^{2}\;\allowbreak 32^{2}\;\allowbreak 33^{1}}
\edgerow{8}{10^{2}\;\allowbreak 28^{2}\;\allowbreak 32^{2}\;\allowbreak 35^{2}}
\edgerow{9}{23^{1}\;\allowbreak 24^{2}\;\allowbreak 27^{1}}
\edgerow{10}{22^{1}\;\allowbreak 26^{1}\;\allowbreak 27^{2}\;\allowbreak 29^{1}\;\allowbreak 31^{1}}
\edgerow{11}{18^{2}\;\allowbreak 23^{1}\;\allowbreak 30^{2}\;\allowbreak 33^{1}}
\edgerow{12}{17^{1}\;\allowbreak 22^{1}\;\allowbreak 23^{1}\;\allowbreak 32^{1}\;\allowbreak 34^{1}}
\edgerow{13}{20^{2}\;\allowbreak 21^{1}\;\allowbreak 25^{2}\;\allowbreak 27^{1}\;\allowbreak 29^{2}\;\allowbreak 35^{2}}
\edgerow{14}{22^{2}\;\allowbreak 30^{2}}
\edgerow{15}{17^{1}\;\allowbreak 22^{1}\;\allowbreak 23^{1}\;\allowbreak 24^{1}}
\edgerow{16}{19^{1}\;\allowbreak 21^{1}\;\allowbreak 23^{2}\;\allowbreak 24^{1}\;\allowbreak 30^{1}\;\allowbreak 31^{1}\;\allowbreak 34^{2}}
\edgerow{17}{30^{1}}
\edgerow{18}{19^{1}\;\allowbreak 24^{2}\;\allowbreak 25^{1}\;\allowbreak 29^{2}\;\allowbreak 31^{1}}
\edgerow{19}{22^{1}\;\allowbreak 28^{1}\;\allowbreak 30^{2}\;\allowbreak 32^{1}\;\allowbreak 33^{2}}
\edgerow{20}{31^{1}\;\allowbreak 35^{1}}
\edgerow{21}{25^{1}\;\allowbreak 30^{2}\;\allowbreak 33^{1}\;\allowbreak 35^{1}}
\edgerow{22}{25^{2}}
\edgerow{23}{28^{2}\;\allowbreak 32^{2}\;\allowbreak 35^{1}}
\edgerow{24}{25^{1}}
\edgerow{25}{32^{1}\;\allowbreak 34^{1}\;\allowbreak 36^{2}}
\edgerow{26}{31^{1}}
\edgerow{27}{34^{2}\;\allowbreak 36^{1}}
\edgerow{28}{29^{2}\;\allowbreak 34^{2}}
\edgerow{29}{33^{1}}
\edgerow{30}{32^{2}}
\edgerow{31}{34^{1}}
\edgerow{32}{36^{1}}
\par\endgroup
\par\addvspace{10pt}

\needspace{8\baselineskip}\paragraph{Cubic graph for \eqref{eq:three-cubic-3}.}
$f=y^3+xy-y-x^4-4$.
\par\smallskip\noindent
\textit{Graph: $K=\Q(\omega),\ n=3,\ m=47$, $180$ edges.}
\begingroup
\footnotesize
\setlength{\tabcolsep}{3pt}
\setlength{\LTpre}{3pt}
\setlength{\LTpost}{3pt}
\renewcommand{\arraystretch}{1.05}
\begin{longtable}{@{}rll@{\hspace{12pt}}rll@{}}
	\toprule
	$i$ & $H_i$ & $c_i$ & $i$ & $H_i$ & $c_i$\\
	\midrule\endfirsthead
	\multicolumn{6}{l}{\textit{Vertices (continued)}}\\
	\toprule
	$i$ & $H_i$ & $c_i$ & $i$ & $H_i$ & $c_i$\\
	\midrule\endhead
	\bottomrule\endfoot
	0 & $\langle y,0\rangle$ & $\langle0,-3\rangle$ & 24 & $\langle x+y-2,-1\rangle$ & $\langle-2,-1\rangle$\\
	1 & $\langle y+1,1\rangle$ & $\langle0,1\rangle$ & 25 & $\langle x+y,x-2\rangle$ & $\langle0,-3\rangle$\\
	2 & $\langle y+1,0\rangle$ & $\langle-4,-4\rangle$ & 26 & $\langle x+y,y\rangle$ & $\langle0,1\rangle$\\
	3 & $\langle y-1,0\rangle$ & $\langle4,0\rangle$ & 27 & $\langle x+y,0\rangle$ & $\langle-1,-1\rangle$\\
	4 & $\langle y-1,-1\rangle$ & $\langle-1,-1\rangle$ & 28 & $\langle x,y\rangle$ & $\langle6,6\rangle$\\
	5 & $\langle y+2,1\rangle$ & $\langle-1,-1\rangle$ & 29 & $\langle x,-y\rangle$ & $\langle-4,-4\rangle$\\
	6 & $\langle y-2,0\rangle$ & $\langle2,0\rangle$ & 30 & $\langle x-y,-y\rangle$ & $\langle-6,0\rangle$\\
	7 & $\langle x,0\rangle$ & $\langle1,0\rangle$ & 31 & $\langle x-y+1,1\rangle$ & $\langle1,-1\rangle$\\
	8 & $\langle x+1,0\rangle$ & $\langle-2,-2\rangle$ & 32 & $\langle x-y,-1\rangle$ & $\langle0,2\rangle$\\
	9 & $\langle x,1\rangle$ & $\langle-3,0\rangle$ & 33 & $\langle x-1,-y\rangle$ & $\langle0,-3\rangle$\\
	10 & $\langle x,-1\rangle$ & $\langle1,2\rangle$ & 34 & $\langle x-y-1,0\rangle$ & $\langle0,-3\rangle$\\
	11 & $\langle x-1,0\rangle$ & $\langle12,12\rangle$ & 35 & $\langle x+y+2,1\rangle$ & $\langle3,3\rangle$\\
	12 & $\langle x-1,-1\rangle$ & $\langle1,0\rangle$ & 36 & $\langle x-2,y-1\rangle$ & $\langle4,-4\rangle$\\
	13 & $\langle x-1,1\rangle$ & $\langle0,1\rangle$ & 37 & $\langle y+1,x^2-x+1\rangle$ & $\langle0,4\rangle$\\
	14 & $\langle x-2,-1\rangle$ & $\langle1,0\rangle$ & 38 & $\langle x^2+x+y+1,0\rangle$ & $\langle-1,-1\rangle$\\
	15 & $\langle x,-2\rangle$ & $\langle1,-1\rangle$ & 39 & $\langle-x^2+x+y,-x^2+x-1\rangle$ & $\langle1,2\rangle$\\
	16 & $\langle-x+y+1,1\rangle$ & $\langle1,0\rangle$ & 40 & $\langle-x^2-x+y-1,-x^2-x\rangle$ & $\langle2,-2\rangle$\\
	17 & $\langle x+y+1,0\rangle$ & $\langle0,1\rangle$ & 41 & $\langle-x+y-1,-x^2-x-1\rangle$ & $\langle1,0\rangle$\\
	18 & $\langle y+1,-x\rangle$ & $\langle4,8\rangle$ & 42 & $\langle x^2+y+1,x^2-1\rangle$ & $\langle0,1\rangle$\\
	19 & $\langle-x+y,-x-1\rangle$ & $\langle1,0\rangle$ & 43 & $\langle x^2+x+y+2,x^2+x+1\rangle$ & $\langle-2,-2\rangle$\\
	20 & $\langle x+y-1,x\rangle$ & $\langle1,0\rangle$ & 44 & $\langle y+1,-x^2-x-1\rangle$ & $\langle-1,-1\rangle$\\
	21 & $\langle y+1,x+2\rangle$ & $\langle0,4\rangle$ & 45 & $\langle x+y^2-y-1,1-y\rangle$ & $\langle1,0\rangle$\\
	22 & $\langle-x+y+1,-x-1\rangle$ & $\langle3,3\rangle$ & 46 & $\langle xy+x+y^2+y,xy+x-2\rangle$ & $\langle2,0\rangle$\\
	23 & $\langle y-1,x+1\rangle$ & $\langle1,2\rangle$ & & &\\
\end{longtable}\addtocounter{table}{-1}
\par\noindent
\edgerow{0}{22^{2}\;\allowbreak 39^{1}\;\allowbreak 40^{1}}
\edgerow{1}{21^{1}\;\allowbreak 25^{1}\;\allowbreak 26^{2}\;\allowbreak 27^{1}\;\allowbreak 37^{2}\;\allowbreak 38^{1}\;\allowbreak 42^{1}}
\edgerow{2}{28^{1}\;\allowbreak 35^{1}\;\allowbreak 37^{1}\;\allowbreak 38^{1}\;\allowbreak 43^{2}\;\allowbreak 46^{2}}
\edgerow{3}{18^{2}\;\allowbreak 31^{2}\;\allowbreak 32^{1}\;\allowbreak 38^{1}\;\allowbreak 39^{2}\;\allowbreak 40^{2}\;\allowbreak 43^{2}\;\allowbreak 44^{2}\;\allowbreak 45^{2}}
\edgerow{4}{25^{1}\;\allowbreak 30^{2}\;\allowbreak 32^{1}\;\allowbreak 37^{1}\;\allowbreak 44^{1}}
\edgerow{5}{11^{2}\;\allowbreak 34^{2}\;\allowbreak 39^{1}\;\allowbreak 41^{2}\;\allowbreak 44^{1}}
\edgerow{6}{7^{2}\;\allowbreak 8^{1}}
\edgerow{7}{16^{2}\;\allowbreak 19^{1}\;\allowbreak 20^{2}\;\allowbreak 21^{2}\;\allowbreak 23^{2}\;\allowbreak 25^{2}\;\allowbreak 29^{1}\;\allowbreak 34^{1}\;\allowbreak 37^{2}\;\allowbreak 39^{1}\;\allowbreak 46^{1}}
\edgerow{8}{24^{1}\;\allowbreak 30^{1}\;\allowbreak 40^{2}\;\allowbreak 43^{2}\;\allowbreak 44^{1}}
\edgerow{9}{16^{1}\;\allowbreak 24^{2}\;\allowbreak 30^{2}\;\allowbreak 34^{2}\;\allowbreak 37^{2}\;\allowbreak 42^{1}\;\allowbreak 45^{2}}
\edgerow{10}{19^{1}\;\allowbreak 21^{1}\;\allowbreak 23^{1}\;\allowbreak 28^{1}\;\allowbreak 30^{2}\;\allowbreak 37^{1}\;\allowbreak 44^{1}\;\allowbreak 45^{2}}
\edgerow{11}{18^{2}\;\allowbreak 21^{2}\;\allowbreak 22^{2}\;\allowbreak 27^{2}\;\allowbreak 29^{2}\;\allowbreak 39^{2}\;\allowbreak 40^{1}\;\allowbreak 44^{2}\;\allowbreak 45^{1}\;\allowbreak 46^{1}}
\edgerow{12}{19^{2}\;\allowbreak 33^{2}\;\allowbreak 36^{2}\;\allowbreak 42^{1}\;\allowbreak 43^{1}}
\edgerow{13}{19^{1}\;\allowbreak 21^{1}\;\allowbreak 24^{1}\;\allowbreak 29^{1}\;\allowbreak 34^{2}\;\allowbreak 36^{1}\;\allowbreak 39^{2}\;\allowbreak 44^{2}}
\edgerow{14}{20^{1}\;\allowbreak 26^{2}\;\allowbreak 29^{2}\;\allowbreak 34^{2}\;\allowbreak 39^{1}}
\edgerow{15}{19^{1}\;\allowbreak 20^{1}\;\allowbreak 23^{1}\;\allowbreak 25^{2}\;\allowbreak 28^{1}\;\allowbreak 34^{2}\;\allowbreak 45^{1}}
\edgerow{16}{22^{1}\;\allowbreak 24^{1}\;\allowbreak 25^{2}\;\allowbreak 30^{2}\;\allowbreak 36^{2}\;\allowbreak 39^{1}}
\edgerow{17}{21^{1}\;\allowbreak 44^{2}\;\allowbreak 46^{2}}
\edgerow{18}{21^{1}\;\allowbreak 40^{2}\;\allowbreak 44^{2}}
\edgerow{19}{23^{1}\;\allowbreak 34^{2}\;\allowbreak 41^{1}\;\allowbreak 42^{2}\;\allowbreak 43^{2}\;\allowbreak 45^{2}}
\edgerow{20}{21^{2}\;\allowbreak 27^{1}\;\allowbreak 39^{1}\;\allowbreak 40^{1}\;\allowbreak 45^{1}}
\edgerow{21}{25^{2}\;\allowbreak 34^{2}\;\allowbreak 43^{1}\;\allowbreak 44^{1}}
\edgerow{22}{25^{1}\;\allowbreak 30^{2}\;\allowbreak 34^{2}\;\allowbreak 39^{1}\;\allowbreak 40^{2}}
\edgerow{23}{39^{2}\;\allowbreak 45^{1}}
\edgerow{24}{28^{1}\;\allowbreak 37^{1}}
\edgerow{25}{30^{1}\;\allowbreak 34^{1}\;\allowbreak 37^{1}\;\allowbreak 39^{1}\;\allowbreak 44^{1}}
\edgerow{26}{30^{2}\;\allowbreak 37^{1}\;\allowbreak 43^{2}}
\edgerow{27}{33^{2}\;\allowbreak 34^{1}\;\allowbreak 40^{2}}
\edgerow{28}{39^{2}\;\allowbreak 45^{1}}
\edgerow{29}{42^{2}\;\allowbreak 44^{1}}
\edgerow{30}{37^{1}\;\allowbreak 39^{2}\;\allowbreak 44^{1}}
\edgerow{31}{39^{2}\;\allowbreak 40^{2}\;\allowbreak 43^{1}}
\edgerow{32}{35^{1}\;\allowbreak 37^{2}\;\allowbreak 40^{1}\;\allowbreak 43^{2}}
\edgerow{33}{37^{2}\;\allowbreak 43^{1}}
\edgerow{34}{36^{2}\;\allowbreak 37^{2}\;\allowbreak 40^{1}}
\edgerow{35}{37^{1}\;\allowbreak 40^{1}}
\edgerow{36}{39^{1}}
\edgerow{38}{43^{2}\;\allowbreak 44^{1}}
\edgerow{39}{41^{1}\;\allowbreak 43^{1}\;\allowbreak 45^{1}}
\edgerow{41}{45^{1}}
\edgerow{42}{43^{2}}
\edgerow{43}{46^{1}}
\par\endgroup
\par\addvspace{10pt}

\begin{proposition}\label{prop:three-cubic}
	None of the equations \eqref{eq:three-cubic-1},
	\eqref{eq:three-cubic-2} and \eqref{eq:three-cubic-3}
	has an integer solution.
\end{proposition}

\begin{proof}[Proof outline]
	For each equation use the cubic graph listed under its equation
	number. Set $d=4$, $h=2$, and take the common exceptional set
	\begin{equation}\label{eq:three-cubic-S}
		S=\{2,3,5,7,11,13\}.
	\end{equation}
	We verify the eight assumptions of
	Theorem~\ref{prop:cubic-certificate}.
	
	\emph{Assumption~\ref{ass:cubic-nonzero}: nonzero auxiliary values.}
	For a vertex linear in one coordinate, solve for that coordinate
	and substitute into $f$. The greatest common divisor of the two
	rational coefficient polynomials has no integer root giving an
	integer value of the other coordinate. There are two remaining
	conics. For vertex $36$ of \eqref{eq:three-cubic-2}, an integer
	zero forces $y=1$ and $x\in\{0,-1\}$, where $f=-2,-4$.
	For vertex $46$ of \eqref{eq:three-cubic-3}, its rational component
	forces $y=-1$ or $x=-y$; its $\omega$-coefficient is then $-2$
	or $-y^2-y-2$, respectively, and cannot vanish. Hence no vertex
	vanishes at an integer solution.
	
	\emph{Assumption~\ref{ass:cubic-charts}: parametrizations.}
	Except for these two conics, every vertex has the form
	$H_i=a_i z+b_i(t)$ with $a_i\in K\setminus\{0\}$ and
	$\{z,t\}=\{x,y\}$. Use $z=-b_i(t)/a_i$, with the other
	coordinate equal to $t$, and take $\Delta_i=1$.
	For vertex $36$ of \eqref{eq:three-cubic-2}, use
	\[
	\Delta_{36}=t^2-\omega t-\omega,\qquad
	X_{36}=-t^2+2\omega t,\qquad
	Y_{36}=(1+\omega)t^2-2\omega t-\omega.
	\]
	For vertex $46$ of \eqref{eq:three-cubic-3}, use
	\[
	\Delta_{46}=(1+\omega)t,\qquad
	X_{46}=2\omega-t^2+t,\qquad
	Y_{46}=(t-1)(1+\omega)t.
	\]
	Substitution verifies $H_i\circ\sigma_i=0$ in every case;
	all these charts are nonconstant.
	
	\emph{Assumption~\ref{ass:cubic-identities}: cube identities.}
	Form $F_i,G_{ij}$ by \eqref{eq:cubic-FG}; every $F_i$ is nonzero.
	Recover $W_i$ by the procedure in Section~\ref{sec:cubic-local},
	retaining root choices with denominators supported on $S$, and
	compute $n_i,Q_i$ by \eqref{eq:cubic-vertex}. Exact division gives
	zero remainder at all $21$, $37$ and $47$ vertices, respectively.
	
	\emph{Assumption~\ref{ass:cubic-resultants}: edge resultants.}
	The $54$, $123$ and $180$ edge resultants, respectively, are all
	nonzero. Their Sylvester adjugates give the B\'ezout polynomials
	in \eqref{eq:cubic-bezout}.
	
	\emph{Assumption~\ref{ass:cubic-primes}: exceptional primes.}
	All constant and edge-resultant norms, and all coefficient
	denominators required by Assumption~\ref{ass:cubic-primes},
	are supported on \eqref{eq:three-cubic-S}. This includes the
	reconstructed roots, quotients and B\'ezout polynomials.
	
	\emph{Assumption~\ref{ass:cubic-coverage}: coverage.}
	The coefficients $a_i$ used to solve linear vertices have norms
	supported on $S$, so their charts cover every zero outside $S$.
	For vertex $36$ of \eqref{eq:three-cubic-2}, use
	$t=(y-1)/(x+1)$ when $x\ne-1$.
	If $\Delta_{36}(t)\equiv0\pmod\pi$, the conic equation would
	force $t\equiv1$ and $1-2\omega\equiv0\pmod\pi$. This is
	impossible outside $S$, since $\Norm(1-2\omega)=7$.
	The points with $x=-1$ are $(-1,1)$ and $(-1,1+\omega)$,
	where $f=-4,-6$, so neither is a common zero outside $S$.
	For vertex $46$ of \eqref{eq:three-cubic-3}, put $t=y+1$.
	The conic equation gives
	$t\bigl(y+(1+\omega)x\bigr)=2\omega$, so $t\ne0$ outside $S$
	and its chart covers every zero.
	
	\emph{Assumption~\ref{ass:cubic-local}: local bounds.}
	Apply the lifting procedure \eqref{eq:cubic-lifting} and the
	compatible-coordinate calculations of
	Section~\ref{sec:cubic-local}. At $2$, start with all solution
	classes modulo $2$ for \eqref{eq:three-cubic-1}, and with the
	two classes modulo $8$ in \eqref{eq:cubic-preliminary} for each
	of the other equations. At $3$, start with all solutions modulo
	$3$. Refine a class until the graph contribution is constant
	on its compatible symbol coordinates, retaining the restrictions of Section~\ref{sec:cubic-local} for vertices with integer values (character $0$ at $2$, and \eqref{eq:cubic-rational} at $3$). The complete covers give
	Table~\ref{tab:three-cubic-output}.
	
	\begin{table}[!htbp]
		\centering\small\setlength{\tabcolsep}{5pt}
		\caption{Local computations for the three cubic graphs.
			``Bound'' is the largest modulus used; classes may stop earlier.
			All contributions lie in $\mathbb F_3$.}
		\label{tab:three-cubic-output}
		\begin{tabular}{@{}lrrrccc@{}}
			\toprule
			Equation&Vertices&Edges&$p$&Bound&Terminal classes&$B_p$\\
			\midrule
			\eqref{eq:three-cubic-1}&21&54&2&4&5&1\\
			&&&3&81&30&$0,1$\\\addlinespace
			\eqref{eq:three-cubic-2}&37&123&2&8&2&2\\
			&&&3&27&9&2\\\addlinespace
			\eqref{eq:three-cubic-3}&47&180&2&8&2&0\\
			&&&3&81&30&1\\
			\bottomrule
		\end{tabular}
	\end{table}
	
	At $5,7,11,13$, the same computation gives only $B_p=0$;
	sufficient common precision bounds are $5^2,7^4,11,13^3$,
	respectively. Thus the local-bound sets are
	\[
	\begin{array}{c|cc|c}
		\text{Equation}&C_2&C_3&C_p\ (p=5,7,11,13)\\\hline
		\eqref{eq:three-cubic-1}&\{1\}&\{0,1\}&\{0\}\\
		\eqref{eq:three-cubic-2}&\{2\}&\{2\}&\{0\}\\
		\eqref{eq:three-cubic-3}&\{0\}&\{1\}&\{0\}
	\end{array}
	\]
	The preliminary quadratic restrictions hold for every integer
	solution, so their use at $2$ preserves the required local bounds.
	
	\emph{Assumption~\ref{ass:cubic-exclusion}: incompatibility.}
	For \eqref{eq:three-cubic-1}, \eqref{eq:three-cubic-2} and
	\eqref{eq:three-cubic-3}, respectively, the sets
	$\sum_{p\in S}C_p$ are $\{1,2\}$, $\{1\}$ and $\{1\}$.
	None contains zero. Theorem~\ref{prop:cubic-certificate}
	therefore excludes an integer solution of each equation.
\end{proof}

\subsection{The last equation of size 34}

In this subsection we prove that the equation
\begin{equation}\label{eq:H34last}
	y^3+xy-y=x^4-4
\end{equation}
has no integer solutions. 

For this equation, the certificate of insolubility uses three cubic graphs and two quadratic graphs. 
Their contributions must be computed together on each congruence class.
Allowing each graph to come from a different class would admit combinations
that no integer solution can produce, and would lose the contradiction.

For clarity, we state the additional quadratic formulas here. For
nonzero integers $a,b$, write $a=p^eu$, $b=p^fw$, where $p\nmid uw$.
Define $q_p(a,b)\in\mathbb F_2=\Z/2\Z$ by
\begin{equation}\label{eq:cubic-quadratic-odd}
 (-1)^{q_p(a,b)}
   =\leg{-1}{p}^{ef}\leg{u}{p}^{f}\leg{w}{p}^{e}
       \quad(p\text{ odd}).
\end{equation}
At $p=2$, with $u,w$ odd, set
\begin{equation*}
 q_2(a,b)=\frac{u-1}{2}\frac{w-1}{2}
       +e\frac{w^2-1}{8}+f\frac{u^2-1}{8}\pmod2.
\end{equation*}
All displayed quotients are integers. Finally, put $q_\infty(a,b)=1$
when both $a,b$ are negative, and $0$ otherwise. The subscript
$\infty$ records this real sign test. The quantities
$(-1)^{q_p(a,b)}$ are known as the rational quadratic Hilbert symbols, here
specified by explicit formulas. Their classical product law is
\begin{equation*}
 q_\infty(a,b)+\sum_pq_p(a,b)=0\quad\text{in }\mathbb F_2;
\end{equation*}
see \cite[Chapter~III, Theorems~1 and~3]{serre1973course}.

\paragraph{The joint form of the numbered assumptions.}
Here is the precise extension used for several graphs. Suppose
there are $r\ge0$ cubic graphs and $s\ge0$ quadratic graphs,
with $r+s\ge1$, for the same polynomial $f$. Each cubic graph
is subject to Assumptions~\ref{ass:cubic-nonzero}--\ref{ass:cubic-coverage}
with a common finite set $S$. For each quadratic graph use those
same six assumptions with $H_i\in\Z[x,y]$, $c_i\in\Z\setminus\{0\}$,
$K$ replaced by $\Q$, and weights in $\{0,1\}$. Form $\rho$
over $\mathbb F_2$: $\rho_{ij}=\rho_{ji}=e_{ij}$ for $i<j$,
$\rho_{ii}=0$, and $\widetilde\rho_{ij}\in\{0,1\}$ is its
integer representative. Replace the exponent $3$ in
\eqref{eq:cubic-vertex} by $2$. The degrees $d,h$,
chart identities, resultant identities, coverage and nonvanishing
conditions retain their stated meanings; reduction is now modulo
the ordinary prime $p$. Include $2$ in $S$. For rational coefficients
the norm convention is still $\Norm(a)=a^2$, so the prime-support
condition is unchanged.

For each quadratic graph define its contribution to be
\[
 B_{p}=\sum_{i<j}e_{ij}q_p(H_i,H_j)+\sum_iq_p(H_i,c_i)
       \quad\text{in }\mathbb F_2,
\]
and define $B_\infty$ by the same formula with $q_\infty$.
Let $\mathcal A=\mathbb F_3^r\times\mathbb F_2^s$ and let
$\mathbf B_p\in\mathcal A$ collect the contributions of all graphs
at the same integer pair $(x,y)$, in their fixed order. At infinity
put zero in the cubic coordinates and the sign contributions in
the quadratic coordinates.

\emph{Assumption~\ref{ass:cubic-local}, joint form:} specify sets
$C_p\subseteq\mathcal A$ for $p\in S$ and
$C_\infty\subseteq\{0\}^r\times\mathbb F_2^s$ such that
\begin{equation*}
 \forall(x,y)\in\Z^2,\quad f(x,y)=0\ \Longrightarrow\
 \mathbf B_p(x,y)\in C_p\ (p\in S),\quad
 \mathbf B_\infty(x,y)\in C_\infty.
\end{equation*}
\emph{Assumption~\ref{ass:cubic-exclusion}, joint form:} require
\begin{equation}\label{eq:cubic-joint-exclusion}
 \mathbf0\notin C_\infty+\sum_{p\in S}C_p.
\end{equation}
The set on the right consists of all sums $b_\infty+\sum_{p\in S}b_p$
with $b_\infty\in C_\infty$ and $b_p\in C_p$; addition is componentwise.

Under these assumptions $f=0$ has no integer solutions. Indeed,
Theorem~\ref{prop:cubic-certificate} gives zero outside $S$ in
the cubic coordinates. Its proof applies to the quadratic coordinates
as well: adjacent values cannot both be divisible by $p\notin S$,
and the square identity makes the product of the neighbouring
values and its constant a nonzero square modulo $p$. Formula
\eqref{eq:cubic-quadratic-odd} then gives zero. Summing the
reciprocity laws would put $\mathbf0$ in the set excluded by
\eqref{eq:cubic-joint-exclusion}. Thus the first six assumptions
are checked graph by graph, and the last two are checked jointly.

\paragraph{The graphs for \eqref{eq:H34last}.}\label{data:h34-last} $f=y^{3}+xy-y-x^{4}+4$.
\par\smallskip\noindent\textit{The cubic graphs $G_0$, $G_1$ and $G_2$: $K=\Q(\omega)$, $n=3$, $m=98$; the auxiliary polynomials and the constants are given in Table~\ref{tab:O1-vertices}, and the edge lists below. A vertex without edges and with constant $\langle1,0\rangle$ does not contribute to the graph.}
\begingroup\footnotesize\setlength{\tabcolsep}{4pt}\setlength{\LTpre}{3pt}\setlength{\LTpost}{6pt}\renewcommand{\arraystretch}{1.05}
\begin{longtable}{@{}rllll@{}}
\caption{The auxiliary polynomials of the cubic graphs $G_0$, $G_1$ and $G_2$ for \eqref{eq:H34last}, and their constants in the three graphs. Here $\langle A,B\rangle=A+B\omega$.}\label{tab:O1-vertices}\\
\toprule $i$ & $H_i$ & $c_i$ in $G_0$ & $c_i$ in $G_1$ & $c_i$ in $G_2$\\\midrule\endfirsthead
\multicolumn{5}{l}{\textit{Table \thetable\ continued}}\\\toprule $i$ & $H_i$ & $c_i$ in $G_0$ & $c_i$ in $G_1$ & $c_i$ in $G_2$\\\midrule\endhead\bottomrule\endfoot
0 & $\langle x - 1,0\rangle$ & $\langle 1,0\rangle$ & $\langle 1,0\rangle$ & $\langle 1,0\rangle$\\
1 & $\langle x - 2,0\rangle$ & $\langle 392,147\rangle$ & $\langle 21,14\rangle$ & $\langle 14,-7\rangle$\\
2 & $\langle x + 1,0\rangle$ & $\langle 1,0\rangle$ & $\langle 1,0\rangle$ & $\langle 1,0\rangle$\\
3 & $\langle - x + y + 2,0\rangle$ & $\langle 57,-38\rangle$ & $\langle 1,0\rangle$ & $\langle -1,-1\rangle$\\
4 & $\langle x + 2,0\rangle$ & $\langle 0,-1\rangle$ & $\langle -7,-21\rangle$ & $\langle 21,14\rangle$\\
5 & $\langle y,0\rangle$ & $\langle 0,-1\rangle$ & $\langle 1,0\rangle$ & $\langle 1,0\rangle$\\
6 & $\langle x,y + 1\rangle$ & $\langle 2,3\rangle$ & $\langle 2,0\rangle$ & $\langle 6,16\rangle$\\
7 & $\langle - x + y - 1,- 2 x - 1\rangle$ & $\langle 2,-1\rangle$ & $\langle 1,0\rangle$ & $\langle 1,0\rangle$\\
8 & $\langle - 2 x + y,1 - x\rangle$ & $\langle 1,5\rangle$ & $\langle 3,1\rangle$ & $\langle 12,0\rangle$\\
9 & $\langle - x + y,1 - x\rangle$ & $\langle 1,0\rangle$ & $\langle 2,13\rangle$ & $\langle 2,-1\rangle$\\
10 & $\langle 2 x + y,x + 1\rangle$ & $\langle 28,35\rangle$ & $\langle 3,-6\rangle$ & $\langle 4,26\rangle$\\
11 & $\langle x + y - 1,x - 1\rangle$ & $\langle 15,24\rangle$ & $\langle 2,0\rangle$ & $\langle 12,36\rangle$\\
12 & $\langle y,1\rangle$ & $\langle 4,0\rangle$ & $\langle 1,1\rangle$ & $\langle 0,3\rangle$\\
13 & $\langle x - 1,-1\rangle$ & $\langle 0,3\rangle$ & $\langle 3,1\rangle$ & $\langle 3,1\rangle$\\
14 & $\langle - 2 x + y - 1,- x - 2\rangle$ & $\langle 3,3\rangle$ & $\langle 1,1\rangle$ & $\langle 2,1\rangle$\\
15 & $\langle x + y - 1,- x - 2\rangle$ & $\langle 2,0\rangle$ & $\langle 100,-100\rangle$ & $\langle 1,0\rangle$\\
16 & $\langle y - 1,-2\rangle$ & $\langle 4,0\rangle$ & $\langle 2,0\rangle$ & $\langle 3,3\rangle$\\
17 & $\langle y + 1,2\rangle$ & $\langle 16,0\rangle$ & $\langle 4,0\rangle$ & $\langle -9,-9\rangle$\\
18 & $\langle x - 1,-2\rangle$ & $\langle 28,-14\rangle$ & $\langle 20,-80\rangle$ & $\langle 1,3\rangle$\\
19 & $\langle x + 1,2\rangle$ & $\langle 1568,980\rangle$ & $\langle 3600,9600\rangle$ & $\langle -5,3\rangle$\\
20 & $\langle x + 2,3\rangle$ & $\langle 0,3\rangle$ & $\langle 3,3\rangle$ & $\langle 1,-1\rangle$\\
21 & $\langle x - y - 2,-1\rangle$ & $\langle 12,8\rangle$ & $\langle 0,2\rangle$ & $\langle 32,12\rangle$\\
22 & $\langle - x + y - 1,x + 1\rangle$ & $\langle 3,9\rangle$ & $\langle 50,50\rangle$ & $\langle 10,2\rangle$\\
23 & $\langle x + y - 1,1\rangle$ & $\langle 3,1\rangle$ & $\langle 96,36\rangle$ & $\langle 5,8\rangle$\\
24 & $\langle x + y - 1,x + 1\rangle$ & $\langle 16,10\rangle$ & $\langle 210,140\rangle$ & $\langle 7,0\rangle$\\
25 & $\langle x - 2,-1\rangle$ & $\langle 31,30\rangle$ & $\langle 3,-5\rangle$ & $\langle 8,3\rangle$\\
26 & $\langle y - 2,- x - 2\rangle$ & $\langle 6,-12\rangle$ & $\langle 49,-49\rangle$ & $\langle 2,-4\rangle$\\
27 & $\langle x + y,x + 2\rangle$ & $\langle 180,288\rangle$ & $\langle 7203,7203\rangle$ & $\langle 20,32\rangle$\\
28 & $\langle 2 x + y,2 x + 2\rangle$ & $\langle 10,4\rangle$ & $\langle 1,0\rangle$ & $\langle 1,0\rangle$\\
29 & $\langle y - 2,-2\rangle$ & $\langle 1,0\rangle$ & $\langle 4,0\rangle$ & $\langle 4,4\rangle$\\
30 & $\langle - x + 3 y - 6,x - 6\rangle$ & $\langle 15,9\rangle$ & $\langle 940,-120\rangle$ & $\langle 1,0\rangle$\\
31 & $\langle - 2 x + 3 y,6 - x\rangle$ & $\langle -45,-189\rangle$ & $\langle 1109200,240000\rangle$ & $\langle 1,0\rangle$\\
32 & $\langle x - y + 1,1\rangle$ & $\langle 1,0\rangle$ & $\langle 2,-1\rangle$ & $\langle 1,-1\rangle$\\
33 & $\langle x + y + 1,1\rangle$ & $\langle 9,6\rangle$ & $\langle 25,-25\rangle$ & $\langle 1,2\rangle$\\
34 & $\langle x + y,-1\rangle$ & $\langle -27,-72\rangle$ & $\langle 1875,1875\rangle$ & $\langle -3,0\rangle$\\
35 & $\langle - x + y + 1,1 - 2 x\rangle$ & $\langle 4,0\rangle$ & $\langle 1,1\rangle$ & $\langle 1,0\rangle$\\
36 & $\langle x,-1\rangle$ & $\langle 9,-15\rangle$ & $\langle 2,-11\rangle$ & $\langle 7,-28\rangle$\\
37 & $\langle x + 1,1\rangle$ & $\langle 351,495\rangle$ & $\langle 48,165\rangle$ & $\langle 441,1176\rangle$\\
38 & $\langle - x + y + 1,- x - 1\rangle$ & $\langle 4,8\rangle$ & $\langle 5,-3\rangle$ & $\langle 2,3\rangle$\\
39 & $\langle x + y + 2,2 x + 1\rangle$ & $\langle 22,26\rangle$ & $\langle 50,0\rangle$ & $\langle 4,6\rangle$\\
40 & $\langle - x + y + 1,- 2 x - 1\rangle$ & $\langle -660,-468\rangle$ & $\langle 2500,0\rangle$ & $\langle -32,-12\rangle$\\
41 & $\langle y + 1,x - 2\rangle$ & $\langle 8,-4\rangle$ & $\langle 700,-350\rangle$ & $\langle 15,-9\rangle$\\
42 & $\langle - x + y + 3,2 - x\rangle$ & $\langle 128,80\rangle$ & $\langle 980000,612500\rangle$ & $\langle 495,351\rangle$\\
43 & $\langle x + y + 1,y\rangle$ & $\langle 1,0\rangle$ & $\langle 1,0\rangle$ & $\langle 1,0\rangle$\\
44 & $\langle - 2 x + y + 1,- x\rangle$ & $\langle 6,16\rangle$ & $\langle 350,525\rangle$ & $\langle 2,13\rangle$\\
45 & $\langle 2 x + y + 2,2 x - 1\rangle$ & $\langle 96,-30\rangle$ & $\langle 925,275\rangle$ & $\langle 1,0\rangle$\\
46 & $\langle x + 2,-1\rangle$ & $\langle 1,-2\rangle$ & $\langle 1,0\rangle$ & $\langle 1,0\rangle$\\
47 & $\langle x - y,- y\rangle$ & $\langle 16,10\rangle$ & $\langle 2,2\rangle$ & $\langle 2,-4\rangle$\\
48 & $\langle x^{2} + y - 3,x^{2} - 2 x\rangle$ & $\langle 2,-3\rangle$ & $\langle 1,0\rangle$ & $\langle 1,0\rangle$\\
49 & $\langle - x^{2} - 3 x + y - 6,-2\rangle$ & $\langle 21,16\rangle$ & $\langle 1,0\rangle$ & $\langle 1,0\rangle$\\
50 & $\langle - x^{2} - x + y + 1,- x^{2} - 2 x\rangle$ & $\langle 20,-64\rangle$ & $\langle 1,0\rangle$ & $\langle 1,1\rangle$\\
51 & $\langle x + y + 1,x^{2} + 2 x\rangle$ & $\langle 2960,6656\rangle$ & $\langle 1,0\rangle$ & $\langle -1,-1\rangle$\\
52 & $\langle - x + y + 3,x^{2} - 2\rangle$ & $\langle 64,-20\rangle$ & $\langle 1,1\rangle$ & $\langle 1,0\rangle$\\
53 & $\langle - x^{2} - x + y + 5,2 - x^{2}\rangle$ & $\langle 6656,2960\rangle$ & $\langle -1,-1\rangle$ & $\langle 1,0\rangle$\\
54 & $\langle - 2 x^{2} - 2 x + y,- x^{2} - 4 x\rangle$ & $\langle 48,-15\rangle$ & $\langle 1,0\rangle$ & $\langle 1,0\rangle$\\
55 & $\langle 2 x - y + 1,0\rangle$ & $\langle 14,-7\rangle$ & $\langle 392,147\rangle$ & $\langle -245,-392\rangle$\\
56 & $\langle x^{2} + y - 2,0\rangle$ & $\langle -7,-21\rangle$ & $\langle 0,-1\rangle$ & $\langle -7,-21\rangle$\\
57 & $\langle y + 2,x\rangle$ & $\langle 0,3\rangle$ & $\langle 10,-6\rangle$ & $\langle 1,-1\rangle$\\
58 & $\langle - x + y + 1,- x^{2}\rangle$ & $\langle 1,1\rangle$ & $\langle 1,0\rangle$ & $\langle 1,0\rangle$\\
59 & $\langle x^{2} - 2,- y\rangle$ & $\langle 2,3\rangle$ & $\langle 4,4\rangle$ & $\langle 0,1\rangle$\\
60 & $\langle x^{2} - y + 2,- y\rangle$ & $\langle 0,12\rangle$ & $\langle 12,0\rangle$ & $\langle 0,1\rangle$\\
61 & $\langle x^{2} + 2,y\rangle$ & $\langle 0,144\rangle$ & $\langle 144,0\rangle$ & $\langle 0,1\rangle$\\
62 & $\langle x + y,y + 1\rangle$ & $\langle 1,0\rangle$ & $\langle 2,-1\rangle$ & $\langle 11,-2\rangle$\\
63 & $\langle y,1 - x\rangle$ & $\langle -495,-144\rangle$ & $\langle 4,0\rangle$ & $\langle -720,432\rangle$\\
64 & $\langle x + y,- x\rangle$ & $\langle 3,1\rangle$ & $\langle 3,0\rangle$ & $\langle 0,4\rangle$\\
65 & $\langle x - y - 1,- y - 1\rangle$ & $\langle -8,-3\rangle$ & $\langle 4,0\rangle$ & $\langle -156,64\rangle$\\
66 & $\langle x + y,2 x + 1\rangle$ & $\langle 8,5\rangle$ & $\langle 1,0\rangle$ & $\langle 1,0\rangle$\\
67 & $\langle - x + y - 1,x - 1\rangle$ & $\langle -9,15\rangle$ & $\langle 3,-5\rangle$ & $\langle 144,0\rangle$\\
68 & $\langle y - 1,x - 1\rangle$ & $\langle 1,0\rangle$ & $\langle -48,117\rangle$ & $\langle 8,5\rangle$\\
69 & $\langle x + y - 1,- x - 1\rangle$ & $\langle -1176,-735\rangle$ & $\langle 45,72\rangle$ & $\langle -192,468\rangle$\\
70 & $\langle y - 1,-1\rangle$ & $\langle 16,0\rangle$ & $\langle -1,-1\rangle$ & $\langle 0,9\rangle$\\
71 & $\langle x,1\rangle$ & $\langle 0,9\rangle$ & $\langle 3,-5\rangle$ & $\langle 3,-5\rangle$\\
72 & $\langle - x + y + 1,x + 2\rangle$ & $\langle -9,-9\rangle$ & $\langle -1,-1\rangle$ & $\langle 0,-3\rangle$\\
73 & $\langle 2 x + y + 1,x + 2\rangle$ & $\langle 4,0\rangle$ & $\langle 30000,30000\rangle$ & $\langle 1,0\rangle$\\
74 & $\langle x - 1,-3\rangle$ & $\langle 0,9\rangle$ & $\langle -9,-9\rangle$ & $\langle 3,3\rangle$\\
75 & $\langle x - y - 1,1\rangle$ & $\langle -48,-128\rangle$ & $\langle 0,4\rangle$ & $\langle 256,-624\rangle$\\
76 & $\langle - 2 x + y - 2,- x - 1\rangle$ & $\langle -45,27\rangle$ & $\langle -2500,-2500\rangle$ & $\langle 60,-36\rangle$\\
77 & $\langle x + y - 2,-1\rangle$ & $\langle 3,-5\rangle$ & $\langle 2304,-5616\rangle$ & $\langle -55,-16\rangle$\\
78 & $\langle y - 2,- x - 1\rangle$ & $\langle -64,-220\rangle$ & $\langle -14700,-39200\rangle$ & $\langle 49,0\rangle$\\
79 & $\langle x - 1,1\rangle$ & $\langle -899,-960\rangle$ & $\langle 39,55\rangle$ & $\langle 16,-39\rangle$\\
80 & $\langle y - 2,- 2 x - 2\rangle$ & $\langle 20,-64\rangle$ & $\langle 1,0\rangle$ & $\langle 1,0\rangle$\\
81 & $\langle y,2\rangle$ & $\langle 1,0\rangle$ & $\langle 16,0\rangle$ & $\langle -16,-16\rangle$\\
82 & $\langle x - y,-1\rangle$ & $\langle 1,0\rangle$ & $\langle 8,5\rangle$ & $\langle 3,3\rangle$\\
83 & $\langle x + y,2 x - 1\rangle$ & $\langle 16,0\rangle$ & $\langle -1,-1\rangle$ & $\langle 1,0\rangle$\\
84 & $\langle y + 2,x + 1\rangle$ & $\langle -48,0\rangle$ & $\langle 55,39\rangle$ & $\langle -8,-3\rangle$\\
85 & $\langle x + 1,- y\rangle$ & $\langle 1,0\rangle$ & $\langle 1,0\rangle$ & $\langle 1,0\rangle$\\
86 & $\langle - x + y + 1,x\rangle$ & $\langle -156,64\rangle$ & $\langle -245000,-91875\rangle$ & $\langle -48,117\rangle$\\
87 & $\langle y + 3,1 - 2 x\rangle$ & $\langle 14976,6660\rangle$ & $\langle 346875,-433125\rangle$ & $\langle 1,0\rangle$\\
88 & $\langle x + 3,1\rangle$ & $\langle 5,8\rangle$ & $\langle 1,0\rangle$ & $\langle 1,0\rangle$\\
89 & $\langle x,y\rangle$ & $\langle -64,-220\rangle$ & $\langle -4,-4\rangle$ & $\langle 20,32\rangle$\\
90 & $\langle 2 x + y - 3,- x^{2} + 2 x\rangle$ & $\langle 16,21\rangle$ & $\langle 1,0\rangle$ & $\langle 1,0\rangle$\\
91 & $\langle - x^{2} - 3 x + y - 4,2\rangle$ & $\langle -231,-416\rangle$ & $\langle 1,0\rangle$ & $\langle 1,0\rangle$\\
92 & $\langle - x^{2} + 2 x + y,x^{2} + 4 x\rangle$ & $\langle 3744,1665\rangle$ & $\langle 1,0\rangle$ & $\langle 1,0\rangle$\\
93 & $\langle - x + y + 2,- x\rangle$ & $\langle 0,9\rangle$ & $\langle 220,156\rangle$ & $\langle 3,3\rangle$\\
94 & $\langle x^{2} - x + y + 1,x^{2}\rangle$ & $\langle -1,-1\rangle$ & $\langle 1,0\rangle$ & $\langle 1,0\rangle$\\
95 & $\langle x^{2} + y - 2,y\rangle$ & $\langle -8,-3\rangle$ & $\langle -16,-16\rangle$ & $\langle 0,1\rangle$\\
96 & $\langle x - 1,- y - 1\rangle$ & $\langle 1,0\rangle$ & $\langle 8,5\rangle$ & $\langle 165,48\rangle$\\
97 & $\langle 2 x + y,x\rangle$ & $\langle 3,-5\rangle$ & $\langle 9,0\rangle$ & $\langle 0,16\rangle$\\
\end{longtable}
\endgroup
\par\smallskip\noindent\textit{Weights of $G_0$: $766$ edges.}\par\begingroup\footnotesize\noindent
\edgerow{0}{6^{2}\;\allowbreak 9^{2}\;\allowbreak 16^{1}\;\allowbreak 17^{2}\;\allowbreak 32^{1}\;\allowbreak 33^{1}\;\allowbreak 34^{2}\;\allowbreak 44^{1}\;\allowbreak 58^{1}\;\allowbreak 64^{1}\;\allowbreak 65^{1}\;\allowbreak 68^{1}\;\allowbreak 82^{2}\;\allowbreak 86^{2}\;\allowbreak 94^{2}\;\allowbreak 97^{2}}
\edgerow{1}{9^{1}\;\allowbreak 11^{2}\;\allowbreak 14^{2}\;\allowbreak 23^{2}\;\allowbreak 29^{2}\;\allowbreak 30^{2}\;\allowbreak 31^{1}\;\allowbreak 39^{2}\;\allowbreak 40^{1}\;\allowbreak 43^{1}\;\allowbreak 44^{1}\;\allowbreak 48^{2}\;\allowbreak 57^{2}\;\allowbreak 59^{2}\;\allowbreak 63^{1}\;\allowbreak 64^{1}\;\allowbreak 68^{2}\;\allowbreak 72^{1}\;\allowbreak 77^{1}\;\allowbreak 81^{1}\;\allowbreak 85^{2}\;\allowbreak 86^{2}\;\allowbreak 90^{1}\;\allowbreak 93^{1}\;\allowbreak 95^{1}\;\allowbreak 97^{2}}
\edgerow{2}{23^{2}\;\allowbreak 26^{1}\;\allowbreak 27^{2}\;\allowbreak 28^{2}\;\allowbreak 33^{2}\;\allowbreak 34^{1}\;\allowbreak 38^{1}\;\allowbreak 47^{1}\;\allowbreak 54^{1}\;\allowbreak 59^{2}\;\allowbreak 60^{1}\;\allowbreak 61^{2}\;\allowbreak 62^{1}\;\allowbreak 77^{1}\;\allowbreak 80^{1}\;\allowbreak 84^{2}\;\allowbreak 89^{2}\;\allowbreak 92^{2}\;\allowbreak 95^{1}\;\allowbreak 96^{2}}
\edgerow{3}{6^{1}\;\allowbreak 8^{1}\;\allowbreak 18^{1}\;\allowbreak 19^{2}\;\allowbreak 20^{2}\;\allowbreak 22^{1}\;\allowbreak 29^{1}\;\allowbreak 36^{2}\;\allowbreak 37^{1}\;\allowbreak 41^{1}\;\allowbreak 42^{2}\;\allowbreak 43^{1}\;\allowbreak 52^{2}\;\allowbreak 53^{1}\;\allowbreak 54^{2}\;\allowbreak 57^{1}\;\allowbreak 60^{2}\;\allowbreak 61^{1}\;\allowbreak 62^{1}\;\allowbreak 65^{2}\;\allowbreak 67^{2}\;\allowbreak 74^{1}\;\allowbreak 76^{2}\;\allowbreak 81^{2}\;\allowbreak 85^{2}\;\allowbreak 92^{1}\;\allowbreak 93^{2}\;\allowbreak 96^{2}}
\edgerow{4}{14^{2}\;\allowbreak 15^{1}\;\allowbreak 16^{2}\;\allowbreak 17^{1}\;\allowbreak 22^{1}\;\allowbreak 30^{2}\;\allowbreak 31^{1}\;\allowbreak 41^{1}\;\allowbreak 42^{2}\;\allowbreak 49^{1}\;\allowbreak 50^{2}\;\allowbreak 51^{1}\;\allowbreak 64^{1}\;\allowbreak 72^{1}\;\allowbreak 73^{2}\;\allowbreak 76^{2}\;\allowbreak 91^{2}\;\allowbreak 97^{2}}
\edgerow{5}{9^{2}\;\allowbreak 10^{1}\;\allowbreak 11^{1}\;\allowbreak 23^{1}\;\allowbreak 26^{1}\;\allowbreak 27^{2}\;\allowbreak 43^{1}\;\allowbreak 52^{2}\;\allowbreak 53^{1}\;\allowbreak 54^{2}\;\allowbreak 63^{2}\;\allowbreak 68^{1}\;\allowbreak 69^{2}\;\allowbreak 77^{2}\;\allowbreak 85^{2}\;\allowbreak 92^{1}}
\edgerow{6}{7^{2}\;\allowbreak 15^{1}\;\allowbreak 17^{1}\;\allowbreak 19^{1}\;\allowbreak 20^{1}\;\allowbreak 40^{2}\;\allowbreak 44^{1}\;\allowbreak 46^{1}\;\allowbreak 48^{2}\;\allowbreak 51^{2}\;\allowbreak 53^{1}\;\allowbreak 55^{1}\;\allowbreak 57^{1}\;\allowbreak 58^{2}\;\allowbreak 59^{1}\;\allowbreak 62^{1}}
\edgerow{7}{8^{1}\;\allowbreak 16^{2}\;\allowbreak 19^{2}\;\allowbreak 30^{2}\;\allowbreak 41^{2}\;\allowbreak 43^{1}\;\allowbreak 50^{1}\;\allowbreak 54^{2}}
\edgerow{8}{13^{1}\;\allowbreak 16^{1}\;\allowbreak 18^{2}\;\allowbreak 19^{1}\;\allowbreak 21^{2}\;\allowbreak 26^{1}\;\allowbreak 29^{2}\;\allowbreak 32^{1}\;\allowbreak 35^{2}\;\allowbreak 42^{2}\;\allowbreak 49^{2}\;\allowbreak 53^{2}}
\edgerow{9}{13^{1}\;\allowbreak 14^{1}\;\allowbreak 16^{2}\;\allowbreak 18^{1}\;\allowbreak 21^{2}\;\allowbreak 24^{1}\;\allowbreak 27^{2}\;\allowbreak 34^{1}\;\allowbreak 36^{2}\;\allowbreak 37^{2}\;\allowbreak 40^{1}\;\allowbreak 44^{2}\;\allowbreak 47^{2}\;\allowbreak 51^{1}\;\allowbreak 53^{1}\;\allowbreak 55^{1}\;\allowbreak 57^{2}\;\allowbreak 59^{2}\;\allowbreak 61^{2}\;\allowbreak 63^{1}}
\edgerow{10}{13^{2}\;\allowbreak 19^{2}\;\allowbreak 20^{1}\;\allowbreak 24^{1}\;\allowbreak 25^{2}\;\allowbreak 30^{1}\;\allowbreak 31^{1}\;\allowbreak 33^{1}\;\allowbreak 34^{2}\;\allowbreak 36^{1}\;\allowbreak 61^{2}\;\allowbreak 64^{1}}
\edgerow{11}{15^{2}\;\allowbreak 18^{1}\;\allowbreak 32^{2}\;\allowbreak 39^{2}\;\allowbreak 60^{1}\;\allowbreak 68^{1}\;\allowbreak 72^{2}\;\allowbreak 75^{1}\;\allowbreak 77^{2}\;\allowbreak 81^{2}\;\allowbreak 82^{2}\;\allowbreak 84^{2}\;\allowbreak 95^{1}}
\edgerow{12}{13^{1}\;\allowbreak 25^{2}\;\allowbreak 32^{1}\;\allowbreak 36^{2}\;\allowbreak 38^{1}\;\allowbreak 40^{2}\;\allowbreak 42^{1}\;\allowbreak 45^{1}\;\allowbreak 47^{2}\;\allowbreak 50^{1}\;\allowbreak 56^{1}\;\allowbreak 58^{1}\;\allowbreak 60^{2}\;\allowbreak 61^{2}\;\allowbreak 62^{2}}
\edgerow{13}{21^{1}\;\allowbreak 23^{2}\;\allowbreak 31^{1}\;\allowbreak 35^{2}\;\allowbreak 39^{2}\;\allowbreak 40^{1}\;\allowbreak 47^{1}\;\allowbreak 59^{1}\;\allowbreak 60^{1}\;\allowbreak 61^{1}\;\allowbreak 62^{2}\;\allowbreak 64^{2}}
\edgerow{14}{26^{2}\;\allowbreak 36^{2}\;\allowbreak 41^{2}\;\allowbreak 50^{2}\;\allowbreak 59^{1}\;\allowbreak 60^{1}\;\allowbreak 61^{2}\;\allowbreak 63^{2}}
\edgerow{15}{16^{1}\;\allowbreak 17^{2}\;\allowbreak 18^{2}\;\allowbreak 22^{1}\;\allowbreak 27^{1}\;\allowbreak 33^{1}\;\allowbreak 55^{2}\;\allowbreak 59^{2}\;\allowbreak 61^{2}}
\edgerow{16}{18^{2}\;\allowbreak 19^{2}\;\allowbreak 22^{2}\;\allowbreak 24^{2}\;\allowbreak 33^{1}\;\allowbreak 40^{2}\;\allowbreak 41^{2}\;\allowbreak 42^{2}\;\allowbreak 47^{2}\;\allowbreak 51^{2}\;\allowbreak 53^{2}\;\allowbreak 56^{1}\;\allowbreak 58^{2}\;\allowbreak 61^{2}\;\allowbreak 64^{2}\;\allowbreak 65^{1}\;\allowbreak 73^{2}\;\allowbreak 75^{2}\;\allowbreak 77^{1}\;\allowbreak 95^{1}\;\allowbreak 97^{1}}
\edgerow{17}{18^{1}\;\allowbreak 19^{1}\;\allowbreak 21^{1}\;\allowbreak 23^{2}\;\allowbreak 34^{2}\;\allowbreak 39^{1}\;\allowbreak 41^{1}\;\allowbreak 42^{1}\;\allowbreak 50^{1}\;\allowbreak 52^{1}\;\allowbreak 56^{2}\;\allowbreak 59^{2}\;\allowbreak 60^{1}\;\allowbreak 64^{2}\;\allowbreak 66^{2}\;\allowbreak 67^{1}\;\allowbreak 68^{2}\;\allowbreak 73^{1}\;\allowbreak 76^{1}\;\allowbreak 78^{1}\;\allowbreak 89^{1}\;\allowbreak 94^{1}\;\allowbreak 97^{1}}
\edgerow{18}{21^{1}\;\allowbreak 23^{2}\;\allowbreak 24^{1}\;\allowbreak 30^{1}\;\allowbreak 31^{2}\;\allowbreak 33^{1}\;\allowbreak 35^{1}\;\allowbreak 38^{2}\;\allowbreak 43^{2}\;\allowbreak 47^{2}\;\allowbreak 51^{1}\;\allowbreak 58^{2}\;\allowbreak 62^{1}\;\allowbreak 64^{2}\;\allowbreak 65^{1}\;\allowbreak 66^{2}\;\allowbreak 67^{1}\;\allowbreak 69^{2}\;\allowbreak 77^{2}\;\allowbreak 81^{2}\;\allowbreak 83^{2}\;\allowbreak 89^{2}\;\allowbreak 91^{2}\;\allowbreak 96^{1}\;\allowbreak 97^{2}}
\edgerow{19}{23^{1}\;\allowbreak 29^{1}\;\allowbreak 30^{1}\;\allowbreak 31^{2}\;\allowbreak 34^{2}\;\allowbreak 35^{1}\;\allowbreak 47^{1}\;\allowbreak 49^{1}\;\allowbreak 50^{2}\;\allowbreak 62^{2}\;\allowbreak 63^{1}\;\allowbreak 64^{1}\;\allowbreak 67^{2}\;\allowbreak 68^{1}\;\allowbreak 73^{2}\;\allowbreak 75^{2}\;\allowbreak 77^{1}\;\allowbreak 78^{2}\;\allowbreak 83^{2}\;\allowbreak 84^{1}\;\allowbreak 85^{1}\;\allowbreak 89^{1}\;\allowbreak 94^{1}\;\allowbreak 96^{2}\;\allowbreak 97^{1}}
\edgerow{20}{22^{1}\;\allowbreak 30^{1}\;\allowbreak 32^{2}\;\allowbreak 34^{2}\;\allowbreak 43^{2}\;\allowbreak 60^{1}\;\allowbreak 62^{1}}
\edgerow{21}{25^{2}\;\allowbreak 34^{2}\;\allowbreak 36^{2}\;\allowbreak 37^{1}\;\allowbreak 43^{2}\;\allowbreak 44^{1}\;\allowbreak 53^{2}\;\allowbreak 60^{2}\;\allowbreak 63^{1}}
\edgerow{22}{27^{2}\;\allowbreak 31^{1}\;\allowbreak 32^{2}\;\allowbreak 34^{1}\;\allowbreak 40^{1}\;\allowbreak 51^{1}\;\allowbreak 55^{2}\;\allowbreak 57^{1}\;\allowbreak 60^{1}\;\allowbreak 61^{1}\;\allowbreak 62^{2}\;\allowbreak 64^{1}}
\edgerow{23}{25^{1}\;\allowbreak 35^{1}\;\allowbreak 37^{1}\;\allowbreak 40^{2}\;\allowbreak 46^{1}\;\allowbreak 49^{2}\;\allowbreak 51^{2}\;\allowbreak 52^{2}\;\allowbreak 56^{2}\;\allowbreak 59^{2}\;\allowbreak 60^{2}\;\allowbreak 61^{1}\;\allowbreak 63^{2}}
\edgerow{24}{37^{2}\;\allowbreak 38^{2}\;\allowbreak 51^{1}\;\allowbreak 56^{1}\;\allowbreak 60^{1}}
\edgerow{25}{34^{2}\;\allowbreak 39^{1}\;\allowbreak 40^{2}\;\allowbreak 41^{1}\;\allowbreak 43^{2}\;\allowbreak 47^{2}\;\allowbreak 48^{2}\;\allowbreak 55^{1}\;\allowbreak 57^{1}\;\allowbreak 58^{1}\;\allowbreak 60^{2}\;\allowbreak 61^{2}\;\allowbreak 62^{2}}
\edgerow{26}{28^{1}\;\allowbreak 32^{1}\;\allowbreak 38^{2}\;\allowbreak 40^{2}\;\allowbreak 55^{2}\;\allowbreak 57^{1}\;\allowbreak 60^{2}\;\allowbreak 68^{2}\;\allowbreak 73^{1}\;\allowbreak 76^{2}\;\allowbreak 88^{1}\;\allowbreak 91^{2}\;\allowbreak 96^{1}\;\allowbreak 97^{2}}
\edgerow{27}{39^{1}\;\allowbreak 46^{2}\;\allowbreak 49^{1}\;\allowbreak 55^{1}\;\allowbreak 61^{1}\;\allowbreak 62^{2}\;\allowbreak 64^{1}\;\allowbreak 67^{1}\;\allowbreak 72^{2}\;\allowbreak 80^{2}\;\allowbreak 82^{2}\;\allowbreak 84^{1}\;\allowbreak 93^{2}}
\edgerow{28}{47^{2}\;\allowbreak 54^{2}}
\edgerow{29}{34^{1}\;\allowbreak 36^{1}\;\allowbreak 37^{2}\;\allowbreak 48^{1}\;\allowbreak 49^{2}\;\allowbreak 56^{1}\;\allowbreak 59^{2}\;\allowbreak 61^{2}\;\allowbreak 63^{2}}
\edgerow{30}{39^{1}\;\allowbreak 40^{2}\;\allowbreak 41^{1}\;\allowbreak 43^{1}\;\allowbreak 53^{1}\;\allowbreak 55^{2}\;\allowbreak 56^{1}\;\allowbreak 61^{1}\;\allowbreak 69^{1}\;\allowbreak 71^{1}\;\allowbreak 76^{1}\;\allowbreak 85^{2}\;\allowbreak 87^{1}\;\allowbreak 89^{2}}
\edgerow{31}{39^{1}\;\allowbreak 40^{2}\;\allowbreak 42^{2}\;\allowbreak 43^{1}\;\allowbreak 45^{2}\;\allowbreak 47^{1}\;\allowbreak 52^{2}\;\allowbreak 55^{1}\;\allowbreak 56^{2}\;\allowbreak 60^{2}\;\allowbreak 66^{2}\;\allowbreak 69^{1}\;\allowbreak 74^{1}\;\allowbreak 85^{2}}
\edgerow{32}{35^{1}\;\allowbreak 46^{2}\;\allowbreak 51^{2}\;\allowbreak 60^{1}\;\allowbreak 61^{2}\;\allowbreak 62^{1}\;\allowbreak 63^{2}}
\edgerow{33}{38^{1}\;\allowbreak 39^{1}\;\allowbreak 47^{1}\;\allowbreak 50^{2}\;\allowbreak 55^{1}\;\allowbreak 56^{1}\;\allowbreak 58^{2}\;\allowbreak 59^{1}\;\allowbreak 68^{1}\;\allowbreak 69^{2}\;\allowbreak 74^{2}\;\allowbreak 75^{2}\;\allowbreak 76^{1}\;\allowbreak 79^{2}\;\allowbreak 81^{1}}
\edgerow{34}{40^{2}\;\allowbreak 51^{1}\;\allowbreak 55^{2}\;\allowbreak 56^{2}\;\allowbreak 69^{1}\;\allowbreak 73^{1}\;\allowbreak 84^{2}\;\allowbreak 89^{2}\;\allowbreak 94^{1}\;\allowbreak 95^{2}}
\edgerow{35}{37^{2}\;\allowbreak 46^{1}\;\allowbreak 47^{1}\;\allowbreak 48^{1}\;\allowbreak 56^{2}}
\edgerow{36}{48^{2}\;\allowbreak 49^{1}\;\allowbreak 57^{2}\;\allowbreak 60^{1}\;\allowbreak 64^{2}\;\allowbreak 68^{2}\;\allowbreak 75^{1}\;\allowbreak 77^{1}\;\allowbreak 78^{2}\;\allowbreak 81^{2}\;\allowbreak 83^{2}\;\allowbreak 85^{1}\;\allowbreak 95^{2}}
\edgerow{37}{43^{2}\;\allowbreak 59^{1}\;\allowbreak 61^{2}\;\allowbreak 68^{2}\;\allowbreak 69^{1}\;\allowbreak 70^{2}\;\allowbreak 72^{2}\;\allowbreak 75^{2}\;\allowbreak 81^{1}\;\allowbreak 90^{1}\;\allowbreak 91^{2}\;\allowbreak 93^{1}\;\allowbreak 97^{1}}
\edgerow{38}{40^{1}\;\allowbreak 47^{2}\;\allowbreak 48^{2}\;\allowbreak 63^{2}}
\edgerow{39}{40^{2}\;\allowbreak 46^{2}\;\allowbreak 56^{1}\;\allowbreak 57^{1}\;\allowbreak 58^{1}\;\allowbreak 61^{1}\;\allowbreak 65^{2}\;\allowbreak 68^{1}\;\allowbreak 70^{2}\;\allowbreak 71^{1}\;\allowbreak 76^{1}\;\allowbreak 77^{2}\;\allowbreak 79^{2}\;\allowbreak 84^{1}\;\allowbreak 86^{1}\;\allowbreak 96^{2}\;\allowbreak 97^{1}}
\edgerow{40}{44^{2}\;\allowbreak 56^{2}\;\allowbreak 60^{2}\;\allowbreak 62^{1}\;\allowbreak 63^{2}\;\allowbreak 64^{2}\;\allowbreak 71^{2}\;\allowbreak 79^{1}\;\allowbreak 88^{1}\;\allowbreak 93^{2}\;\allowbreak 94^{2}}
\edgerow{41}{43^{2}\;\allowbreak 52^{2}\;\allowbreak 55^{1}\;\allowbreak 64^{1}\;\allowbreak 67^{2}\;\allowbreak 70^{1}\;\allowbreak 90^{2}}
\edgerow{42}{48^{1}\;\allowbreak 53^{1}\;\allowbreak 55^{2}\;\allowbreak 66^{2}\;\allowbreak 72^{2}\;\allowbreak 79^{1}\;\allowbreak 85^{1}\;\allowbreak 97^{2}}
\edgerow{44}{55^{2}\;\allowbreak 62^{1}\;\allowbreak 64^{2}}
\edgerow{46}{49^{2}\;\allowbreak 53^{1}\;\allowbreak 56^{1}}
\edgerow{47}{56^{1}\;\allowbreak 60^{2}\;\allowbreak 61^{2}}
\edgerow{48}{51^{2}\;\allowbreak 56^{2}\;\allowbreak 58^{1}}
\edgerow{49}{54^{1}}
\edgerow{50}{53^{1}\;\allowbreak 59^{1}\;\allowbreak 64^{2}\;\allowbreak 65^{2}\;\allowbreak 68^{1}\;\allowbreak 76^{1}\;\allowbreak 77^{2}\;\allowbreak 78^{1}\;\allowbreak 82^{2}\;\allowbreak 90^{2}\;\allowbreak 93^{2}\;\allowbreak 96^{2}}
\edgerow{51}{52^{2}\;\allowbreak 57^{1}\;\allowbreak 62^{1}\;\allowbreak 66^{1}\;\allowbreak 70^{1}\;\allowbreak 72^{2}\;\allowbreak 95^{2}\;\allowbreak 97^{1}}
\edgerow{52}{53^{2}\;\allowbreak 62^{2}\;\allowbreak 65^{1}\;\allowbreak 67^{2}\;\allowbreak 68^{1}\;\allowbreak 75^{2}\;\allowbreak 88^{1}\;\allowbreak 93^{2}\;\allowbreak 97^{2}}
\edgerow{53}{57^{1}\;\allowbreak 64^{1}\;\allowbreak 77^{2}\;\allowbreak 96^{1}}
\edgerow{55}{57^{2}\;\allowbreak 65^{1}\;\allowbreak 68^{1}\;\allowbreak 73^{2}\;\allowbreak 76^{2}\;\allowbreak 79^{1}\;\allowbreak 86^{2}\;\allowbreak 93^{1}}
\edgerow{56}{70^{1}\;\allowbreak 77^{2}\;\allowbreak 78^{1}\;\allowbreak 81^{1}\;\allowbreak 83^{2}\;\allowbreak 88^{1}\;\allowbreak 89^{1}\;\allowbreak 90^{2}}
\edgerow{57}{60^{1}}
\edgerow{59}{63^{1}}
\edgerow{60}{62^{1}\;\allowbreak 64^{1}\;\allowbreak 68^{2}\;\allowbreak 69^{2}\;\allowbreak 70^{2}\;\allowbreak 71^{1}\;\allowbreak 72^{2}\;\allowbreak 73^{2}\;\allowbreak 76^{1}\;\allowbreak 77^{1}\;\allowbreak 79^{2}\;\allowbreak 81^{2}\;\allowbreak 82^{2}\;\allowbreak 89^{2}\;\allowbreak 97^{1}}
\edgerow{61}{63^{1}\;\allowbreak 64^{2}\;\allowbreak 70^{2}\;\allowbreak 71^{1}\;\allowbreak 72^{1}\;\allowbreak 74^{1}\;\allowbreak 75^{2}\;\allowbreak 76^{1}\;\allowbreak 77^{2}\;\allowbreak 78^{1}\;\allowbreak 79^{2}\;\allowbreak 82^{1}\;\allowbreak 89^{2}\;\allowbreak 93^{1}\;\allowbreak 96^{2}\;\allowbreak 97^{2}}
\edgerow{63}{73^{1}\;\allowbreak 82^{1}}
\edgerow{65}{66^{1}\;\allowbreak 73^{2}\;\allowbreak 74^{2}\;\allowbreak 86^{2}\;\allowbreak 88^{2}\;\allowbreak 90^{1}\;\allowbreak 93^{2}\;\allowbreak 94^{1}\;\allowbreak 95^{2}\;\allowbreak 96^{2}}
\edgerow{66}{67^{2}\;\allowbreak 85^{2}\;\allowbreak 92^{1}}
\edgerow{67}{71^{2}\;\allowbreak 75^{1}\;\allowbreak 81^{1}\;\allowbreak 82^{2}\;\allowbreak 83^{1}\;\allowbreak 91^{1}}
\edgerow{68}{71^{2}\;\allowbreak 72^{2}\;\allowbreak 75^{1}\;\allowbreak 78^{2}\;\allowbreak 86^{1}\;\allowbreak 89^{1}\;\allowbreak 93^{1}\;\allowbreak 95^{1}}
\edgerow{69}{71^{1}\;\allowbreak 74^{2}\;\allowbreak 78^{2}\;\allowbreak 79^{1}\;\allowbreak 97^{2}}
\edgerow{70}{71^{2}\;\allowbreak 79^{1}\;\allowbreak 82^{2}\;\allowbreak 84^{2}\;\allowbreak 87^{2}\;\allowbreak 89^{1}\;\allowbreak 94^{2}\;\allowbreak 96^{1}}
\edgerow{71}{75^{2}\;\allowbreak 77^{1}\;\allowbreak 83^{1}\;\allowbreak 89^{2}\;\allowbreak 95^{2}\;\allowbreak 96^{1}\;\allowbreak 97^{1}}
\edgerow{72}{95^{2}}
\edgerow{73}{76^{2}\;\allowbreak 95^{1}}
\edgerow{74}{76^{2}\;\allowbreak 82^{1}\;\allowbreak 85^{1}\;\allowbreak 96^{2}}
\edgerow{75}{79^{1}\;\allowbreak 85^{1}\;\allowbreak 86^{2}}
\edgerow{76}{82^{1}\;\allowbreak 93^{2}\;\allowbreak 96^{1}\;\allowbreak 97^{2}}
\edgerow{77}{79^{2}\;\allowbreak 83^{2}\;\allowbreak 88^{2}\;\allowbreak 91^{1}\;\allowbreak 95^{1}}
\edgerow{78}{84^{1}}
\edgerow{79}{85^{1}\;\allowbreak 89^{1}\;\allowbreak 90^{1}\;\allowbreak 93^{2}\;\allowbreak 94^{2}\;\allowbreak 96^{1}}
\edgerow{80}{89^{1}\;\allowbreak 92^{1}}
\edgerow{81}{90^{2}\;\allowbreak 91^{1}\;\allowbreak 95^{1}}
\edgerow{82}{83^{2}\;\allowbreak 88^{1}\;\allowbreak 96^{2}}
\edgerow{83}{88^{2}\;\allowbreak 89^{2}\;\allowbreak 90^{2}}
\edgerow{84}{89^{1}\;\allowbreak 90^{1}}
\edgerow{86}{96^{2}\;\allowbreak 97^{1}}
\edgerow{88}{91^{1}}
\edgerow{90}{94^{2}}
\edgerow{91}{92^{2}}
\par
\endgroup\par\smallskip\noindent\textit{Weights of $G_1$: $623$ edges.}\par\begingroup\footnotesize\noindent
\edgerow{0}{6^{1}\;\allowbreak 12^{1}\;\allowbreak 14^{2}\;\allowbreak 15^{1}\;\allowbreak 16^{1}\;\allowbreak 17^{2}\;\allowbreak 24^{1}\;\allowbreak 39^{2}\;\allowbreak 40^{1}\;\allowbreak 41^{1}\;\allowbreak 42^{2}\;\allowbreak 44^{1}\;\allowbreak 60^{2}\;\allowbreak 61^{1}\;\allowbreak 62^{1}\;\allowbreak 65^{2}\;\allowbreak 70^{2}\;\allowbreak 72^{1}\;\allowbreak 73^{2}\;\allowbreak 78^{2}\;\allowbreak 86^{2}\;\allowbreak 96^{2}}
\edgerow{1}{6^{1}\;\allowbreak 8^{1}\;\allowbreak 9^{2}\;\allowbreak 14^{2}\;\allowbreak 23^{2}\;\allowbreak 26^{1}\;\allowbreak 27^{2}\;\allowbreak 30^{2}\;\allowbreak 31^{1}\;\allowbreak 33^{1}\;\allowbreak 34^{2}\;\allowbreak 39^{1}\;\allowbreak 40^{2}\;\allowbreak 44^{2}\;\allowbreak 47^{2}\;\allowbreak 60^{1}\;\allowbreak 61^{2}\;\allowbreak 65^{2}\;\allowbreak 67^{2}\;\allowbreak 68^{1}\;\allowbreak 72^{1}\;\allowbreak 77^{1}\;\allowbreak 86^{1}\;\allowbreak 89^{1}}
\edgerow{2}{6^{2}\;\allowbreak 10^{1}\;\allowbreak 21^{1}\;\allowbreak 23^{2}\;\allowbreak 26^{2}\;\allowbreak 27^{1}\;\allowbreak 38^{1}\;\allowbreak 39^{1}\;\allowbreak 40^{2}\;\allowbreak 47^{2}\;\allowbreak 57^{1}\;\allowbreak 59^{1}\;\allowbreak 60^{2}\;\allowbreak 61^{1}\;\allowbreak 65^{1}\;\allowbreak 69^{2}\;\allowbreak 75^{2}\;\allowbreak 77^{1}\;\allowbreak 84^{2}\;\allowbreak 89^{1}\;\allowbreak 93^{2}\;\allowbreak 95^{2}}
\edgerow{3}{8^{2}\;\allowbreak 18^{1}\;\allowbreak 19^{2}\;\allowbreak 20^{2}\;\allowbreak 23^{2}\;\allowbreak 29^{2}\;\allowbreak 36^{1}\;\allowbreak 37^{2}\;\allowbreak 57^{1}\;\allowbreak 67^{1}\;\allowbreak 74^{1}\;\allowbreak 77^{1}\;\allowbreak 81^{1}\;\allowbreak 93^{2}}
\edgerow{4}{26^{1}\;\allowbreak 27^{2}\;\allowbreak 30^{2}\;\allowbreak 31^{1}\;\allowbreak 39^{1}\;\allowbreak 40^{2}\;\allowbreak 43^{1}\;\allowbreak 45^{2}\;\allowbreak 47^{1}\;\allowbreak 59^{2}\;\allowbreak 85^{2}\;\allowbreak 87^{1}\;\allowbreak 89^{2}\;\allowbreak 95^{1}}
\edgerow{5}{6^{2}\;\allowbreak 9^{1}\;\allowbreak 10^{2}\;\allowbreak 22^{1}\;\allowbreak 23^{1}\;\allowbreak 25^{1}\;\allowbreak 26^{1}\;\allowbreak 27^{2}\;\allowbreak 36^{2}\;\allowbreak 37^{1}\;\allowbreak 41^{1}\;\allowbreak 42^{2}\;\allowbreak 43^{1}\;\allowbreak 45^{2}\;\allowbreak 57^{2}\;\allowbreak 64^{2}\;\allowbreak 65^{1}\;\allowbreak 68^{2}\;\allowbreak 69^{1}\;\allowbreak 76^{2}\;\allowbreak 77^{2}\;\allowbreak 79^{2}\;\allowbreak 85^{2}\;\allowbreak 87^{1}\;\allowbreak 93^{1}\;\allowbreak 97^{1}}
\edgerow{6}{8^{1}\;\allowbreak 18^{2}\;\allowbreak 19^{2}\;\allowbreak 20^{1}\;\allowbreak 32^{2}\;\allowbreak 43^{1}\;\allowbreak 44^{1}\;\allowbreak 47^{2}\;\allowbreak 55^{1}\;\allowbreak 60^{2}\;\allowbreak 63^{2}}
\edgerow{8}{9^{1}\;\allowbreak 13^{1}\;\allowbreak 16^{2}\;\allowbreak 18^{2}\;\allowbreak 25^{2}\;\allowbreak 26^{1}\;\allowbreak 29^{1}\;\allowbreak 36^{2}\;\allowbreak 38^{2}\;\allowbreak 42^{1}\;\allowbreak 60^{2}}
\edgerow{9}{10^{2}\;\allowbreak 13^{2}\;\allowbreak 14^{1}\;\allowbreak 16^{1}\;\allowbreak 17^{2}\;\allowbreak 19^{1}\;\allowbreak 24^{2}\;\allowbreak 27^{1}\;\allowbreak 32^{1}\;\allowbreak 36^{1}\;\allowbreak 37^{1}\;\allowbreak 40^{2}\;\allowbreak 45^{2}\;\allowbreak 55^{2}\;\allowbreak 59^{2}\;\allowbreak 60^{2}\;\allowbreak 61^{1}}
\edgerow{10}{13^{1}\;\allowbreak 16^{1}\;\allowbreak 20^{2}\;\allowbreak 25^{1}\;\allowbreak 30^{1}\;\allowbreak 33^{2}\;\allowbreak 34^{2}\;\allowbreak 37^{1}\;\allowbreak 41^{2}\;\allowbreak 45^{2}\;\allowbreak 59^{2}\;\allowbreak 61^{1}\;\allowbreak 64^{2}}
\edgerow{11}{13^{1}\;\allowbreak 15^{1}\;\allowbreak 18^{2}\;\allowbreak 19^{1}\;\allowbreak 21^{2}\;\allowbreak 22^{2}\;\allowbreak 25^{1}\;\allowbreak 26^{1}\;\allowbreak 32^{2}\;\allowbreak 39^{2}\;\allowbreak 45^{1}\;\allowbreak 59^{1}\;\allowbreak 60^{1}\;\allowbreak 61^{2}\;\allowbreak 62^{2}\;\allowbreak 63^{1}\;\allowbreak 65^{2}\;\allowbreak 72^{1}\;\allowbreak 74^{2}\;\allowbreak 95^{1}\;\allowbreak 97^{2}}
\edgerow{12}{13^{2}\;\allowbreak 20^{1}\;\allowbreak 24^{2}\;\allowbreak 33^{2}\;\allowbreak 36^{2}\;\allowbreak 37^{2}\;\allowbreak 40^{1}\;\allowbreak 42^{1}}
\edgerow{13}{16^{1}\;\allowbreak 22^{1}\;\allowbreak 34^{1}\;\allowbreak 40^{2}\;\allowbreak 47^{2}\;\allowbreak 53^{2}\;\allowbreak 55^{1}\;\allowbreak 56^{2}\;\allowbreak 59^{2}\;\allowbreak 60^{1}\;\allowbreak 64^{1}}
\edgerow{14}{16^{2}\;\allowbreak 19^{1}\;\allowbreak 21^{1}\;\allowbreak 27^{1}\;\allowbreak 32^{1}\;\allowbreak 55^{2}\;\allowbreak 59^{1}\;\allowbreak 60^{1}\;\allowbreak 61^{1}\;\allowbreak 63^{1}}
\edgerow{15}{17^{2}\;\allowbreak 22^{2}\;\allowbreak 27^{1}\;\allowbreak 33^{2}\;\allowbreak 34^{1}\;\allowbreak 37^{2}\;\allowbreak 41^{1}\;\allowbreak 42^{1}\;\allowbreak 45^{2}\;\allowbreak 55^{2}\;\allowbreak 59^{1}\;\allowbreak 61^{1}}
\edgerow{16}{19^{2}\;\allowbreak 24^{1}\;\allowbreak 25^{2}\;\allowbreak 26^{2}\;\allowbreak 30^{1}\;\allowbreak 33^{2}\;\allowbreak 36^{2}\;\allowbreak 37^{2}\;\allowbreak 38^{2}\;\allowbreak 40^{2}\;\allowbreak 41^{1}\;\allowbreak 43^{1}\;\allowbreak 44^{2}\;\allowbreak 57^{1}\;\allowbreak 59^{1}\;\allowbreak 60^{2}\;\allowbreak 68^{2}\;\allowbreak 73^{2}\;\allowbreak 76^{2}\;\allowbreak 78^{2}}
\edgerow{17}{18^{1}\;\allowbreak 22^{1}\;\allowbreak 24^{1}\;\allowbreak 27^{1}\;\allowbreak 31^{2}\;\allowbreak 34^{1}\;\allowbreak 36^{1}\;\allowbreak 37^{1}\;\allowbreak 39^{1}\;\allowbreak 42^{2}\;\allowbreak 61^{1}\;\allowbreak 67^{2}\;\allowbreak 68^{1}\;\allowbreak 69^{1}\;\allowbreak 71^{1}\;\allowbreak 72^{2}\;\allowbreak 78^{2}\;\allowbreak 79^{1}\;\allowbreak 84^{1}\;\allowbreak 85^{2}\;\allowbreak 86^{1}\;\allowbreak 93^{2}\;\allowbreak 95^{2}}
\edgerow{18}{21^{2}\;\allowbreak 23^{1}\;\allowbreak 30^{2}\;\allowbreak 31^{2}\;\allowbreak 38^{1}\;\allowbreak 39^{2}\;\allowbreak 41^{2}\;\allowbreak 43^{1}\;\allowbreak 53^{2}\;\allowbreak 56^{2}\;\allowbreak 60^{2}\;\allowbreak 61^{1}\;\allowbreak 62^{1}\;\allowbreak 63^{1}\;\allowbreak 65^{2}\;\allowbreak 68^{1}\;\allowbreak 72^{1}\;\allowbreak 81^{1}\;\allowbreak 86^{1}\;\allowbreak 87^{1}\;\allowbreak 89^{1}\;\allowbreak 96^{1}}
\edgerow{19}{29^{2}\;\allowbreak 30^{1}\;\allowbreak 31^{1}\;\allowbreak 40^{1}\;\allowbreak 42^{1}\;\allowbreak 44^{2}\;\allowbreak 45^{2}\;\allowbreak 47^{2}\;\allowbreak 52^{1}\;\allowbreak 56^{1}\;\allowbreak 60^{2}\;\allowbreak 61^{1}\;\allowbreak 62^{2}\;\allowbreak 63^{2}\;\allowbreak 65^{2}\;\allowbreak 67^{2}\;\allowbreak 75^{1}\;\allowbreak 77^{2}\;\allowbreak 84^{2}\;\allowbreak 85^{2}\;\allowbreak 96^{2}}
\edgerow{20}{22^{1}\;\allowbreak 30^{1}\;\allowbreak 32^{2}\;\allowbreak 55^{2}\;\allowbreak 60^{2}\;\allowbreak 63^{2}}
\edgerow{21}{36^{2}\;\allowbreak 39^{1}\;\allowbreak 41^{2}\;\allowbreak 42^{1}\;\allowbreak 43^{1}\;\allowbreak 45^{2}\;\allowbreak 59^{1}\;\allowbreak 60^{1}\;\allowbreak 61^{2}}
\edgerow{22}{24^{1}\;\allowbreak 31^{2}\;\allowbreak 32^{2}\;\allowbreak 39^{1}\;\allowbreak 41^{1}\;\allowbreak 57^{2}\;\allowbreak 59^{1}\;\allowbreak 61^{2}\;\allowbreak 62^{2}}
\edgerow{23}{25^{2}\;\allowbreak 27^{2}\;\allowbreak 31^{2}\;\allowbreak 35^{2}\;\allowbreak 40^{2}\;\allowbreak 43^{2}\;\allowbreak 47^{2}\;\allowbreak 60^{1}}
\edgerow{24}{37^{1}\;\allowbreak 44^{1}\;\allowbreak 56^{2}\;\allowbreak 62^{2}}
\edgerow{25}{39^{1}\;\allowbreak 40^{1}\;\allowbreak 41^{2}\;\allowbreak 55^{1}\;\allowbreak 57^{2}\;\allowbreak 60^{2}\;\allowbreak 61^{2}\;\allowbreak 62^{2}}
\edgerow{26}{29^{1}\;\allowbreak 30^{1}\;\allowbreak 32^{1}\;\allowbreak 33^{2}\;\allowbreak 38^{1}\;\allowbreak 40^{1}\;\allowbreak 45^{2}\;\allowbreak 55^{2}\;\allowbreak 64^{2}\;\allowbreak 68^{1}\;\allowbreak 72^{1}\;\allowbreak 73^{1}\;\allowbreak 77^{2}\;\allowbreak 85^{1}\;\allowbreak 87^{2}}
\edgerow{27}{31^{2}\;\allowbreak 34^{1}\;\allowbreak 39^{2}\;\allowbreak 43^{2}\;\allowbreak 45^{1}\;\allowbreak 55^{1}\;\allowbreak 63^{1}\;\allowbreak 67^{1}\;\allowbreak 81^{2}\;\allowbreak 82^{2}\;\allowbreak 84^{2}\;\allowbreak 87^{1}\;\allowbreak 97^{1}}
\edgerow{29}{37^{1}\;\allowbreak 38^{2}\;\allowbreak 56^{2}}
\edgerow{30}{32^{1}\;\allowbreak 39^{1}\;\allowbreak 43^{2}\;\allowbreak 61^{1}\;\allowbreak 76^{2}\;\allowbreak 77^{2}\;\allowbreak 85^{1}\;\allowbreak 87^{1}\;\allowbreak 89^{2}}
\edgerow{31}{40^{2}\;\allowbreak 43^{2}\;\allowbreak 45^{2}\;\allowbreak 47^{1}\;\allowbreak 60^{2}\;\allowbreak 69^{1}\;\allowbreak 74^{1}\;\allowbreak 82^{2}\;\allowbreak 85^{1}}
\edgerow{32}{35^{1}\;\allowbreak 55^{2}\;\allowbreak 60^{1}\;\allowbreak 61^{2}}
\edgerow{33}{38^{2}\;\allowbreak 41^{2}\;\allowbreak 55^{2}\;\allowbreak 59^{2}\;\allowbreak 61^{1}\;\allowbreak 69^{2}\;\allowbreak 71^{1}\;\allowbreak 73^{1}\;\allowbreak 93^{1}}
\edgerow{34}{42^{1}\;\allowbreak 55^{1}\;\allowbreak 57^{2}\;\allowbreak 60^{2}\;\allowbreak 69^{2}\;\allowbreak 70^{2}\;\allowbreak 73^{2}\;\allowbreak 84^{1}\;\allowbreak 95^{1}}
\edgerow{36}{47^{2}\;\allowbreak 68^{1}\;\allowbreak 69^{1}\;\allowbreak 70^{2}\;\allowbreak 73^{2}\;\allowbreak 78^{1}\;\allowbreak 81^{1}\;\allowbreak 85^{1}\;\allowbreak 86^{1}}
\edgerow{37}{43^{2}\;\allowbreak 44^{2}\;\allowbreak 67^{2}\;\allowbreak 68^{1}\;\allowbreak 70^{2}\;\allowbreak 75^{2}\;\allowbreak 89^{1}}
\edgerow{38}{40^{2}\;\allowbreak 61^{1}}
\edgerow{39}{40^{2}\;\allowbreak 41^{1}\;\allowbreak 45^{2}\;\allowbreak 47^{2}\;\allowbreak 57^{2}\;\allowbreak 61^{1}\;\allowbreak 68^{2}\;\allowbreak 70^{1}\;\allowbreak 71^{2}\;\allowbreak 77^{2}\;\allowbreak 79^{1}\;\allowbreak 84^{2}\;\allowbreak 95^{1}}
\edgerow{40}{42^{2}\;\allowbreak 59^{2}\;\allowbreak 60^{2}\;\allowbreak 63^{2}\;\allowbreak 75^{1}\;\allowbreak 76^{1}\;\allowbreak 79^{1}\;\allowbreak 87^{1}\;\allowbreak 89^{1}\;\allowbreak 93^{1}}
\edgerow{41}{42^{1}\;\allowbreak 45^{1}\;\allowbreak 55^{1}\;\allowbreak 64^{2}\;\allowbreak 67^{1}\;\allowbreak 70^{1}\;\allowbreak 73^{1}\;\allowbreak 75^{1}}
\edgerow{42}{55^{2}\;\allowbreak 69^{2}\;\allowbreak 73^{1}\;\allowbreak 75^{2}\;\allowbreak 76^{1}\;\allowbreak 79^{2}\;\allowbreak 87^{2}\;\allowbreak 97^{1}}
\edgerow{43}{45^{2}}
\edgerow{45}{47^{2}\;\allowbreak 52^{2}}
\edgerow{47}{61^{1}\;\allowbreak 62^{1}\;\allowbreak 64^{1}}
\edgerow{52}{71^{2}}
\edgerow{53}{87^{2}}
\edgerow{55}{57^{1}\;\allowbreak 65^{1}\;\allowbreak 68^{2}\;\allowbreak 71^{1}\;\allowbreak 72^{2}\;\allowbreak 73^{2}\;\allowbreak 74^{2}\;\allowbreak 79^{1}\;\allowbreak 82^{2}\;\allowbreak 93^{2}}
\edgerow{56}{62^{1}\;\allowbreak 71^{2}\;\allowbreak 78^{2}\;\allowbreak 81^{2}\;\allowbreak 96^{2}}
\edgerow{57}{59^{1}}
\edgerow{59}{61^{2}\;\allowbreak 62^{1}\;\allowbreak 63^{1}}
\edgerow{60}{62^{1}\;\allowbreak 63^{2}\;\allowbreak 68^{1}\;\allowbreak 69^{1}\;\allowbreak 72^{1}\;\allowbreak 73^{1}\;\allowbreak 75^{2}\;\allowbreak 76^{2}\;\allowbreak 79^{2}\;\allowbreak 82^{2}\;\allowbreak 84^{1}\;\allowbreak 89^{1}\;\allowbreak 95^{2}\;\allowbreak 96^{1}\;\allowbreak 97^{2}}
\edgerow{61}{62^{2}\;\allowbreak 63^{1}\;\allowbreak 64^{1}\;\allowbreak 65^{2}\;\allowbreak 67^{2}\;\allowbreak 68^{2}\;\allowbreak 71^{1}\;\allowbreak 72^{1}\;\allowbreak 74^{2}\;\allowbreak 75^{1}\;\allowbreak 77^{1}\;\allowbreak 79^{2}\;\allowbreak 82^{1}\;\allowbreak 96^{2}}
\edgerow{63}{64^{1}\;\allowbreak 71^{2}\;\allowbreak 73^{2}\;\allowbreak 75^{1}\;\allowbreak 76^{1}\;\allowbreak 79^{2}\;\allowbreak 82^{1}\;\allowbreak 87^{2}\;\allowbreak 95^{2}\;\allowbreak 96^{1}}
\edgerow{65}{67^{2}\;\allowbreak 74^{2}\;\allowbreak 82^{1}\;\allowbreak 85^{2}\;\allowbreak 86^{2}\;\allowbreak 89^{1}}
\edgerow{67}{68^{2}\;\allowbreak 71^{2}\;\allowbreak 79^{1}\;\allowbreak 81^{2}\;\allowbreak 84^{1}}
\edgerow{68}{69^{1}\;\allowbreak 71^{1}\;\allowbreak 72^{2}\;\allowbreak 78^{1}\;\allowbreak 82^{2}\;\allowbreak 87^{1}\;\allowbreak 95^{1}}
\edgerow{69}{71^{2}\;\allowbreak 74^{1}\;\allowbreak 79^{2}\;\allowbreak 87^{1}\;\allowbreak 95^{1}\;\allowbreak 97^{1}}
\edgerow{70}{71^{1}\;\allowbreak 74^{2}\;\allowbreak 78^{1}}
\edgerow{71}{76^{2}\;\allowbreak 89^{1}\;\allowbreak 95^{1}\;\allowbreak 97^{2}}
\edgerow{72}{75^{2}\;\allowbreak 82^{2}\;\allowbreak 95^{2}}
\edgerow{73}{76^{1}\;\allowbreak 87^{1}\;\allowbreak 95^{2}}
\edgerow{74}{76^{2}\;\allowbreak 82^{1}}
\edgerow{75}{85^{2}\;\allowbreak 87^{1}\;\allowbreak 95^{2}}
\edgerow{76}{78^{2}\;\allowbreak 82^{1}\;\allowbreak 93^{1}\;\allowbreak 95^{2}\;\allowbreak 96^{1}}
\edgerow{77}{79^{1}\;\allowbreak 83^{1}\;\allowbreak 85^{1}\;\allowbreak 89^{1}}
\edgerow{78}{86^{2}\;\allowbreak 96^{1}}
\edgerow{79}{93^{1}\;\allowbreak 96^{1}}
\edgerow{81}{84^{1}}
\edgerow{82}{83^{2}}
\edgerow{85}{87^{1}}
\edgerow{87}{89^{1}}
\edgerow{89}{96^{2}\;\allowbreak 97^{2}}
\edgerow{93}{95^{2}}
\edgerow{95}{96^{2}}
\par
\endgroup\par\smallskip\noindent\textit{Weights of $G_2$: $412$ edges.}\par\begingroup\footnotesize\noindent
\edgerow{0}{6^{1}\;\allowbreak 11^{2}\;\allowbreak 16^{1}\;\allowbreak 17^{2}\;\allowbreak 23^{1}\;\allowbreak 32^{1}\;\allowbreak 59^{1}\;\allowbreak 60^{2}\;\allowbreak 61^{1}\;\allowbreak 63^{1}\;\allowbreak 65^{2}\;\allowbreak 77^{2}\;\allowbreak 82^{2}\;\allowbreak 95^{2}}
\edgerow{1}{8^{2}\;\allowbreak 11^{1}\;\allowbreak 22^{1}\;\allowbreak 39^{1}\;\allowbreak 40^{2}\;\allowbreak 47^{2}\;\allowbreak 57^{1}\;\allowbreak 59^{1}\;\allowbreak 63^{2}\;\allowbreak 64^{2}\;\allowbreak 67^{1}\;\allowbreak 76^{2}\;\allowbreak 89^{1}\;\allowbreak 93^{2}\;\allowbreak 95^{2}\;\allowbreak 97^{1}}
\edgerow{2}{12^{1}\;\allowbreak 57^{2}\;\allowbreak 59^{2}\;\allowbreak 60^{2}\;\allowbreak 61^{1}\;\allowbreak 70^{2}\;\allowbreak 93^{1}\;\allowbreak 95^{1}}
\edgerow{3}{8^{2}\;\allowbreak 18^{1}\;\allowbreak 19^{2}\;\allowbreak 20^{2}\;\allowbreak 23^{2}\;\allowbreak 29^{2}\;\allowbreak 36^{1}\;\allowbreak 37^{2}\;\allowbreak 62^{1}\;\allowbreak 67^{1}\;\allowbreak 74^{1}\;\allowbreak 77^{1}\;\allowbreak 81^{1}\;\allowbreak 96^{2}}
\edgerow{4}{16^{1}\;\allowbreak 17^{2}\;\allowbreak 26^{1}\;\allowbreak 27^{2}\;\allowbreak 32^{2}\;\allowbreak 39^{1}\;\allowbreak 40^{2}\;\allowbreak 47^{2}\;\allowbreak 50^{1}\;\allowbreak 51^{2}\;\allowbreak 59^{1}\;\allowbreak 64^{2}\;\allowbreak 82^{1}\;\allowbreak 89^{1}\;\allowbreak 95^{2}\;\allowbreak 97^{1}}
\edgerow{5}{8^{1}\;\allowbreak 10^{1}\;\allowbreak 47^{1}\;\allowbreak 64^{2}\;\allowbreak 67^{2}\;\allowbreak 69^{2}\;\allowbreak 89^{2}\;\allowbreak 97^{1}}
\edgerow{6}{8^{2}\;\allowbreak 13^{2}\;\allowbreak 16^{2}\;\allowbreak 17^{2}\;\allowbreak 21^{2}\;\allowbreak 26^{2}\;\allowbreak 32^{1}\;\allowbreak 33^{2}\;\allowbreak 36^{2}\;\allowbreak 40^{2}\;\allowbreak 55^{1}\;\allowbreak 57^{2}\;\allowbreak 60^{2}\;\allowbreak 62^{1}\;\allowbreak 63^{2}}
\edgerow{8}{10^{1}\;\allowbreak 16^{1}\;\allowbreak 21^{2}\;\allowbreak 26^{2}\;\allowbreak 29^{1}\;\allowbreak 32^{1}\;\allowbreak 36^{1}\;\allowbreak 38^{1}\;\allowbreak 60^{1}}
\edgerow{9}{16^{2}\;\allowbreak 18^{1}\;\allowbreak 19^{1}\;\allowbreak 24^{1}\;\allowbreak 27^{1}\;\allowbreak 34^{2}\;\allowbreak 36^{2}\;\allowbreak 37^{2}\;\allowbreak 40^{1}\;\allowbreak 47^{1}\;\allowbreak 55^{2}\;\allowbreak 60^{2}\;\allowbreak 61^{2}}
\edgerow{10}{13^{1}\;\allowbreak 23^{1}\;\allowbreak 61^{1}\;\allowbreak 64^{2}}
\edgerow{11}{19^{1}\;\allowbreak 22^{1}\;\allowbreak 32^{1}\;\allowbreak 37^{1}\;\allowbreak 55^{1}\;\allowbreak 60^{2}\;\allowbreak 65^{2}\;\allowbreak 72^{1}\;\allowbreak 77^{2}}
\edgerow{12}{13^{2}\;\allowbreak 19^{2}\;\allowbreak 20^{1}\;\allowbreak 23^{1}\;\allowbreak 38^{1}\;\allowbreak 40^{1}\;\allowbreak 44^{1}\;\allowbreak 56^{1}\;\allowbreak 60^{2}}
\edgerow{13}{17^{2}\;\allowbreak 34^{1}\;\allowbreak 40^{2}\;\allowbreak 55^{2}\;\allowbreak 60^{1}\;\allowbreak 62^{2}}
\edgerow{14}{19^{1}\;\allowbreak 26^{1}\;\allowbreak 27^{2}\;\allowbreak 32^{1}\;\allowbreak 36^{2}\;\allowbreak 39^{1}\;\allowbreak 40^{1}\;\allowbreak 63^{1}\;\allowbreak 64^{2}}
\edgerow{16}{20^{2}\;\allowbreak 22^{1}\;\allowbreak 24^{2}\;\allowbreak 25^{2}\;\allowbreak 26^{1}\;\allowbreak 36^{1}\;\allowbreak 38^{1}\;\allowbreak 51^{1}\;\allowbreak 57^{2}\;\allowbreak 62^{1}\;\allowbreak 65^{2}\;\allowbreak 71^{2}\;\allowbreak 75^{1}\;\allowbreak 78^{1}\;\allowbreak 95^{2}\;\allowbreak 96^{1}}
\edgerow{17}{21^{2}\;\allowbreak 24^{2}\;\allowbreak 27^{2}\;\allowbreak 37^{2}\;\allowbreak 50^{2}\;\allowbreak 59^{1}\;\allowbreak 62^{2}\;\allowbreak 65^{2}\;\allowbreak 67^{1}\;\allowbreak 68^{2}\;\allowbreak 74^{1}\;\allowbreak 76^{2}\;\allowbreak 78^{1}\;\allowbreak 79^{1}\;\allowbreak 84^{2}\;\allowbreak 93^{1}\;\allowbreak 96^{2}}
\edgerow{18}{21^{2}\;\allowbreak 23^{1}\;\allowbreak 24^{1}\;\allowbreak 32^{1}\;\allowbreak 34^{2}\;\allowbreak 38^{2}\;\allowbreak 39^{2}\;\allowbreak 56^{2}\;\allowbreak 60^{2}\;\allowbreak 61^{2}\;\allowbreak 62^{1}\;\allowbreak 63^{1}\;\allowbreak 64^{1}\;\allowbreak 68^{1}\;\allowbreak 70^{2}\;\allowbreak 72^{1}\;\allowbreak 78^{1}\;\allowbreak 81^{1}\;\allowbreak 89^{2}}
\edgerow{19}{24^{2}\;\allowbreak 29^{2}\;\allowbreak 33^{1}\;\allowbreak 40^{1}\;\allowbreak 47^{1}\;\allowbreak 56^{1}\;\allowbreak 60^{1}\;\allowbreak 61^{1}\;\allowbreak 68^{1}\;\allowbreak 75^{1}\;\allowbreak 77^{2}\;\allowbreak 78^{2}\;\allowbreak 82^{2}\;\allowbreak 84^{1}\;\allowbreak 96^{2}\;\allowbreak 97^{2}}
\edgerow{20}{22^{1}\;\allowbreak 23^{1}\;\allowbreak 33^{1}\;\allowbreak 39^{1}\;\allowbreak 61^{1}\;\allowbreak 62^{2}}
\edgerow{21}{25^{2}\;\allowbreak 60^{1}}
\edgerow{22}{51^{2}\;\allowbreak 55^{2}\;\allowbreak 62^{1}\;\allowbreak 64^{1}}
\edgerow{23}{25^{2}\;\allowbreak 46^{1}\;\allowbreak 56^{2}\;\allowbreak 63^{2}}
\edgerow{24}{32^{1}\;\allowbreak 37^{2}\;\allowbreak 44^{2}\;\allowbreak 56^{2}\;\allowbreak 60^{1}}
\edgerow{25}{60^{2}}
\edgerow{26}{38^{2}\;\allowbreak 40^{1}\;\allowbreak 47^{1}\;\allowbreak 57^{2}\;\allowbreak 59^{2}\;\allowbreak 60^{2}\;\allowbreak 68^{1}\;\allowbreak 72^{2}\;\allowbreak 97^{2}}
\edgerow{27}{39^{2}\;\allowbreak 61^{1}\;\allowbreak 64^{1}\;\allowbreak 65^{2}\;\allowbreak 67^{2}\;\allowbreak 72^{1}\;\allowbreak 84^{1}\;\allowbreak 89^{2}\;\allowbreak 93^{1}\;\allowbreak 95^{1}}
\edgerow{29}{37^{1}\;\allowbreak 38^{2}\;\allowbreak 56^{2}\;\allowbreak 61^{1}}
\edgerow{32}{34^{1}\;\allowbreak 51^{2}\;\allowbreak 56^{2}\;\allowbreak 59^{1}\;\allowbreak 60^{1}\;\allowbreak 61^{2}}
\edgerow{33}{37^{2}\;\allowbreak 38^{1}\;\allowbreak 42^{2}\;\allowbreak 56^{2}\;\allowbreak 64^{2}\;\allowbreak 68^{2}\;\allowbreak 71^{1}\;\allowbreak 82^{1}}
\edgerow{34}{36^{1}\;\allowbreak 41^{1}\;\allowbreak 56^{1}\;\allowbreak 65^{2}\;\allowbreak 74^{1}\;\allowbreak 84^{2}\;\allowbreak 97^{1}}
\edgerow{36}{40^{2}\;\allowbreak 60^{2}\;\allowbreak 61^{2}\;\allowbreak 63^{1}\;\allowbreak 68^{2}\;\allowbreak 78^{2}\;\allowbreak 81^{1}\;\allowbreak 96^{2}}
\edgerow{37}{39^{1}\;\allowbreak 60^{1}\;\allowbreak 61^{1}\;\allowbreak 62^{1}\;\allowbreak 65^{2}\;\allowbreak 67^{1}\;\allowbreak 68^{2}\;\allowbreak 72^{2}}
\edgerow{38}{40^{2}\;\allowbreak 61^{2}}
\edgerow{39}{40^{1}\;\allowbreak 47^{1}\;\allowbreak 60^{2}\;\allowbreak 65^{2}\;\allowbreak 68^{1}\;\allowbreak 70^{1}\;\allowbreak 71^{2}\;\allowbreak 72^{1}\;\allowbreak 84^{2}}
\edgerow{40}{61^{1}\;\allowbreak 72^{1}\;\allowbreak 74^{1}\;\allowbreak 89^{2}}
\edgerow{41}{57^{2}\;\allowbreak 64^{1}}
\edgerow{42}{93^{1}\;\allowbreak 97^{2}}
\edgerow{46}{50^{1}}
\edgerow{47}{60^{1}\;\allowbreak 61^{1}\;\allowbreak 64^{2}}
\edgerow{50}{51^{2}\;\allowbreak 56^{1}\;\allowbreak 59^{2}\;\allowbreak 64^{1}\;\allowbreak 76^{2}\;\allowbreak 82^{2}\;\allowbreak 96^{1}}
\edgerow{51}{56^{2}\;\allowbreak 62^{2}\;\allowbreak 88^{1}\;\allowbreak 95^{1}\;\allowbreak 97^{2}}
\edgerow{55}{57^{2}\;\allowbreak 63^{1}\;\allowbreak 65^{1}\;\allowbreak 68^{2}\;\allowbreak 71^{2}\;\allowbreak 76^{2}\;\allowbreak 93^{1}}
\edgerow{56}{62^{2}\;\allowbreak 70^{1}\;\allowbreak 77^{2}\;\allowbreak 78^{2}\;\allowbreak 81^{2}\;\allowbreak 82^{2}\;\allowbreak 96^{1}}
\edgerow{59}{61^{1}}
\edgerow{60}{68^{2}\;\allowbreak 69^{1}\;\allowbreak 74^{1}\;\allowbreak 81^{1}\;\allowbreak 82^{2}\;\allowbreak 84^{2}\;\allowbreak 89^{1}\;\allowbreak 95^{1}\;\allowbreak 97^{2}}
\edgerow{61}{63^{2}\;\allowbreak 64^{1}\;\allowbreak 65^{2}\;\allowbreak 67^{1}\;\allowbreak 68^{2}\;\allowbreak 70^{2}\;\allowbreak 71^{1}\;\allowbreak 75^{1}\;\allowbreak 78^{1}\;\allowbreak 79^{2}\;\allowbreak 82^{1}\;\allowbreak 89^{1}}
\edgerow{62}{64^{1}}
\edgerow{63}{76^{2}\;\allowbreak 82^{2}}
\edgerow{65}{67^{1}\;\allowbreak 71^{1}\;\allowbreak 75^{1}\;\allowbreak 82^{2}\;\allowbreak 93^{1}\;\allowbreak 96^{2}}
\edgerow{67}{69^{2}\;\allowbreak 75^{1}\;\allowbreak 81^{2}\;\allowbreak 82^{2}\;\allowbreak 84^{2}}
\edgerow{68}{78^{2}\;\allowbreak 89^{2}}
\edgerow{69}{71^{2}\;\allowbreak 77^{2}\;\allowbreak 97^{1}}
\edgerow{70}{71^{1}\;\allowbreak 74^{2}\;\allowbreak 77^{2}\;\allowbreak 84^{2}\;\allowbreak 86^{2}}
\edgerow{71}{96^{1}}
\edgerow{72}{82^{2}\;\allowbreak 97^{1}}
\edgerow{74}{76^{2}\;\allowbreak 77^{2}\;\allowbreak 96^{1}}
\edgerow{75}{79^{1}}
\edgerow{76}{96^{2}\;\allowbreak 97^{2}}
\edgerow{77}{79^{1}\;\allowbreak 88^{2}}
\edgerow{78}{82^{2}\;\allowbreak 86^{1}}
\edgerow{81}{84^{1}}
\edgerow{82}{95^{2}}
\edgerow{89}{97^{1}}
\edgerow{96}{97^{2}}
\par
\endgroup\par\smallskip\noindent\textit{The graph $G_3$: $K=\Q,\ n=2,\ m=14$, $23$ edges.}
\begingroup\footnotesize\setlength{\tabcolsep}{3pt}\setlength{\LTpre}{3pt}\setlength{\LTpost}{3pt}\renewcommand{\arraystretch}{1.05}
\begin{longtable}{@{}rll@{\hspace{12pt}}rll@{}}
\toprule $i$ & $H_i$ & $c_i$ & $i$ & $H_i$ & $c_i$\\\midrule\endfirsthead
\multicolumn{6}{l}{\textit{Vertices (continued)}}\\\toprule $i$ & $H_i$ & $c_i$ & $i$ & $H_i$ & $c_i$\\\midrule\endhead\bottomrule\endfoot
0 & $y - 1$ & $1$ & 7 & $2 x + y + 7$ & $38$\\
1 & $x - 1$ & $-3$ & 8 & $2 x - y - 2$ & $-1$\\
2 & $2 x - 2 y - 2$ & $1$ & 9 & $2 x^{2} - 2 x - y - 1$ & $-1$\\
3 & $x - 2$ & $1$ & 10 & $x^{2} + x - y + 1$ & $1$\\
4 & $2 x + 1$ & $-3$ & 11 & $x^{2} - x - y - 1$ & $2$\\
5 & $x + y + 2$ & $1$ & 12 & $x^{2} - 3 x - 2 y - 1$ & $-1$\\
6 & $4 x + y + 3$ & $19$ & 13 & $x + y^{2} - 1$ & $1$\\
\end{longtable}\addtocounter{table}{-1}
\par\noindent\edgerow{0}{2\;\allowbreak 3\;\allowbreak 7}
\edgerow{1}{5\;\allowbreak 8\;\allowbreak 9\;\allowbreak 12\;\allowbreak 13}
\edgerow{2}{8\;\allowbreak 11\;\allowbreak 13}
\edgerow{3}{5\;\allowbreak 9}
\edgerow{4}{7\;\allowbreak 8\;\allowbreak 11}
\edgerow{5}{10\;\allowbreak 11}
\edgerow{6}{9\;\allowbreak 11}
\edgerow{8}{11\;\allowbreak 12}
\edgerow{10}{13}
\par\endgroup
\par\smallskip\noindent\textit{The graph $G_4$: $K=\Q,\ n=2$, one vertex} $2 x^{2} + 2 y - 3$ \textit{with constant} $2$\textit{, no edges.}

\begin{proposition}\label{prop:h34-last}
Equation \eqref{eq:H34last} has no integer solutions.
\end{proposition}
\begin{proof}[Proof outline]
Put $f=y^3+xy-y-x^4+4$. Use the three cubic graphs
$G_0,G_1,G_2$ and the two quadratic graphs $G_3,G_4$ just specified.
Their numbers of vertices and edges are
\[
\begin{array}{c|rrrrr}
 &G_0&G_1&G_2&G_3&G_4\\\hline
\text{vertices}&98&98&98&14&1\\
\text{edges}&766&623&412&23&0
\end{array}
\]
 Set $d=4$, $h=2$ and
\begin{equation}\label{eq:h34-S}
 S=\{2,3,5,7,11,13,19,31,37,43,61,67,73,97,109,379\}.
\end{equation}
We check the first six assumptions for each graph, and the last two jointly.

\emph{Assumption~\ref{ass:cubic-nonzero}: nonzero auxiliary values.}
Each listed vertex can be written $H_i=a_i z+b_i(t)$, where
$\{z,t\}=\{x,y\}$ and $a_i$ is a nonzero constant. Use $z=y,t=x$
when this is possible, and otherwise $z=x,t=y$.
Substitute $z=-b_i(t)/a_i$ into $f$ and write the result as
$u_i(t)+\omega v_i(t)$ with $u_i,v_i\in\Q[t]$; in a quadratic
graph $v_i=0$. Any integer common zero of $f,H_i$ would give an
integer root of the greatest common divisor of $u_i,v_i$.
Testing these integer roots, with the second coordinate required
to be integral, leaves none, for all $98$ cubic vertices and all
$15$ quadratic vertices. An identically zero component is omitted
from this greatest common divisor. There are no excluded parameter
values, since every $a_i$ is constant and nonzero.
For $G_4$ nonvanishing also follows immediately because
$2x^2+2y-3$ is odd at every integer pair.

\emph{Assumption~\ref{ass:cubic-charts}: parametrizations.}
Take the other coordinate to be $t$, solve $z=-b_i(t)/a_i$,
and use $\Delta_i=1$. These are polynomial charts with coefficients
in $K$ (in $\Q$ for $G_3,G_4$). Substitution into $H_i$ gives zero,
and the coordinate $t$ is nonconstant.

\emph{Assumption~\ref{ass:cubic-identities}: power identities.}
Compute $F_i=f\circ\sigma_i$ and $G_{ij}=H_j\circ\sigma_i$;
all $F_i$ are nonzero. For each cubic graph recover $W_i$ by
\eqref{eq:cubic-recover} and the factorization procedure in
Section~\ref{sec:cubic-local}. For a quadratic graph use
$Z^2-A_i$ and the same procedure with square roots. Retain a root
choice whose coefficient denominators are supported on $S$.
Since $\Delta_i=1$, take $n_i=0$ and compute
\[
 Q_i=\frac{W_i^{n}-c_i\prod_jG_{ij}^{\widetilde\rho_{ij}}}{F_i},
 \qquad n=3\text{ or }2.
\]
The remainder is zero for every vertex of every graph. This checks
$294$ cube identities and $15$ square identities, including the
identities at isolated vertices.

\emph{Assumption~\ref{ass:cubic-resultants}: edge resultants.}
For every printed edge compute $R_{ij}=\Res_t(F_i,G_{ij})$.
All are nonzero. Computing both orientations gives the same
nonvanishing check; only $i<j$ is required by the assumption.
Obtain the B\'ezout polynomials from the Sylvester adjugate.

\emph{Assumption~\ref{ass:cubic-primes}: exceptional primes.}
The norms of all constants $c_i$, all elimination coefficients $a_i$
and all edge resultants are supported on \eqref{eq:h34-S}.
So are the coefficient denominators in the charts, the reconstructed
$W_i$, and their quotients $Q_i$. The adjugate construction introduces
no new denominators beyond those of $F_i,G_{ij}$.
Thus all parts of Assumption~\ref{ass:cubic-primes} hold; in particular $2\in S$, as the joint form requires.

\emph{Assumption~\ref{ass:cubic-coverage}: coverage.}
Outside $S$ every $a_i$ is invertible modulo the prime under
consideration. The equation $a_i z+b_i(t)=0$ therefore gives
exactly the displayed chart, for every zero of $H_i$. In particular
it represents every common zero with $f$. Its denominator is $1$,
so no omitted points need checking. The first six assumptions now
give $\mathbf B_p=\mathbf0$ for every ordinary prime $p\notin S$.

\emph{Assumption~\ref{ass:cubic-local}: finite local bounds.}
Put
\[
 \mathbf B_p=(B^{(0)}_p,B^{(1)}_p,B^{(2)}_p;B^{(3)}_p,B^{(4)}_p)
 \in\mathcal A=\mathbb F_3^3\times\mathbb F_2^2,
\]
where the superscript specifies the graph. All coordinates are
computed at the same pair $(x,y)$.
For each $p\in S$, start with every solution of $f=0$ modulo $p$
and refine by \eqref{eq:cubic-lifting}. Compute compatible cubic
coordinates by Section~\ref{sec:cubic-local}, using each unknown
coordinate consistently throughout all three graphs.
Conjugation of a polynomial means conjugating its coefficients and
leaving $x,y$ fixed. Whenever two listed vertices are conjugates,
retain this relation between their coordinates. At an inert prime
it sends $(v,\chi)$ to $(v,-\chi)$; at a split prime it exchanges
the two prime factors and negates the character exponents. At $3$
it sends
\[
 (v,k,m,\ell)\longmapsto(v,2v-k,m,m-\ell).
\]
For a vertex with rational coefficients also retain
\eqref{eq:cubic-rational}. These relations are read directly from
the displayed polynomials, without a separately chosen pairing
of vertex indices.
For a quadratic vertex at an odd prime use its valuation modulo
$2$ and the Legendre character of its unit part. At $2$ use its
valuation modulo $2$ and odd unit modulo $8$. If the value is
zero at the current precision, allow every compatible valuation
and unit class. Refine a class unless all its possible joint
contributions belong to the specified set $C_p$. This is a finite
cover calculation, retaining every lift that satisfies $f=0$ at
the new precision.

The resulting sets $C_p$ are given in Table~\ref{tab:h34-local}.
In this table $(abc;de)$ means $(a,b,c;d,e)$, and
$\mathbf0=(000;00)$.
\begin{center}\small
\setlength{\tabcolsep}{4pt}
\begin{longtable}{@{}crrl@{}}
\caption{Complete joint local covers for \eqref{eq:H34last}.
The bound is the largest modulus used; terminal classes may stop earlier.}
\label{tab:h34-local}\\
\toprule $p$ & Bound & Classes & $C_p\subseteq\mathcal A$\\\midrule\endfirsthead
\toprule $p$ & Bound & Classes & $C_p\subseteq\mathcal A$\\\midrule\endhead
\bottomrule\endfoot
$2$ & $2^{12}$ & 60 & $\{(200;00),(200;01),(200;11),(212;01),(212;10)\}$\\[2pt]
$3$ & $3^{4}$ & 54 & $\{(011;10),(100;10),(200;10),(211;10),(222;10)\}$\\[2pt]
$5$ & $5^{3}$ & 41 & $\{(000;00),(010;00)\}$\\[2pt]
$7$ & $7^{7}$ & 519 & $\{(121;00)\}$\\[2pt]
$11$ & $11^{3}$ & 145 & $\{\mathbf0\}$\\[2pt]
$13$ & $13^{4}$ & 175 & $\{\mathbf0\}$\\[2pt]
$19$ & $19^{4}$ & 444 & $\{(010;00),(010;10),(110;00),(210;00),(210;10)\}$\\[2pt]
$31$ & $31^{3}$ & 452 & $\{\mathbf0\}$\\[2pt]
$37$ & $37^{2}$ & 221 & $\{\mathbf0\}$\\[2pt]
$43$ & $43^{2}$ & 155 & $\{\mathbf0\}$\\[2pt]
$61$ & $61^{2}$ & 244 & $\{\mathbf0\}$\\[2pt]
$67$ & $67^{2}$ & 153 & $\{\mathbf0\}$\\[2pt]
$73$ & $73^{2}$ & 160 & $\{\mathbf0\}$\\[2pt]
$97$ & $97^{2}$ & 394 & $\{\mathbf0\}$\\[2pt]
$109$ & $109^{2}$ & 225 & $\{\mathbf0\}$\\[2pt]
$379$ & $379^{2}$ & 769 & $\{\mathbf0\}$\\[2pt]
\end{longtable}
\end{center}

\emph{Assumption~\ref{ass:cubic-local}: the real place.}
For $x=-1,0,1$, the polynomials in $y$ obtained from $f$ have no
integer root. For $|x|\ge2$, any integer solution has $y>0$.
Indeed, if $x=-n\le-2$ and $y=-z\le-1$, a positive left side
in \eqref{eq:H34last} would require $z^2<n+1$, but then
$z(n+1-z^2)\le n(n+1)<n^4-4$. The cases $y=0$ or $x\ge2,y<0$
are immediate.
For each real $x\le-2$ or $x\ge2$ there is exactly one positive
root $y$ of $f=0$. To see this, write
$y(y^2+x-1)=x^4-4>0$: on the interval where the left side is
positive its derivative is $3y^2+x-1>0$. These roots form two
connected branches. On either branch the graph sum for $G_3$
does not change when an auxiliary polynomial changes sign. At
such a zero no neighbour is zero, by the nonzero edge resultants,
and the square identity makes the product of its neighbours and
its constant positive. Thus an even number of the corresponding
negative-sign terms change. Several nonadjacent zeros cause the
same cancellation separately.
At $x=10$ and $x=-10$ the positive root satisfies $21<y<22$.
The negative vertices of $G_3$ are respectively
$\{2,8\}$ and $\{1,2,3,4,6,8\}$. There are respectively one and
three edges between negative vertices, and respectively one and
three negative vertex constants at those vertices. Both sums
are even, so $B^{(3)}_\infty=0$ on both branches. For $G_4$ its
constant is $2>0$, hence $B^{(4)}_\infty=0$ identically.
The cubic coordinates are zero at infinity. Consequently take
$C_\infty=\{\mathbf0\}$, completing the joint local-bound check.

\emph{Assumption~\ref{ass:cubic-exclusion}: incompatibility.}
Add the sets in Table~\ref{tab:h34-local} componentwise, modulo
$3$ in the first three coordinates and modulo $2$ in the last two.
There are $5\cdot5\cdot2\cdot1\cdot5=250$ choices of summands,
giving $105$ distinct sums. Those whose first three coordinates
are zero are exactly
\[
 (000;01),\qquad (000;10),\qquad (000;11).
\]
In particular, none is $\mathbf0$. Thus
\[
 \mathbf0\notin C_\infty+C_2+C_3+C_5+C_7+C_{19}
             =C_\infty+\sum_{p\in S}C_p.
\]
All eight assumptions, in their joint form, have been checked.
The joint criterion excludes an integer solution of
\eqref{eq:H34last}.
\end{proof}

\section{Extensions, discussion, AI use, and open questions}\label{sec:open}

Theorem \ref{prop:table-one} finishes the analysis of Problem \ref{prob:main} for all two-variable equations of size $h\leq 34$. In fact, we have solved Problem \ref{prob:main} for all two-variable equations of size $h\leq 35$. If we identify equations that differ by permuting or renaming the variables, changing the signs of variables, or multiplying the equation by a nonzero constant, then there are exactly $243{,}145$ two-variable equations of size $h\leq 35$. For $179{,}860$ of them we found an integer solution, and for the remaining $63{,}285$ equations, we proved that no integer solutions exist. Most of the equations are solvable by known methods described in \cite[Chapter 3]{grechuk2024polynomial}. However, $71$ equations resisted these methods: $30$ equations of size $h\leq 34$ listed in Table~\ref{tab:2varH34} and $41$ equations of size $h=35$. We resolved all these equations by the methods described in this paper: $63$ by the elementary method of Section \ref{sec:method}, and $8$ equations ($5$ listed in Table~\ref{tab:remaining} and $3$ equations of size $h=35$) by reciprocity graphs: $7$ by the cubic reciprocity method described in Section~\ref{sec:graphs}, and one equation of size $h=35$ by its analogue over $\Q(i)$, based on quartic reciprocity. The repository \cite{certificates} contains detailed proofs together with complete Lean formalizations for all $71$ equations. It also contains a separate document, thousands of pages long, with recorded integer solutions and non-existence proofs for all $243{,}145$ equations. 

Our analysis shows that every two-variable equation of size $h\leq 35$ that has an integer solution in fact has one with $\min\{|x|,|y|\}\leq 767$. Taking the contrapositive, this implies the following result.

\begin{theorem}\label{th:h35}
	Let $f(x,y)$ be a polynomial with integer coefficients of size $h\leq 35$. If the equation $f(x,y)=0$ has no integer solution with $\min\{|x|,|y|\}\leq 767$, then it has no integer solutions.
\end{theorem}

The constant $767$ in Theorem \ref{th:h35} is the best possible, as certified by the equation $2x^3+xy+2y^2+7=0$, whose only integer solution is $(x,y)=(-767, -21051)$.

Except for a small number of easy-to-solve quintic equations, all equations of size $h\leq 35$ have degree at most $4$. To test our methods for higher-degree equations, we also applied them to equations ordered by length \eqref{eq:ldef}, up to length $l\leq 12$, and resolved all such equations with two exceptions:
\begin{equation}\label{eq:l12open}
4y^3+xy=2x^4-1 \quad \text{and} \quad 4y^3+y=x^5+x+1.
\end{equation}
In this case, the database consists of $751{,}836$ two-variable equations; $527{,}512$ of them have integer solutions, for $224{,}322$ equations we provided insolubility proofs, and $2$ equations remain open. As expected, most of these equations are either solvable by known methods of \cite[Chapter 3]{grechuk2024polynomial}, or have size $h\leq 35$ and are covered by Theorem \ref{th:h35}. However, there are $140$ equations of length $l\leq 12$ outside of these categories. The repository \cite{certificates} contains the insolubility proofs for $138$ of them, accompanied by Lean verifications. It also contains a document, thousands of pages long, recording the results for all $751{,}836$ equations, including partial progress for the two open cases. The analogue of Theorem \ref{th:h35} is the following result.

\begin{theorem}\label{th:l12}
	Let $f(x,y)$ be a polynomial with integer coefficients of length $l\leq 12$, and suppose that $f(x,y)=0$ is not equivalent, under the transformations specified above, to either equation in \eqref{eq:l12open}. If the equation $f(x,y)=0$ either (i) has no integer solution with $\min\{|x|,|y|\}\leq 82{,}596{,}761$, or (ii) has degree at least three and has no integer solution with $\min\{|x|,|y|\}\leq 25{,}121$, then it has no integer solutions. 
\end{theorem}

Both constants in Theorem \ref{th:l12} are the best possible, as certified by (i) the negative Pell equation $x^2-181y^2=-1$, whose smallest integer solutions are $(x,y)=(\pm 1111225770, \pm 82596761)$, and (ii) the equation $9y^2 = 6x^3-x-1$, whose only integer solutions are $(x,y)= (25121, \pm 3250946)$. Theorem \ref{th:l12} provides a finite test for whether any given equation of length $l\leq 12$, apart from those equivalent to \eqref{eq:l12open}, has an integer solution; the test for equations of degree at least three is especially efficient.

In total, we applied our methods to $211$ equations that either have size $h\leq 35$ or length $l\leq 12$, and resisted the known methods described in \cite[Chapter 3]{grechuk2024polynomial}. We were able to solve $209$ of them ($>99\%$). In our opinion, it is even more impressive that for $174$ equations ($\approx 82.5\%$) only quadratic reciprocity was sufficient. This is consistent with statistics in Table~\ref{tab:2varH34}, for which the elementary method of Section \ref{sec:method} resolved $25$ out of $30$ equations ($\approx 83\%$). This gives some preliminary evidence that the success rate of this method does not significantly decrease if we increase the degrees and coefficients of the equations.  

Our methods are based on reciprocity arguments and are closely related to integral Brauer--Manin obstructions \cite{Manin1971}. In fact, Stoll's conjecture \cite[Conjecture 9.1]{stoll2007finite} implies that if the equation $f(x,y)=0$ describes a smooth, geometrically integral affine curve of genus at least $2$ and has no integer solutions, then this can \emph{always} be certified by an integral Brauer--Manin obstruction\footnote{Applied to the smooth projective model, the conjecture identifies its Brauer--Manin set in the modified adelic space with its finite set of rational points; imposing the integral local conditions gives the stated integral obstruction.}. Together with factorization, normalization, a separate check of singular points, and the effective methods for genera $0$ and $1$, this would give a method for solving Problem \ref{prob:main} for all two-variable equations. However, this method is impractical in full generality, because finding a sufficient finite collection of Brauer--Manin conditions can be difficult. Our methods can be interpreted as a practical version of this general approach, which works well in many examples.   

This paper is a nice example of cooperation between humans and AI. We tried to solve all equations from Table~\ref{tab:2varH34} with all available AI models, including ChatGPT 6 Astra Pro, initially with no success. After repeated attempts, the model finally solved \emph{one} equation from the table. The original solution was long and hard to read, and wasted a lot of time on trivial arguments, while the crucial identity \eqref{eq:identity} was hidden inside the proof, came out of nowhere, and was verified by direct computation. The authors then cleaned up the proof, emphasized the central importance of the identity \eqref{eq:identity}, and wrote up the method for finding such identities for other equations, in the clean form of Section \ref{sec:method}. After this method was uploaded to ChatGPT 6 Astra Pro, it was able to solve the remaining equations from Table~\ref{tab:2varH34}, some of them instantly, while others required more rounds and further hints from the authors. Sections \ref{sec:intro}, \ref{sec:method}, and \ref{sec:open} were written entirely by the authors; we only asked AI to make minor language corrections to the final text. The proofs in Sections \ref{sec:smallest}, \ref{sec:extensions}, and \ref{sec:graphs} were produced by AI, rewritten by the authors, and then AI did the language polishing. The crucial identities from Table \ref{tab:compact-norms} were found by ChatGPT 6 Astra Pro, as were the long certificates in Section \ref{sec:graphs}. It is infeasible for a human to verify the full correctness of these long certificates, so we present only proof sketches and rely on Lean formalizations \cite{certificates} to ensure correctness. The Lean certificates were produced by ChatGPT 6 Astra Pro, but some required extensive guidance over multiple rounds. We also verified the correctness of the Lean statements. Claude Opus 5.5 was used to check the identities and certificates in this paper independently, to write the self-contained proofs of the individual equations in the repository \cite{certificates}, and to proofread the final text.  

\begin{table}
	\begin{center}
		\begin{tabular}{ |c|c|c|c| } 
			\hline
			$h$ & Equation & $h$ & Equation \\ 
			\hline\hline
			$34$ & $6+x^3-y^2+y^2 z^2=0$ & $34$ &  $2+x^3+x y+y^2+y^2 z^2=0$ \\ 
			\hline
			$34$ & $6+x^3+x y^2 z+z^2=0$ &  &  \\ 
			\hline	
		\end{tabular}
		\caption{\label{tab:H34unsolved} Equations of size $h\leq 34$ for which Problem \ref{prob:main} is open.}
	\end{center} 
\end{table} 

We end this section with some open questions. After the equations from Table~\ref{tab:2varH34} were solved, the new smallest equations for which Problem \ref{prob:main} remains open are the three-variable equations of size $h=34$ listed in Table \ref{tab:H34unsolved}. If the equations are instead ordered by length \eqref{eq:ldef}, then the smallest open ones are the equations
$$
	y^2+x^3y+z^4+1=0 \quad \text{and} \quad x^3y^2=z^4+2
$$
of length $l=10$. These equations have three variables. If we restrict ourselves to two-variable equations only, then the shortest open ones are the equations \eqref{eq:l12open}, whose length is $l=12$. See \cite{grechuk2024systematic} for the continuously updated list of the smallest and shortest open Diophantine equations in various categories.

\bibliography{paper2var_main_refs}

@article{MR4610191,
 author={Majumdar, D. and Sury, B.},
 title={Fruit {Diophantine} equation},
 journal={The Mathematical Gazette}, volume={107}, number={569},
 pages={302--306}, year={2023}
}

@article{vaishya2022class,
 author={Vaishya, L. and Sharma, R.}, title={A class of fruit {Diophantine} equations},
 journal={Monatshefte f{\"u}r Mathematik}, volume={199}, number={4},
 pages={899--907}, year={2022}
}

@article{prakash2024generalized,
 author={Prakash, O. and Chakraborty, K.},
 title={Generalized fruit {Diophantine} equation and hyperelliptic curves},
 journal={Monatshefte f{\"u}r Mathematik}, volume={203}, number={3},
 pages={667--676}, year={2024}
}

@article{sahoo2024generalized,
 author={Sahoo, S. and Laishram, S.},
 title={Generalized fruit {Diophantine} equation over number fields},
 journal={Integers}, volume={26}, pages={A21}, year={2026},
 note={\url{https://doi.org/10.5281/zenodo.18305078}}
}

@book{grechuk2024polynomial,
 author={Grechuk, Bogdan},
 title={Polynomial {Diophantine} Equations: A Systematic Approach},
 publisher={Springer}, address={Cham}, year={2024},
 note={\url{https://doi.org/10.1007/978-3-031-62949-5}}
}

@article{siegel1929uber,
 author={Siegel, Carl Ludwig},
 title={{{\"U}ber einige Anwendungen diophantischer Approximationen}},
 journal={Abhandlungen der Preu{\ss}ischen Akademie der Wissenschaften,
          Physikalisch-Mathematische Klasse},
 volume={1929}, number={1}, pages={1--70}, year={1929}
}

@article{baker1970integer,
 author={Baker, Alan and Coates, John}, title={Integer points on curves of genus 1},
 journal={Proceedings of the Cambridge Philosophical Society},
 volume={67}, pages={595--602}, year={1970},
 note={\url{https://doi.org/10.1017/S0305004100045904}}
}

@article{faltings1983endlichkeitssatze,
 author={Faltings, Gerd},
 title={{Endlichkeitss{\"a}tze f{\"u}r abelsche Variet{\"a}ten {\"u}ber Zahlk{\"o}rpern}},
 journal={Inventiones Mathematicae}, volume={73}, pages={349--366}, year={1983},
 note={\url{https://doi.org/10.1007/BF01388432}}
}

@misc{stoll2026determining,
 author={Stoll, Michael},
 title={Determining the rational points on a curve of genus 2 and {Mordell--Weil} rank 1},
 year={2025}, note={arXiv:2509.24604, \url{https://arxiv.org/abs/2509.24604}}
}

@article{grechuk2026cubic,
 author={Grechuk, Bogdan},
 title={Cubic {Diophantine} equations: integer solutions beyond direct search},
 journal={The American Mathematical Monthly}, year={2026}, note={To appear}
}

@article{grechuk2026shortest,
 author={Grechuk, Bogdan and Ratcliffe, Ashleigh},
 title={On the shortest open cubic equations},
 journal={International Journal of Number Theory},
 volume={22}, number={9}, pages={2117--2127}, year={2026},
 note={\url{https://doi.org/10.1142/S1793042126501058}}
}

@misc{grechuk2026polynomial,
 author={Grechuk, Bogdan and Agbanwa, Jamal},
 title={On the polynomial values represented by {binary} quadratic forms},
 year={2026}, note={arXiv:2607.06627, \url{https://arxiv.org/abs/2607.06627}}
}

@book{ireland1990classical,
 author={Ireland, Kenneth and Rosen, Michael},
 title={A Classical Introduction to Modern Number Theory},
 edition={Second}, series={Graduate Texts in Mathematics}, volume={84},
 publisher={Springer}, address={New York}, year={1990}
}

@book{serre1973course,
 author={Serre, Jean-Pierre}, title={A Course in Arithmetic},
 series={Graduate Texts in Mathematics}, volume={7},
 publisher={Springer}, address={New York}, year={1973}
}

@inproceedings{moura2021lean,
 author={de Moura, Leonardo and Ullrich, Sebastian},
 title={The {Lean} 4 Theorem Prover and Programming Language},
 booktitle={Automated Deduction -- CADE 28},
 series={Lecture Notes in Computer Science}, volume={12699},
 pages={625--635}, year={2021}, publisher={Springer},
 note={\url{https://doi.org/10.1007/978-3-030-79876-5_37}}
}

@inproceedings{mathlib2020,
 author={{The mathlib Community}}, title={The {Lean} Mathematical Library},
 booktitle={Proceedings of the 9th ACM SIGPLAN International Conference on
 Certified Programs and Proofs}, pages={367--381}, year={2020},
 note={\url{https://doi.org/10.1145/3372885.3373824}}
}

@techreport{bosma2005cubic,
 author={Bosma, Wieb},
 title={Cubic reciprocity and explicit primality tests for {$h\cdot3^k\pm1$}},
 institution={Department of Mathematics, University of Nijmegen},
 number={05.15}, year={2005},
 note={\url{https://www.math.ru.nl/~bosma/pubs/cubicreport.pdf}}
}

@book{cox2015ideals,
 author={Cox, David A. and Little, John and O'Shea, Donal},
 title={Ideals, Varieties, and Algorithms}, edition={Fourth},
 publisher={Springer}, year={2015}
}

@misc{certificates,
 author = {Bogdan Grechuk and Eugene Go},
 title = {{Diophantine} certificates},
 year = {2026},
 howpublished = {GitHub repository, \url{https://github.com/II-eugene-II/diophantine-certificates}},
 url = {https://github.com/II-eugene-II/diophantine-certificates},
 note = {Complete proofs and {Lean}~4 formalizations. Accessed: 12 October 2026}
}

@inproceedings{Manin1971,
 author    = {Manin, Yuri I.},
 title     = {Le groupe de {Brauer--Grothendieck} en
  g{\'e}om{\'e}trie diophantienne},
 booktitle = {Actes du Congr{\`e}s International
  des Math{\'e}maticiens (Nice, 1970)},
 volume    = {1},
 pages     = {401--411},
 year      = {1971},
 publisher = {Gauthier-Villars},
 address   = {Paris}
}

@article{stoll2007finite,
 title={Finite descent obstructions and rational points on curves},
 author={Stoll, Michael},
 journal={Algebra \& Number Theory},
 volume={1},
 number={4},
 pages={349--391},
 year={2007},
 publisher={Mathematical Sciences Publishers}
}

@misc{grechuk2024systematic,
 title={A systematic approach to {Diophantine} equations: open problems},
 author={Grechuk, Bogdan},
 year={2024},
 note={arXiv:2404.08518, \url{https://arxiv.org/abs/2404.08518}}
}
\bibliographystyle{plain}

\end{document}